\documentclass[reqno]{amsart}
\usepackage{amsmath, amssymb, amsthm, fancyhdr, verbatim, graphicx}
\usepackage{enumerate}
\usepackage[all]{xy}
\usepackage[usenames,dvipsnames]{xcolor}
\usepackage{mathrsfs}
\usepackage{tikz-cd}
\usepackage[T1]{fontenc} 
\usepackage{framed, hyperref}
\usepackage[OT2, T1]{fontenc}
\usepackage[titletoc]{appendix}

\numberwithin{equation}{section}

\DeclareSymbolFont{cyrletters}{OT2}{wncyr}{m}{n}
\DeclareMathSymbol{\Sha}{\mathalpha}{cyrletters}{"58}

\usepackage{color}
\newcommand{\tony}[1]{{\color{blue} \sf
    $\spadesuit\spadesuit\spadesuit$ TONY: [#1]}}

\newcommand{\F}{\mathbf{F}}

\newcommand{\CC}{\mathbf{C}}

\newcommand{\wt}[1]{\widetilde{#1}}

\newcommand{\Q}{\mathbf{Q}}
\newcommand{\Z}{\mathbf{Z}}

\newcommand{\mf}[1]{\mathfrak{#1}}

\newcommand{\Gal}{\operatorname{Gal}}

\newcommand{\ul}[1]{\underline{#1}}
\newcommand{\ol}[1]{\overline{#1}}
\newcommand{\wh}[1]{\widehat{#1}}

\newcommand{\mbb}[1]{\mathbb{#1}}

\newcommand{\Cal}[1]{\mathcal{#1}}
\newcommand{\A}{\mathbf{A}}

\newcommand{\co}{\colon}
\newcommand{\mrm}[1]{\mathrm{#1}}

\newcommand{\cX}{X^{\circ}}
\newcommand{\bs}{\backslash}

\DeclareMathOperator{\GL}{GL}
\DeclareMathOperator{\SL}{SL}
\DeclareMathOperator{\Frob}{Frob}

\DeclareMathOperator{\N}{\mathbb{N}}
\DeclareMathOperator{\Tr}{Tr}

\DeclareMathOperator{\Aut}{Aut}

\DeclareMathOperator{\Nm}{Nm}
\DeclareMathOperator{\Spec}{Spec\,}

\DeclareMathOperator{\End}{End}

\DeclareMathOperator{\Res}{Res}

\DeclareMathOperator{\Stab}{Stab}
\DeclareMathOperator{\Bun}{Bun}

\DeclareMathOperator{\Id}{Id}

\DeclareMathOperator{\Gr}{Gr}

\DeclareMathOperator{\Sym}{Sym}

\DeclareMathOperator{\inv}{inv}

\DeclareMathOperator{\Fix}{Fix}

\DeclareMathOperator{\Sht}{Sht}

\DeclareMathOperator{\Hecke}{Hecke}
\DeclareMathOperator{\Sat}{Sat}

\newtheorem{thm}{Theorem}[section]
\newtheorem{lemma}[thm]{Lemma}
\newtheorem{prop}[thm]{Proposition}
\newtheorem{cor}[thm]{Corollary}

\theoremstyle{remark}
\newtheorem{remark}[thm]{Remark} 
\newtheorem{defn}[thm]{Definition}
\newtheorem{const}[thm]{Construction}
\newtheorem{example}[thm]{Example}

\newtheorem{exdef}[thm]{Example/Definition}

\makeatletter
\def\th@remark{%
  \thm@headfont{\bfseries}%
  \normalfont 
  \thm@preskip \thm@preskip 
  \thm@postskip\thm@preskip
}
\def\imod#1{\allowbreak\mkern5mu({\operator@font mod}\,\,#1)}
\makeatother

\numberwithin{equation}{section}

\title[Nearby cycles of parahoric shtukas, and a fundamental lemma]{Nearby cycles of parahoric shtukas,\\ and a fundamental lemma for base change}

\author{Tony Feng}

\begin{document}

\begin{abstract}
Using the Langlands-Kottwitz paradigm, we compute the trace of Frobenius composed with Hecke operators on the cohomology of nearby cycles, at places of parahoric reduction, of perverse sheaves on certain moduli stacks of shtukas. Following an argument of Ng\^{o}, we then use this to give a geometric proof of a base change fundamental lemma for parahoric Hecke algebras for $\GL_n$ over local function fields. This generalizes a theorem of Ng\^{o}, who proved the base change fundamental lemma for spherical Hecke algebras for $\GL_n$ over local function fields, and extends to positive characteristic (for $\GL_n$) a fundamental lemma originally introduced and proved by Haines for $p$-adic local fields. 
\end{abstract}

\maketitle

\tableofcontents



\section{Introduction}\label{intro}
\subsection{Motivation}
There are two main goals of this paper: 
\begin{enumerate}
\item To compute the trace of Frobenius composed with Hecke operators on the cohomology of nearby cycles at places of \emph{parahoric} reduction for certain moduli stacks of shtukas, and 
\item To parlay the resulting formulas into a \emph{geometric} proof of a fundamental lemma for base change for central elements in parahoric Hecke algebras over local function fields. 
\end{enumerate} 
The first goal is accomplished by using the Grothendieck-Lefschetz trace formula to break up the computation of the trace into two pieces: (1) counting points on certain moduli spaces, and (2) understanding the stalks of the nearby cycles sheaves. These pieces are then each resolved by a sequence of technical steps whose overall strategy is rather well-known, and which would require a considerable amount of notation to describe. Therefore, in this introduction we will focus on informally explaining the idea of the second goal.

The fundamental lemma of interest was proposed and proved by Haines \cite{Haines09} for $p$-adic (i.e. mixed characteristic) local fields, and generalizes the fundamental lemma for base change for spherical Hecke algebras proved (independently) in the $p$-adic case by Clozel \cite{Clo90} and Labesse \cite{Lab90}, building on work of Kottwitz \cite{Kott86b}, and in the function field case (for $\GL_n$) by Ng\^{o} \cite{Ngo06}. 

The original motivation for this fundamental lemma was to study the cohomology of a Shimura variety with parahoric level structure, and in particular to determine the semisimple zeta factor at a place of parahoric reduction. The fundamental lemma enters in comparing the trace of Frobenius and Hecke operators on this cohomology with the geometric side of the Arthur-Selberg trace formula. We refer the interested reader to \cite{Haines09}, especially p. 573, for more details. 

The same applications are available in the function field setting, with Shimura varieties replaced by the \emph{moduli stacks of shtukas}, which have been utilized by Drinfeld (\cite{Drin87}, for $\GL_2$), L. Lafforgue (\cite{Laff02}, for $\GL_n$), and V. Lafforgue (\cite{Laff18}, for general reductive groups) to spectacular success towards the global Langlands correspondence over function fields.

 However, in this paper we have chosen to emphasize the \emph{geometric} aspect of the fundamental lemma, rather than its applications to the Langlands program. In contrast to the proof of \cite{Haines09} for the $p$-adic case, which following in the tradition of \cite{Clo90} and \cite{Lab90} is via $p$-adic harmonic analysis, our proof works by exploiting additional geometry and structure which is available in the function field setting. Our strategy is very much based on that of \cite{Ngo06}, and indeed specializes to it in the case of spherical Hecke algebras. 
 
 Broadly speaking, the base change fundamental lemma compares an orbital integral with a twisted orbital integral. To elaborate, let $F$ be a local field, $G$ a reductive group over $F$, $\gamma \in G(F)$, and $f$ a function on $G(F)$. The \emph{orbital integral} corresponding to this data is 
 \begin{equation}\label{orbital integral}
\mrm{O}_{\gamma}(f)  := \int_{G(F)/G_{\gamma}(F)} f(g^{-1} \gamma g) \, dg
 \end{equation}
 where $G_{\gamma}(F)$ is the centralizer of $\gamma$ in $G(F)$. We will take $f$ to be in an appropriate Hecke algebra $\Cal{H}_G(F)$. (Of course we also need to discuss the normalization of Haar measures, but let us leave that for later, in \S \ref{FL statement}.) 
 
 Let $E/F$ be an unramified extension of degree $r$, $\delta \in G(E)$, and $f_E$ a function on $G(E)$. The \emph{twisted orbital integral} corresponding to this data is 
 \begin{equation}\label{twisted orbital integral}
 \mrm{TO}_{\delta \sigma}(f_E)  := \int_{G(E)/G_{\delta \sigma}(F)} f_E(g^{-1} \delta \sigma (g)) \, dg
 \end{equation}
 where $\sigma \in \Gal(E/F)$ is the lift of (arithmetic) Frobenius, and 
 \[
 G_{\delta \sigma}(F)  = \{ g \in G(E) \co g^{-1} \delta \sigma(g) = \delta\}
 \]
 is the twisted centralizer of $\gamma$ in $G(E)$. Again, we will take $f_E$ to be in an appropriate Hecke algebra $\Cal{H}_G(E)$.
 
 If $\Cal{H}_{G(E),J}$ and $\Cal{H}_{G(F),J}$ are corresponding \emph{parahoric} Hecke algebras, then there is a base change homomorphism for their centers
 \[
 b  \co Z(\Cal{H}_{G(E),J}) \rightarrow Z(\Cal{H}_{G(F),J}).
 \]
 There is also a norm map $N$ from stable twisted conjugacy classes in $G(E)$ to stable conjugacy classes in $G(F)$. 
 
In the special case $G =\GL_n$, the base change fundamental lemma for the center of parahoric Hecke algebras predicts that for $\sigma$-regular, $\sigma$-semisimple $\delta \in G(E)$ and $f_E \in Z(\Cal{H}_{G(E),J})$, we have 
 \begin{equation}\label{intro BCFL}
  \mrm{TO}_{\delta \sigma}(f_E) = \mrm{O}_{N(\delta)} (b(f_E)).
 \end{equation}
This is almost what we will prove. (For more general $G$, the formulation is more complicated; see \cite{Haines09}, Theorem 1.0.3 and \S 5.) 

\subsection{The idea of the proof}
 
 Now we can describe our strategy of proof of \eqref{intro BCFL}. The starting point is the seminal work of Kottwitz on counting points of Shimura varieties over finite fields. In \cite{Kott92} Kottwitz proves a formula expressing the trace of Frobenius composed with a Hecke operator on the cohomology of certain PEL Shimura varieties as a sum of a product of (twisted) orbital integrals: 
\begin{equation}\label{intro Kottwitz formula}
\Tr(h \circ \Frob_p, H^*(\mrm{Sh}_K, \ol{\Q}_{\ell})) = \sum (\ldots) \mrm{O}_{\gamma} (h^p) \mrm{TO}_{\delta \sigma}(h_p)
\end{equation}
where $\mrm{Sh}_K$ is an appropriate Shimura variety and $h$ is a Hecke operator. In fact the purpose of the fundamental lemma is to re-express the twisted orbital integrals in \eqref{intro Kottwitz formula}, so as to be able to compare the expression with the geometric side of the Arthur-Selberg trace formula. But in this paper we adopt an opposite perspective, instead viewing \eqref{intro Kottwitz formula} as  \emph{giving a geometric interpretation of (twisted) orbital integrals} (in the $p$-adic case) in terms of the cohomology of Shimura varieties. 

In the function field setting, which is the one of interest to this paper, one can prove an analogous formula of the form
\begin{equation}\label{intro Ngo formula A}
\Tr(h_A \circ \Frob_{x_0}\circ \tau, H^*(\Sht_A, \Cal{A})) = \sum (\ldots) \mrm{O}_{\gamma} (h_A^{x_0}) \mrm{TO}_{\delta \sigma}(h_{A,x_0})
\end{equation}
for an appropriate moduli stack $\Sht_A$, an appropriate sheaf $\Cal{A}$, an appropriate Hecke operator $h_A$, and an additional symmetry $\tau$. (Roughly, $\tau$ is a ``rotation'' symmetry that arises from the moduli problem.) 

However, it turns out that we can \emph{also} construct a moduli stack $\Sht_B$ such that 
\begin{equation}\label{intro Ngo formula B}
\Tr(h_B \circ \Frob_{x_0} \circ \tau, H^*(\Sht_B, \Cal{B})) = \sum (\ldots) \mrm{O}_{\gamma} (h_B^{x_0}) \mrm{O}_{\Nm(\sigma)}(b(h_{B,x_0}))
\end{equation}
for an appropriate sheaf $\Cal{B}$, an appropriate Hecke operator $h_B$, and an additional symmetry $\tau$ similar to that from \eqref{intro Ngo formula A}. The crucial point is that in \eqref{intro Ngo formula B} the twisted orbital integral is replaced with the orbital integral of a base changed function. 

We remark that the computations \eqref{intro Ngo formula A} and \eqref{intro Ngo formula B} were obtained in \cite{Ngo06} for places of good (hyperspecial) reduction, in which case one finds a spherical Hecke operator. In the present work, which concerns places of parahoric bad reduction, the analogous computations \eqref{intro Ngo formula A} and \eqref{intro Ngo formula B} are of independent interest, and actually form the main content of this paper. They require several nontrivial inputs, including, for the parahoric setting that we study here, a version of the Kottwitz Conjecture for shtukas, as well as a geometric interpretation of the base change homomorphism for Hecke algebras. Nevertheless, let us elide these points for now.

The upshot is that \eqref{intro Ngo formula A} and \eqref{intro Ngo formula B} translate the problem of comparing orbital integrals and twisted orbital integrals into comparing (the cohomology of) two different moduli spaces $\Sht_A$ and $\Sht_B$. (We remark that the relationship we seek turns out to be subtler than equality, but again we elide this issue for now.) To do this, we realize $\Sht_A$ and $\Sht_B$ as specializations of ``bigger'' moduli spaces $\wt{\Sht}_A$ and $\wt{\Sht}_B$. We then apply the Langlands-Kottwitz method to the spaces $\wt{\Sht}_A$ and $\wt{\Sht}_B$, obtaining formulas analogous to \eqref{intro Ngo formula A} and \eqref{intro Ngo formula B}, but the crucial point is that over ``many'' points of these larger moduli spaces (necessarily away from $\Sht_A$ and $\Sht_B$, the original spaces of interest), the output of the Langlands-Kottwitz method has \emph{no} twisted orbital integrals, hence does not require any fundamental lemma to compare. We then deduce the desired equality over the specializations to $\Sht_A$ and $\Sht_B$ by a continuation principle. The key to making this strategy work is a strategic design of the moduli spaces $\Sht_A$ and $\Sht_B$, which we take from \cite{Ngo06}.

\subsection{Statement of the base change fundamental lemma}\label{FL statement}

We now give a precise formulation of the fundamental lemma of interest. It is an exact analogue for local function fields of the fundamental lemma studied in \cite{Haines09}. We will impose several assumptions that simplify the formulation, referring the general case to \cite{Haines09}. In particular we let $G$ be a reductive group over a local field $F$, and assume that $G$ is unramified and $G_{\mrm{der}}$ is simply connected.

\subsubsection{Normalization of Haar measures}\label{sssec: Haar measures}
 Recall the notation of \S \ref{intro}. To give a well-defined meaning to the orbital integral \eqref{orbital integral} and twisted orbital integral \eqref{twisted orbital integral}, we need to specify Haar measures on  $G, G_{\gamma}$ and $G_{\delta \sigma}$. We assume that $\gamma$ is regular semisimple.

We fix a hyperspecial vertex and an alcove containing it in the Bruhat-Tits building for $G$ over $F$. By Bruhat-Tits theory this induces maximal compact subgroups $K_F \subset  G(F)$ and $K_E \subset G(E)$. 

\begin{itemize}
\item We pick the left-invariant Haar measures $dg$ on $G(F)$ and $G(E)$ such that $dg(K_F)  =1$ and $dg(K_E)  =1$. 
\item We pick the left-invariant Haar measures $dh$ on $G_{\gamma}(F)$ and $G_{\delta \sigma}(F)$ such that $dg(K_F \cap G_{\gamma}(F) ) = 1$ and $dh$ on $G_{\delta \sigma}(E) $ is the canonical transfer of Haar measure from $G_{\gamma}$ to its inner form $G_{\delta \sigma}$. \tony{put a reference to Kottwitz}
\end{itemize}

Taking the quotient measure $\frac{dg}{dh}$ on $G(F)/G_{\gamma}(F)$ and $G(E)/G_{\delta \sigma}(E)$, now \eqref{orbital integral} and \eqref{twisted orbital integral} have been fully defined. \\

\subsubsection{Parahoric Hecke algebras} We now fix a facet in the given alcove whose closure contains the fixed hyperspecial point, which induces corresponding (compact open) parahoric groups $J_F \subset G(F)$ and $J_E \subset G(E)$. Let $\Cal{H}_{G(F), J} = \mrm{Fun}_c(J_F \backslash G(F) / J_F, \ol{\Q}_{\ell})$ and $\Cal{H}_{G(E), J} = \mrm{Fun}_c(J_E \backslash G(E) / J_E,  \ol{\Q}_{\ell})$ be the corresponding parahoric Hecke algebras. (Parahoric Hecke algebras are discussed in more detail in \S \ref{parahoric HA}.)

\subsubsection{The base change homomorphism} Let $Z(\Cal{H}_{G(F),J})$ be the center of $\Cal{H}_{G(F),J}$, and define $Z(\Cal{H}_{G(E),J})$ similarly. There is a base change homomorphism 
\[
b \co Z(\Cal{H}_{G(E),J}) \rightarrow Z(\Cal{H}_{G(F),J}),
\]
which is defined in \S \ref{defn base change}. To give a brief characterization of the base change homomorphism: under the Bernstein isomorphism 
\[
- *_J \mbb{I}_K\co Z(\Cal{H}_{G(F),J}) \xrightarrow{\sim} \Cal{H}_{G(F),K}
\]
obtained by convolving with the indicator function $\mbb{I}_K$, it corresponds to the usual base change homomorphism for spherical Hecke algebras.  
\[
\begin{tikzcd}
Z(\Cal{H}_{G(E), J}) \ar[r, "b"] \ar[d, "- *_J \mbb{I}_K", "\sim"']  & Z(\Cal{H}_{G(F), J}) \ar[d, "- *_J \mbb{I}_K", "\sim"'] \\
\Cal{H}_{G(E), K} \ar[r, "b"] & \Cal{H}_{G(F), K}
\end{tikzcd}
\]

\subsubsection{The norm map} Let $\sigma \in \Gal(E/F)$ be a lift of (arithmetic) Frobenius. The ``concrete norm''
\[
\Nm_{E/F} \co G(E) \rightarrow G(F)
\]
defined by 
\[
\Nm_{E/F} (\delta) := \delta \cdot \sigma(\delta) \cdot \ldots \sigma^{r-1} (\delta)
\]
descends to a norm map 
\[
N \co G(E)/\text{stable $\sigma$-conjugacy} \rightarrow G(F)/\text{stable conjugacy}.
\]

\subsubsection{Formulation of the fundamental lemma}

The following fundamental lemma was proved by Haines in the $p$-adic setting \cite[Theorem 1.0.3]{Haines09}.

\begin{thm}[Haines]\label{FL formulation}
Let $E/F$ be an unramified extension of $p$-adic local fields of degree $r$ and residue characteristic $p$. Let $\psi \in Z(\Cal{H}_{G(E),J})$ and $\delta \in G(E)$ such that $N(\delta)$ is semisimple. Then we have
\[
  \mrm{SO}_{\delta \sigma}^{G(E)}(\psi) = \mrm{SO}_{N(\delta)}^G (b(\psi)).
\]
\end{thm}

Here $\mrm{SO}$ are \emph{stable (twisted) orbital integrals}, for whose definition we refer to \cite[\S 5.1]{Haines09}. Since our eventual result will be for $G = \GL_n$, where stable conjugacy coincides with conjugacy, we can ignore the issue of stabilization.

\begin{remark}
Haines has informed us that his proof, which is based on the global simple trace formula and Kottwitz's stabilization of the twisted trace formula, does not carry over (at least, not without nontrivial additional work) to the positive characteristic setting. \footnote{However, we note that W. Ray Dulany proved the base change fundamental lemma for $\GL_2$ by hand in the function field case \cite{RD10}. We thank Tom Haines for informing us about Ray Dulany's work.}	
\end{remark}

\subsection{Statement of results}
 We now formulate our main result, which is an extension (in a special case) of Theorem \ref{FL formulation} to positive characteristic. 
	
By the Bernstein isomorphism, a basis for $ Z(\Cal{H}_{G(E),J})$ is given by the functions $\psi_{\mu}$ for $\mu$ a dominant coweight of $G$, which correspond under $- * _J \mbb{I}_K$ to the indicator functions of the double coset in $K_E \backslash G(E)/K_E$ indexed by $\mu$. 

\begin{exdef} If $G = \GL_n$, and  $T \subset \GL_n$ is the usual (diagonal) maximal torus, then we may identify $ X_*(T) \cong \Z^n$ in the standard way. The dominant weights coweights $X_*(T)_+$ are those $\mu = (\mu_1,\ldots, \mu_n) $ with $\mu_1 \geq \mu_2 \geq \ldots  \geq \mu_n$. We define
\[
|\mu| := \mu_1  + \ldots + \mu_n.
\]
\end{exdef}

In this paper we prove: 	

\begin{thm}\label{main}
Let $E/F$ be an unramified degree $r$ extension of characteristic $p$ local fields. If $\delta \in \GL_n(E)$ is such that $N(\delta)$ is regular semisimple and separable, and $\psi \in Z(\Cal{H}_{\GL_n(E),J})$ is a linear combination of $\psi_{\mu}$ with $|\mu| = 0$, then we have
\[
  \mrm{TO}_{\delta \sigma}(\psi) = \mrm{O}_{N(\delta)} (b(\psi)).
\]
\end{thm}


\begin{remark}
Let us make some remarks on the hypotheses. The hypothesis $|\mu| = 0$ arises geometrically as a condition for the non-emptiness of moduli stacks of shtukas. It can be interpreted as saying that $\psi$ comes from the Hecke algebra of $\SL_n$. 

The restriction to $\GL_n$ comes from a need to obtain a proper moduli stack, in order to have enough control over the cohomology of the relevant moduli stacks of shtukas. In general the moduli stacks of shtukas are of infinite type, and their cohomology not constructible. However, for $\GL_n$ we can use the trick of globalizing to a division algebra in order to create a proper global space with the right local behavior. 

\end{remark}




\begin{remark}
As T. Haines pointed out to us, another key aspect of the fundamental lemma is the assertion that $\mrm{O}_{\gamma}(b(\psi))= 0$ if $\gamma$ is not a norm. Our strategy does not seem to naturally give access to this statement. On the other hand, since Labesse gave a \emph{purely local} proof of this statement for the spherical case in mixed characteristic, which was extended to the center of parahoric Hecke algebras in \cite{Haines09} \S 5.2, it should generalize to positive characteristic. 
\end{remark}

\subsubsection{Related work} The fundamental lemma for base change for \emph{spherical} Hecke algebras, which arises from Theorem \ref{FL formulation} in the special case where $J= K$ is a hyperspecial maximal compact subgroup, was proved in the $p$-adic case by Clozel \cite{Clo90} and Labesse \cite{Lab90}, using key input from Kottwitz \cite{Kott86b} who checked it for the unit element. These arguments were generalized by Haines to proved the base change fundamental lemma for centers of parahoric Hecke algebras, as has been discussed.\footnote{See the last paragraph of the introduction to \cite{Haines09} for a discussion of to what extent Theorem \ref{FL formulation} follows from the fundamental lemma for twisted endoscopy. Although base change is a special case of twisted endoscopy, the fundamental lemma for twisted endoscopy should imply that there is a matching function for $\psi$, but does not identify it in terms of the base change homomorphism.}


For local function fields (i.e. positive characteristic), the spherical case $J=K$ of Theorem \ref{main} was established by Ng\^{o} \cite{Ngo06}, also for $\GL_n$ and also for $|\mu|=0$ (with the same reasons for the restrictions). Indeed, our  strategy as described in \S \ref{intro} is the one pioneered by \cite{Ngo06}. Similar results were obtained independently and simultaneously by Lau \cite{Lau04}.

\subsection{Summary of the paper}

Although our current argument does not work beyond $\GL_n$, we hope that after future technical improvements in the theory of shtukas it can be generalized to a much wider class of reductive groups. For this reason, for the individual steps we have tried to work with more general groups when possible. It seems worthwhile to give a brief outline of the organization of the paper, pointing out exactly where we can be more general.

In \S \ref{shtukas background} we review the theory of shtukas for nonconstant reductive group schemes, summarizing the essential background facts. Also, a key point is to define an ``integral model'' for the moduli stack of shtukas, extending over points of parahoric bad reduction. 

 In \S \ref{kottwitz for shtukas} we establish an analogue of the Kottwitz Conjecture for moduli of shtukas. This works even for fairly general (not necessarily constant) reductive groups $\Cal{G} \rightarrow X$: Gaitsgory originally proved it for constant groups, and his argument was generalized by Zhu in \cite[Theorem 7.3]{Zhu14} and Pappas-Zhu in \cite{PZ13}. 
 
 In \S \ref{counting fixed points} we establish some counting formulas for points of shtukas over finite fields. This is a minor variant of the work of Ng\^{o} B.C. and Ng\^{o} Dac T., which was previously only formulated at places of hyperspecial level structure. Our contribution is to write it out for the case of parahoric bad reduction that we require. 
 
 
 In \S \ref{hecke base change} we provide a geometric interpretation of the base change homomorphism for spherical Hecke algebras and the center of parahoric Hecke algebras for general split reductive groups $G$ over a local field. For $\GL_n$ this was proved by Ng\^{o}, in a formulation that was rather specific to $\GL_n$. We generalize the argument to spherical Hecke algebras for arbitrary split reductive groups, and then use that to deduce a result for (central elements in) parahoric Hecke algebras. 
 
 In \S \ref{moduli  problems} we introduce the two moduli problems $\Sht_A$ and $\Sht_B$ which are to be compared, and recall Ng\^{o}'s theorem stating the precise comparison. Using this we deduce an equality of traces on the nearby cycles sheaves at the point of parahoric reduction. Here we also crucially use that the moduli of shtukas associated to a sufficiently ramified division algebra is proper, which implies that the cohomology is a local system. 
 
 In \S \ref{computation of trace} we compute these traces in terms of (twisted) orbital integrals, giving formulas in the paradigm of Kottwitz, and then use them in \S \ref{FL proof} to deduce the cases of the base change fundamental lemma claimed in Theorem \ref{main}.

\subsection{Acknowledgments}

I am indebted to Zhiwei Yun for teaching me basically everything that I know about shtukas, and in particular for directing me to \cite{Ngo06}. I thank Zhiwei, Laurent Clozel, Gurbir Dhillon, Tom Haines, Jochen Heinloth, Urs Hartl, Bao Le Hung, and Rong Zhou for helpful conversations related to this work, and Brian Conrad and Timo Richarz for comments and suggestions on a draft. I am particularly grateful to Tom Haines and the very thorough referee for crucial corrections and explanations. 

This project was conceived at the 2017 Arbeitsgemeinschaft in Oberwolfach, and completed while I was a guest at the Institute for Advanced Study, and under the support of an NSF Graduate Fellowship. I am grateful to these institutions for their support. 

\section{Notation}\label{sec: notation}

We collect some notation that will be used frequently throughout the paper.

\begin{itemize}
\item Let $X$ be a smooth projective curve over a finite field $k=\F_q$, and -- changing notation from the introduction -- let $F=k(X)$ be its global function field for the rest of the paper. We assume that $X$ has a rational point, and fix such a point $x_0\in X(\F_{q})$.
\item We will let $\cX$ be an open subset of $X$, usually the complement of some points for ramification and possibly also $x_0$.
\item We denote by $|X|$ the set of closed points of $X$, and for $x\in X$ we write $k(x)$ for the residue field of $x$. 
\item For $x \in X$, we let $\Cal{O}_x$  be the completion of $\Cal{O}_{X,x}$ at its maximal, and $F_x$ be the fraction field of $\Cal{O}_x$. We set $D_x := \Spec \Cal{O}_x$.
\item We let $G$ be a connected reductive group over $F$, whose derived group is simply connected. We extend $G$ to a parahoric group scheme $\Cal{G} \rightarrow X$ (which is also possible -- cf. \cite{Laff18} \S 12.1), and that $\Cal{G} \otimes F_x$ is split at all points $x \in X$ where $\Cal{G}(\Cal{O}_x)$ is not hyperspecial. 
\item We let $U \subset X$ be the dense open subscheme where $\Cal{G}$ is reductive. 
\item We denote by $\Cal{E}^0$ the trivial (fppf) $\Cal{G}$-torsor over $X$. 	
\end{itemize}

\section{Moduli of shtukas}\label{shtukas background}

In this section we recall material concerning shtukas and their perverse sheaves. This is mostly background, but we emphasize that it is important for us to work at the generality of nonconstant groups. This allows us to define an ``integral model'' for parahoric shtukas, which is \emph{much} easier than the corresponding problem for Shimura varieties. References for this section are \cite[ \S 3]{Zhu14}, \cite[\S 12]{Laff18}, and \cite{HR16}.

\subsection{$\Cal{G}$-bundles}\label{subsec: G-bundles}

Let $\Cal{G} \rightarrow X$ be a smooth affine group scheme with (connected) reductive generic fiber $G$, such that $\Cal{G}|_{\Cal{O}_{x}}$ is a parahoric group scheme for each $x \in X$. We assume that $\Cal{G}|_{ F_x}$ is split at all points $x \in X$ where $\Cal{G}(\Cal{O}_x)$ is not hyperspecial.

\subsubsection{} We recall the notion of $\Cal{G}$-bundles and affine Grassmannians, the study of which seems to have been initiated by \cite{PR10}.

\begin{defn}
A \emph{$\Cal{G}$-bundle $\Cal{E}$} is a $\Cal{G}$-torsor for the fppf topology. We define  $\Bun_{\Cal{G}}$ to be the (Artin) stack\footnote{For the fact that this is really an Artin stack in the generality required here, see \cite{Bro13}. We thank Brian Conrad for bringing this reference to our attention.} representing the functor 
\[
\Bun_{\Cal{G}} \colon S \mapsto \left\{ \text{$\Cal{G}$-bundles on $X \times S$}\right\}   .
\]
\end{defn}

\begin{defn} We define the \emph{global affine Grassmannian} $\Gr_{\Cal{G}}$ to be the ind-scheme representing the functor
\[
\Gr_{\Cal{G}} \colon S \mapsto \left\{ (x,\Cal{E} , \beta)  \colon  \begin{array}{@{}c@{}} 
x \in X(S) \\
  \Cal{E} \in \Bun_{\Cal{G}}(S)  \\
  \beta \co \Cal{E}|_{X_S - \Gamma_x} \cong \Cal{E}^0 |_{X_S-\Gamma_x} \\ 
 \end{array} \right\}   
\]
where here and throughout $\Cal{E}^0$ denotes the  trivial $\Cal{G}$-torsor, and $\Gamma_x$ is the graph of $x$, viewed as a divisor in $S \times X$. 
\end{defn}

We have a map 
\[
\pi \co \Gr_{\Cal{G}} \rightarrow X
\]
sending $(x,\Cal{E} , \beta)  \mapsto x$.

\begin{example}
For any closed point $x$, the fiber $\Gr_G|_x$ is the partial affine flag variety associated with $G|_{D_x}$, as was studied for example in \cite{PR08}. If $G|_{D_x}$ happens to be \emph{reductive}, then $G|_{D_ x} \cong (G |_x) \otimes_{k(x)} \Cal{O}_x$, where $G|_x := G \times_X {x}$ is a constant group scheme, hence $\Gr_G|_x$ is the usual affine Grassmannian attached to $G|_x$ over $k(x)$. If on the other hand $G|_{D_x}$ happens to be an \emph{Iwahori} group scheme, then $\Gr_G |_x$ is an affine flag variety.
\end{example}

\subsubsection{Adding level structure}\label{sssec: level structure}

  Let $\Bun_{\Cal{G}, n \Gamma	} $ be the moduli stack of $\Cal{G}$-bundles with ``$n$-th order level structure'', i.e. 
\[
\Bun_{\Cal{G}, n \Gamma } \colon S \mapsto \left\{ (x,\Cal{E}, \psi )  \colon  \begin{array}{@{}c@{}} 
  x \in X(S) \\ 
  \Cal{E} \in \Bun_{\Cal{G}}(S)  \\
\psi \colon \Cal{E}|_{ n \Gamma_x} \xrightarrow{\sim} \Cal{E}^0|_{ n \Gamma_x}
 \end{array} \right\}   
 \]
  where $n\Gamma_x$ is viewed as a divisor in $S \times X$. 

  Let $\Bun_{\Cal{G}, \infty \Gamma	} $ be the moduli stack of $\Cal{G}$-bundles with ``infinite level structure'', i.e. 
\[
\Bun_{\Cal{G}, \infty \Gamma } \colon S \mapsto \left\{ (x,\Cal{E}, \psi )  \colon  \begin{array}{@{}c@{}} 
  x \in X(S) \\ 
  \Cal{E} \in \Bun_{\Cal{G}}(S)  \\
\psi \colon \Cal{E}|_{ \wh{\Gamma}_x} \xrightarrow{\sim} \Cal{E}^0|_{ \wh{\Gamma}_x}
 \end{array} \right\}   
 \]
 where $\wh{\Gamma}_x$ is the completion of $X \times S$ along $\Gamma_x$. One can also think of $\psi$ as a compatible family of level structures over $ n \Gamma_x$ as $n \rightarrow \infty$. We will use the notation $\wh{\Gamma}_x^{\circ}  ``:= \wh{\Gamma}_x - \Gamma_x"$ with the meaning as in \cite[Notation 1.7]{Laff18}.  
 
 Let $\Cal{L}^+ \Cal{G}$ be the global ``arc group'', defined by 
 \[
 \Cal{L}^+\Cal{G} \co S \mapsto  \left\{ (x, \beta )  \colon  \begin{array}{@{}c@{}}  
x \in X(S)  \\\beta \in \Cal{G}(\wh{\Gamma}_x)
 \end{array} \right\}    .
 \]
 
 We clarify that $\Cal{L}^+ \Cal{G}$ is a pro-algebraic group scheme over $X$, as it is the restriction of scalars of smooth affine group schemes over $X$, and $\Bun_{\Cal{G}, \infty \Gamma	} $ is an Artin stack of infinite type, as it is and a $\Cal{L}^+ \Cal{G}$-torsor over $\Bun_{\Cal{G}}$.

 \begin{remark}\label{rem: Gr left arc action}
 There is an action of $\Cal{L}^+ \Cal{G}$ on $\Gr_{\Cal{G}}$ by changing the level structure $\psi$. 
 \end{remark}

\subsubsection{Global Schubert varieties}\label{sssec: schubert Gr} Let $T \subset G$ be a maximal torus. In \cite[\S 2]{Rich16} (generalizing work in the tamely ramified case of \cite[\S 3.3]{Zhu14}) it is shown how to  associate to $\mu \in X_*(T_{\ol{F}}) $ a global Schubert variety $\Gr_{\Cal{G}}^{\leq \mu}$. Of course, this is well-known in the split case.

\subsubsection{Geometric Satake} We fix some notation pertaining to the Geometric Satake correspondence \cite{MV07}. For a space $Y$ over $X$, we denote by $Y|_U$ the fibered product of $Y$ with $U \hookrightarrow X$ (recall that \S \ref{sec: notation} that $U \subset X$ is the locus where $\Cal{G}$ is reductive, so $G$ is the generic fiber of $\Cal{G}$.)

Since we want to work over $\F_q$, we need a slightly modified version of the Geometric Satake equivalence -- \cite{RZ15} for an explanation of the theory over general fields. Fix $\ell$ and a choice of $\sqrt{q} \in \ol{\Q}_{\ell}$, so we get a half Tate twist as in \cite[Definition A.1]{RZ15}, so that \cite[Theorem A.12]{RZ15} applies.

\begin{defn}
Given a finite-dimensional representation $W$ of Langlands' L-group ${}^L G$, we let $\mrm{Sat}_{\Gr_{\Cal{G}}}(W)$ be the associated $\Cal{L}^+ \Cal{G}$-equivariant (for the action of Remark \ref{rem: Gr left arc action}) perverse sheaf on $\Gr_{\Cal{G}}|_U$ furnished by Geometric Satake, in the sense of \cite[Proposition 5.5.16]{Zhu15}. Note that $\mrm{Sat}_{\Gr_{\Cal{G}}}(W)$ is automatically $\Cal{L}^+ \Cal{G}|_U$-equivariant. (Strictly speaking, \cite{RZ15} concerns the local affine Grassmannian. For a statement of Geometric Satake for non-constant groups phrased in terms of the Beilinson-Drinfeld Grassmannian, see \cite[Theorem 12.16]{Laff18}.) 

If $G$ is split then irreducible finite-dimensional representations $W$ of ${}^L  G   = \wh{G}$ are indexed by dominant coweights $\mu \in  X_*(T)_+$ for a maximal split torus $T \subset G$, and we denote by $\Sat_{\Gr_{\Cal{G}}}(\mu) := \Sat_{\Gr_{\Cal{G}}}(W_{\mu})$ the corresponding perverse sheaf.
\end{defn}

This is the primal source for constructing perverse sheaves on a plethora of objects, which will all be denoted $\mrm{Sat}_{\ldots}(W)$ or $\mrm{Sat}_{\ldots}(\mu)$. 

\subsection{Hecke stacks} 

\subsubsection{} We now define objects that geometrize the Hecke operators. 

\begin{defn}
We define the \emph{Hecke stack} $\Hecke_{\Cal{G}}$ by the functor of points 
\[
\Hecke_{\Cal{G}} \co S \mapsto \left\{ (x, \Cal{E}, \Cal{E}',  \varphi) \colon  \begin{array}{@{}c@{}} 
x \in X(S) \\
  \Cal{E}, \Cal{E}' \in \Bun_{\Cal{G}}(S)  \\
\varphi \colon \Cal{E}'|_{X \times S - \Gamma_{x}} \xrightarrow{\sim} \Cal{E}|_{X \times S - \Gamma_{x}}
 \end{array} \right\}  .
\]

\end{defn}

We have structure maps 
\[
\xymatrix{
& \Hecke_{\Cal{G}} \ar[dl]_{h^{\leftarrow}} \ar[dr]^{h^{\rightarrow}} \ar[d]^{\pi} \\
\Bun_{\Cal{G}} & X & \Bun_{\Cal{G}}
}
\]
where the map $h^{\leftarrow}$ takes $(\Cal{E}, \Cal{E}') \mapsto \Cal{E}$, and the map $h^{\rightarrow}$ takes $(\Cal{E}, \Cal{E}') \mapsto \Cal{E}'$. 

One can think of the $\Hecke_{\Cal{G}}$ as looking locally, in the smooth topology, like $\Bun_{\Cal{G}} \times_X \Gr_{\Cal{G}}$. To make this precise, recall that there is an action of $L^+ \Cal{G} $ on $\Bun_{\Cal{G},  \infty \Gamma}$, by changing the level structure. 

 \begin{prop}\label{Hecke to Gr}  There is an isomorphism 
 \[
\xi \colon  \Hecke_{\Cal{G}} \xrightarrow{\sim} (\Gr_{\Cal{G}} \times_{X} \Bun_{\Cal{G}, \infty \Gamma})/\Cal{L}^+ \Cal{G}
 \]
 where the quotient is for the diagonal action of $\Cal{L}^+ \Cal{G}$.
 \end{prop}
 
This is actually taken as the \emph{definition} of the Hecke stack in \cite[\S 12.3.1]{Laff18}. Although it is well-known we have not found the statement formulated in quite this way, so we give a proof for completeness.

\begin{proof}
Giving an isomorphism $\Hecke_{\Cal{G}} \xrightarrow{\sim} (\Gr_{\Cal{G}} \times_X \Bun_{\Cal{G},  \infty \Gamma})/\Cal{L}^+ \Cal{G}$ is equivalent to giving an  $\Cal{L}^+ \Cal{G}$-equivariant isomorphism from an $\Cal{L}^+ \Cal{G}$-torsor over $\Hecke_{\Cal{G}} $ to $\Gr_{\Cal{G}} \times_{X} \Bun_{\Cal{G},  \infty \Gamma}$, so we will construct the latter.

Let $\wt{\Hecke}_{\Cal{G}}  \rightarrow \Hecke_{\Cal{G}} $ be the $\Cal{L}^+ \Cal{G}$-torsor representing $(x, \varphi \co \Cal{E}' \dashrightarrow \Cal{E}) \in \Hecke_{\Cal{G}}$ \emph{plus} a choice of trivialization $\psi \co \Cal{E}|_{\wh{\Gamma}_x} \cong \Cal{E}^0|_{\wh{\Gamma}_x}$.

There is a map 
\[
\wt{\Hecke}_{\Cal{G}}  \rightarrow \Gr_{\Cal{G}} \times_{X^I} \Bun_{\Cal{G},  \infty \Gamma}
\]
 sending
 \[
 (x, \varphi \co \Cal{E}' \dashrightarrow \Cal{E}, \psi) \mapsto (x,\Cal{E}' , \psi \circ \varphi), (x, \Cal{E}, \psi)
 \]
 where we have implicitly used the Beauville-Laszlo theorem \cite[Lemma 3.1]{Zhu14} to extend  $\psi \circ \varphi$, which is a priori only defined on $\wh{\Gamma}_x^{\circ}$, to $X-\Gamma_x$. It is easily checked that this is an isomorphism, by defining an inverse directly, and that it is $\Cal{L}^+ \Cal{G}$-equivariant.

\end{proof}

\begin{remark}In practice, we can always translate these statements into ones about locally finite type Artin stacks by considering substacks obtained by bounding the type of the modification, and noting that the action of $\Cal{L}^+ \Cal{G}$ on such a substack factors through a finite type quotient.
\end{remark}

\subsubsection{Geometric Satake for Hecke stacks}

\begin{defn} Denoting by $D^b_c(-)$ the bounded derived category, we define a functor
\[
\mrm{Sat}_{\Hecke_{\Cal{G}}} \colon \mrm{Rep}_{{}^L G}	 \rightarrow D^b_c( \Hecke_{\Cal{G}}|_U )
\]
as follows (by definition, $D^b_c( \Hecke_{\Cal{G}} )$ is the $\Cal{L}^+G$-equivariant constructible derived category of $\Gr_{\Cal{G}}$ with coefficients in $\Q_{\ell}$). If $W \in \mrm{Rep}_{{}^L G}$, then 
\[
\mrm{Sat}_{\Gr_{\Cal{G}}}(W) \boxtimes \ol{\Q}_{\ell, \Bun_{\Cal{G}, \infty \Gamma}} \in D^b_c(\Gr_{\Cal{G}} \times_{X}  \Bun_{\Cal{G}, \infty \Gamma} |_U)
\]
descends to the quotient by $\Cal{L}^+ \Cal{G}$ by the fact that $\mrm{Sat}_{\Gr_{\Cal{G}}}(W)$ is $\Cal{L}^+ \Cal{G}$-equivariant. We set $\mrm{Sat}_{\Hecke_{\Cal{G}}}(W)$ to be  the pullback of this descent via the isomorphism $\xi^*$ from Proposition  \ref{Hecke to Gr}.  
\end{defn}

\subsubsection{Hecke stacks with bounded modification} For $\mu \in X_*(T_{\ol{F}})$ we define $\Hecke_{\Cal{G}}^{\leq \mu}$ as follows. First, we have the Schubert variety $\Gr_{\Cal{G}}^{\leq \mu} \rightarrow \Gr_{\Cal{G}}$, which has an $\Cal{L}^+ \Cal{G}$-action. This induces a substack of $(\Gr_{\Cal{G}} \times_{X} \Bun_{\Cal{G}, \infty \Gamma})/\Cal{L}^+ \Cal{G}$, and we define $\Hecke_{\Cal{G}}^{\leq \mu}$ to be the pullback via $\xi^*$ of Proposition \ref{Hecke to Gr}.

If $\Cal{G}  = G \times X$ is constant and split over $X$, then up to taking reduced substacks, $\Hecke_{\Cal{G}}^{\leq \mu}$ admits a very concrete definition as ``modifications of $G$-bundles with invariant bounded by $\mu$''. In \S \ref{subsec: D-shtukas} we explicate this for $\GL_n$-bundles, which may be an enlightening example.

\subsection{Shtukas}

\subsubsection{} We now define the moduli stack of $\Cal{G}$-shtukas. At places $x \in X$ where $\Cal{G}|_{D_x}$ is a parahoric group scheme, this should be thought of as an ``integral model'' of the usual moduli stacks in which the legs are demanded to be disjoint from the level structure.

\begin{defn}\label{def: shtukas}
We define the moduli stack of $\Cal{G}$-shtukas by the following cartesian diagram 
\[
\begin{tikzcd}
\Sht_{\Cal{G}} \ar[r]   \ar[d] & \Bun_{\Cal{G}}  \ar[d, "\Id \times \Frob"] \\
\Hecke_{\Cal{G}} \ar[r, "h^{\leftarrow} \times h^{\rightarrow}"] & \Bun_{\Cal{G}}  \times \Bun_{\Cal{G}}
\end{tikzcd}
\]
More explicitly, $\Sht_{\Cal{G}}$ represents the following moduli problem:
\[
\Sht_{\Cal{G}} \co S \mapsto \left\{ (x, \Cal{E} , \varphi) \colon  \begin{array}{@{}c@{}} 
x \in X(S) \\
  \Cal{E}\in \Bun_{\Cal{G}}(S)  \\
\varphi \colon {}^{\sigma} \Cal{E}|_{X \times S - \Gamma_{x}} \xrightarrow{\sim} \Cal{E}|_{X \times S - \Gamma_{x}}
 \end{array} \right\}  
\]
where $\sigma$ is the Frobenius on the $S$ factor in $X \times  S$, and  ${}^{\sigma} \Cal{E}$ is the pullback of $\Cal{E}$ under the map $1 \times \sigma \co X \times S \rightarrow X \times S$. 
	\end{defn}
	
	We have an evident map 
	\[
	\pi \co \Sht_{\Cal{G}} \rightarrow X \quad \text{sending} \quad (x, \Cal{E}, \Cal{E}',  \varphi)  \mapsto x.
	\]

\subsubsection{Perverse sheaves on shtukas}  From Definition \ref{def: shtukas} we have a tautological map 
\[
\iota \co \Sht_{\Cal{G}}  \rightarrow \Hecke_{\Cal{G}} .
\] 

\begin{defn}
For $W \in \mrm{Rep}({}^L G)$, we define $\Sat_{\Sht_{\Cal{G}}}(W) := \iota^*(\Sat_{\Hecke_{\Cal{G}}}(W))$. This is a perverse sheaf up to shift on $\Sht_{\Cal{G}}|_U$, since the affine Grassmannian is \'{e}tale-locally equivalent to $\Sht_{\Cal{G}}$ (see \S \ref{kottwitz for shtukas}, or \cite[\S 1.1]{Laff18}).
\end{defn}

	\subsubsection{Schubert varieties of shtukas} 
	
For $\mu \in X_*(\wh{T}_{\ol{F}})$, we define $\Sht_{\Cal{G}}^{\leq \mu} = \iota^*  \Hecke_{\Cal{G}}^{\leq\mu}$. This is a closed substack  of $\Sht_{\Cal{G}}$ which is the support of $\Sat_{\Sht_{\Cal{G}}}(\mu)$. We call these ``Schubert varieties of shtukas'' even though they are, of course, not varieties but (Deligne-Mumford) stacks.

\subsubsection{Hecke operators on shtukas}\label{sssec: hecke corr for shtuka}
 There are Hecke correspondences of shtukas that induce Hecke operators on $\Sht_{\Cal{G}}$, hence on their cohomology.

\begin{defn}
We define $\Hecke(\Sht_{\Cal{G}})$ to be the moduli stack parametrizing $x,y \in X(S)$ along with a diagram
\[
\begin{tikzcd}
{}^{\sigma} \Cal{E}  \ar[r, dashed, "x"] \ar[d, dashed, "\sigma(y)", "\sigma(\beta)"']  & \Cal{E} \ar[d, dashed, "y", "\beta"'] \\
{}^{\sigma} \Cal{E}'  \ar[r, dashed, "x"] &   \Cal{E}' 
\end{tikzcd}
\]
Here we note: 
\begin{itemize}
\item $\Cal{E}$ and $\Cal{E}'$ are $\Cal{G}$-torsors on $X \times S$, and ${}^{\sigma} \Cal{E} $ and ${}^{\sigma} \Cal{E}'$ are their twists by $1 \times \sigma$. 
\item The $x$ above the horizontal arrows mean an isomorphism on the complement of $\Gamma_x$. 
\item The $y$ (resp. $\sigma(y)$) next to the vertical arrows means an isomorphism on the complement of $\Gamma_y$ (resp. $\Gamma_{\sigma(y)}$). 
\item The map $\sigma(\beta)$ is the twist of  $\beta$. We emphasize that it is determined by $\beta$, rather than being an additional datum.
\end{itemize}
\end{defn}

We evidently have a diagram
\[
\begin{tikzcd}
\Hecke(\Sht_{\Cal{G}}) \ar[rr] \ar[dr, "\pi_2"'] & &  \Hecke_{\Cal{G}} \ar[dl, "\pi"] \\
& X 
\end{tikzcd}
\]
where the horizontal arrow sends this data to $(y, \Cal{E}, \Cal{E}', \beta)$, which allows us to define $\Hecke(\Sht_{\Cal{G}})^{\leq\mu}$ for $\mu \in X_*(T_{\ol{F}})$, and $\Sat_{\Hecke(\Sht_{\Cal{G}})}(W)$ for  $W \in \mrm{Rep}({}^L G)$. 

We also evidently have a diagram
\[
\begin{tikzcd}
&  \Hecke(\Sht_{\Cal{G}}) \ar[dl, "h^{\leftarrow}"'] \ar[dr, "h^{\rightarrow}"] \\
\Sht_{\Cal{G}} & & \Sht_{\Cal{G}}
\end{tikzcd}
\]
where the arrows $h^{\leftarrow}$ and $h^{\rightarrow}$ send this data to $(x, {}^{\sigma}\Cal{E} \dashrightarrow \Cal{E})$ and $(x,  {}^{\sigma}\Cal{E}' \dashrightarrow \Cal{E}')$ respectively. For $v \in X$, let 
\[
\Hecke(\Sht_{\Cal{G}})^{\leq \mu}_v := \pi_2^{-1}(v).
\]
A choice of $v \in X$ and $\mu \in X_*(T_{\ol{F}})$ induces a correspondence 
 \begin{equation}\label{eqn: hecke corr on sht 0}
\begin{tikzcd}
&  \Hecke(\Sht_{\Cal{G}})^{\leq \mu}_v \ar[dl, "h^{\leftarrow}"'] \ar[dr, "h^{\rightarrow}"] \\
\Sht_{\Cal{G}} \ar[dr, "\pi"'] &  &  \Sht_{\Cal{G}} \ar[dl,"\pi"] \\
 & X & 
\end{tikzcd}
\end{equation}
 which is the analogue of the classical Hecke correspondences.
 
\begin{defn}\label{def: hecke action on coh}
Since $\pi \circ h^{\leftarrow}  = \pi \circ h^{\rightarrow}$ and $h^{\leftarrow*}  (\mrm{Sat}_{\Sht_{\Cal{G}}}(W)) \cong h^{\rightarrow*} ( \mrm{Sat}_{\Sht_{\Cal{G}}}(W) )	$, from $ \Hecke(\Sht_{\Cal{G}})^{\leq \mu}_v $ we get a corresponding Hecke operator on $R \pi_!\Sat_{\Sht_{\Cal{G}}}(W) \in D^+(X)$.
\end{defn}

\subsection{Iterated shtukas and factorization}

This entire discussion carries through to ``iterated'' versions of $\Gr_{\Cal{G}}$, $\Hecke_{\Cal{G}}$ and $\Sht_{\Cal{G}}$. We will content ourselves with stating the essentials, leaving the reader to generalize the preceding discussion. (A reference is \cite[\S 1,2]{Laff18}.) 

\subsubsection{Iterated affine Grassmannian}\label{sssec: iterated aff gr}
The iterated global affine Grassmannian  
\[
\pi \co \Gr_{\Cal{G}} \wt{\times} \Gr_{\Cal{G}} \rightarrow X^2
\]
is defined by the functor of points
\[
\Gr_{\Cal{G}} \wt{\times} \Gr_{\Cal{G}}\co S \mapsto  \left\{ (x_1, x_2, \Cal{E}_1, \Cal{E}_2, \varphi,  \beta)  \colon  \begin{array}{@{}c@{}} 
x_1,x_2 \in X(S) \\
  \Cal{E}_1, \Cal{E}_2 \in \Bun_{\Cal{G}}(S)  \\
  \varphi \co \Cal{E}_1 |_{X_S - \Gamma_{x_1}}  \xrightarrow{\sim} \Cal{E}_2|_{X_S-\Gamma_{x_1}} \\ 
  \beta \co \Cal{E}_2|_{X_S - \Gamma_{x_2}} \xrightarrow{\sim} \Cal{E}^0 |_{X_S-\Gamma_{x_2}} \\ 
 \end{array} \right\}   
\]

We may denote $\Gr_{\Cal{G},X^r} = \Gr_{\Cal{G}} \wt{\times} \ldots \wt{\times}\Gr_{\Cal{G}}$ ($r$ times), although the  reader should be warned that this notation is sometimes used elsewhere in the literature to denote a different object. We also have Schubert cells: given $\mu_1, \ldots, \mu_r \in X_*(T_{\ol{F}})$, we can define $\Gr_{\Cal{G},X^r}^{\leq (\mu_1,\ldots,\mu_r)}$ in a way that is by now obvious.

\subsubsection{Iterated shtukas}  We now define the iterated shtukas. 

\begin{defn}\label{def: iterated shtukas}
We define the moduli stack $\Sht_{\Cal{G},X^r}$ by the functor of points:
\[
\Sht_{\Cal{G},X^r} \co S \mapsto \left\{   \begin{array}{@{}c@{}} 
x_1,\ldots,x_r \in X(S)  \\
\Cal{E}_0, \Cal{E}_1, \ldots, \Cal{E}_r \xrightarrow{\sim} {}^{\sigma} \Cal{E}_0  \in \Bun_{\Cal{G}}(S)  \\
\varphi_i \colon \Cal{E}_i|_{X \times S - \Gamma_{x_{i+1}}} \xrightarrow{\sim} \Cal{E}_{i+1}|_{X \times S - \Gamma_{x_{i+1}}}
 \end{array} \right\}
\]
	\end{defn}
	
	\begin{remark}
	The stack of iterated shtukas $\Sht_{\Cal{G},X^r}$ can also be defined as a repeated fibered product of  $\Sht_{\Cal{G}}$ over $\Bun_{\Cal{G}}$, which is more analogous to how we defined $\Sht_{\Cal{G}}$.
	\end{remark}
	
	We have an evident map 
	\[
	\pi \co \Sht_{\Cal{G},X^r} \rightarrow X^r
	\]
	projecting to the datum of $(x_1, \ldots, x_r)$.

We can similarly  define $\Hecke_{\Cal{G},X^r}$ and $\Hecke(\Sht_{\Cal{G},X^r}) $.	For a tuple $W_1, \ldots, W_r \in \mrm{Rep}({}^L G)$ we can define a shifted perverse sheaf $\Sat_{\Sht_{\Cal{G},X^r} }(W_1, \ldots, W_r)$ using Geometric Satake. 

We also have Schubert cells for $\Sht_{\Cal{G},X^r} $: given $\mu_1, \ldots, \mu_r \in X_*(\wh{T}_{\ol{F}})$, we can define $\Sht_{\Cal{G},X^r}^{\leq (\mu_1,\ldots,\mu_r)}$ in a way that is by now obvious.

\subsection{$\Cal{D}$-shtukas}\label{subsec: D-shtukas}

 One of the main difficulties with $\Sht_{\Cal{G}}$ is that it is of infinite type in general. To study the fundamental lemma for $\GL_n$, we can globalize to a division algebra instead of the constant group $\GL_n$, which gives us a \emph{proper} moduli problem. We now explain the salient facts about this special case. Since the literature already contains several excellent expositions of the theory of $\Cal{D}$-shtukas, we will content ourselves with a brief summary. A reference for everything here is \cite[\S 1]{Ngo06}; see \cite{Laff97} or \cite{Lau04} for more extensive treatments.

Let $D$ be a division algebra $F$ of dimension $n^2$, ramified over a (necessarily finite) set of points $Z \subset  X$. We assume that our fixed (rational) point $x_0\notin Z$, so $D_{x_0} := D \otimes_F F_{x_0} \cong \mf{gl}_n(F_{x_0})$. Later we will need to assume that $\# Z$ is sufficiently large. 

 We extend $\Cal{D}$ to an $\Cal{O}_X$-algebra $\Cal{D}$ such that  $\Cal{D}_x$ is a maximal order in $D_x$ for all $x$, and we let $G \rightarrow X$ be the associated group scheme of units. Let $\Cal{G} \rightarrow X$ be a a parahoric group scheme which is hyperspecial away from $x_0$ but perhaps not hyperspecial at $x_0$; we will be most interested in the case where $\Cal{G}$ is not hyperspecial at $x_0$.

\subsubsection{Modification types}
 Let $T \subset \GL_n$ be the standard maximal torus. The dominant coweights are 
\[
 X_*(T)_+ \cong \Z^n_+ := \{ \mu = (\mu^{(1)}, \ldots,\mu^{(n)}) \co \mu^{(1)} \geq \mu^{(2)} \geq \ldots \mu^{(n)}\}.
\]
The \emph{relative position} of lattices in $k[[t]]^n$ is a $\mu \in  X_*(T)_+$ determined by the Cartan decomposition 
\[
\GL_n(k((t))) = \bigcup_{\mu \in  X_*(T)_+} \GL_n(k[[t]]) t^{\mu}  \GL_n(k[[t]]).
\]

Let $\cX :=  X-Z-\{x_0\}$. For $x \in \cX$, a modification of $\Cal{G}$-bundles outside $x$ is an isomorphism $\varphi \co \Cal{E}'|_{\cX-x} \xrightarrow{\sim} \Cal{E}|_{\cX-x}$. Let $F_x$ be the completion of $F$ at $x$, and $\ol{F}_x = \ol{\F}_q \wh{\otimes} F_x$, i.e. the completion of the maximal unramified extension $F_x$. Then $\Cal{E}'|_{\Spec \Cal{O}_x}$ and $\Cal{E}|_{\Spec \Cal{O}_x}$ induce two lattices in $\ol{F}_x^{\oplus  n}$ by using $\varphi \otimes F_x$ to identify their generic fibers. Since a choice of uniformizer at $x$ induces an isomorphism $\ol{F}_x \cong \ol{F}_q((t))$, the previous discussion applies so that we can speak of the ``relative position'' $\mu \in X_*(T)_+$ of these two lattices (it is easily checked to be well-defined, independently of the choice of uniformizer). We will call this the \emph{modification type} of $\varphi$. 

\subsubsection{The global affine Grassmannian}

Up to taking the reduced substack, we can interpret $\Gr_{\Cal{G}}^{\leq \mu}|_{\cX}$ as the subscheme of $\Gr_{\Cal{G}}|_{\cX}$ parametrizing
\[
 \beta \co \Cal{E}|_{X_S - \Gamma_x} \cong \Cal{E}^0 |_{X_S-\Gamma_x} 
 \]
 such that for each geometric point of $S$, the modification type of $\beta$ is $\leq \mu$. The sheaf $\mrm{Sat}_{\Gr_{\Cal{G}}}(\mu)$ is the IC sheaf of $\Gr_{\Cal{G}}^{\leq \mu}$, i.e. the middle extension of the constant weight-zero sheaf on the open cell. 
 
 \begin{prop}\label{prop: gr sat ula}
 The sheaf $\mrm{Sat}_{\Gr_{\Cal{G}}}(\mu)$ is universally locally acyclic with respect to the morphism $\Gr_{\Cal{G}}|_{\cX} \rightarrow \cX$. 
 \end{prop}
 
\begin{proof}
This is \cite[\S 1.1 Corollaire 6]{Ngo06}. Note that Ng\^{o}'s formulation is slightly different, but is actually deduced from the version that we state, which is the usual formulation in Geometric Satake. 
\end{proof}

\subsubsection{The Hecke stack}

The Hecke stack $\Hecke_{\Cal{G}}^{\leq \mu}|_{\cX}$ parametrizes modifications of $G$-torsors over $\cX$
\[
(x, \varphi \colon \Cal{E}'|_{X \times S - \Gamma_{x}} \xrightarrow{\sim} \Cal{E}|_{X \times S - \Gamma_{x}})
 \]
 with modification type $\leq \mu$ at all geometric points of $S$. 
 
 
\subsubsection{Shtukas and iterated shtukas}

The moduli stack $\Sht_{\Cal{G},(X-Z)^r}$ parametrizes 
\[
 \left\{   \begin{array}{@{}c@{}} 
x_1,\ldots,x_r \in (X-Z)(S)  \\
\Cal{E}_0, \Cal{E}_1, \ldots, \Cal{E}_r \cong {}^{\sigma} \Cal{E}_0  \in \Bun_{\Cal{G}}(S)  \\
\varphi_i \colon \Cal{E}_i|_{X \times S - \Gamma_{x_{i+1}}} \xrightarrow{\sim} \Cal{E}_{i+1}|_{X \times S - \Gamma_{x_{i+1}}}
 \end{array} \right\}
\]
The Schubert ``variety'' $\Sht_{\Cal{G},(X-Z)^r}^{\leq  (\mu_1, \ldots, \mu_r) }$ associated to $ (\mu_1, \ldots, \mu_r) \in (X_*(T)_+)^r$ can be interpreted, up to reduced structure, as the substack where the modification type of $\varphi_i$ is bounded by $\mu_i$ at all geometric points of $S$. For such a tuple we can also form $\Sat_{\Sht_{\Cal{G}},(X-Z)^r}(\mu_1, \ldots, \mu_r)$ on $\Sht_{\Cal{G},(X-Z)^r}$, which is perverse up to shift and supported on $\Sht_{\Cal{G},(X-Z)^r}^{\leq  (\mu_1, \ldots, \mu_r) }$.

\begin{prop}\label{prop: sht sat ula}
The (shifted) perverse sheaf $\Sat_{\Sht_{\Cal{G}},(X-Z)^r}(\mu_1, \ldots, \mu_r)|_{(\cX)^r}$ is locally acylic with respect to the morphism  $\pi \co \Sat_{\Sht_{\Cal{G}},(X-Z)^r}|_{(\cX)^r} \rightarrow (\cX)^r$.
\end{prop} 

\begin{proof}
This \cite[\S 1.4 Corollaire 2]{Ngo06}, but with the same remark as in the proof of Proposition \ref{prop: gr sat ula}.
\end{proof}

\subsubsection{Global geometry}

The stack $\Sht_{\Cal{G}}$ has infinitely many connected components owing to the positive-dimensional center of $\Cal{G}$. We wish to and can rectify this in the usual way: let $a \in \A_F^{\times}$ be a non-trivial idele of degree 1, and let $\Sht_{\Cal{G}}/a^{\Z}$ be the quotient obtained by formally adjoining an isomorphism $\Cal{E} \cong \Cal{E} \otimes \Cal{O}(a)$. Similarly define $\Sht_{\Cal{G},(X-Z)^r}/a^{\Z}$. We still have the map 
\[
\pi \co \Sht_{\Cal{G},(X-Z)^r}/a^{\Z} \rightarrow (X-Z)^r
\]
and the Geometric Satake sheaves descend to $\Sht_{\Cal{G},(X-Z)^r}/a^{\Z}$, which in an effort to curtail increasingly monstrous notation we continue to denote by $\mrm{Sat}_{\Sht_{\Cal{G},(X-Z)
^r}}(\mu_1, \ldots, \mu_r)$. Furthermore, we still have: 

\begin{prop}[{\cite[\S 1.4 Corollaire 2]{Ngo06}}]\label{ULA} The shifted perverse sheaf $\Sat_{\Sht_{\Cal{G}},(X-Z)^r}(\mu_1, \ldots, \mu_r)$ restricted to $\Sht_{\Cal{G},(\cX)^r}/a^{\Z}$ is locally acyclic with respect to the map 
\[
\pi \co \Sht_{\Cal{G},(\cX)^r}/a^{\Z} \rightarrow (\cX)^r.
\] 
\end{prop}

We now prepare to state the properness result for the morphism $\Sht_{\Cal{G},(\cX)^r}/a^{\Z} \rightarrow (\cX)^r$. Any $\mu \in X_*(T)_{+}$ can be uniquely written as 
\[
\mu = \mu^+ + \mu^-
\]
where 
\begin{align*}
\mu^+ &:= (\mu_1^+ \geq \ldots \geq \mu_r^+ \geq 0) \\
\mu^- &:= (0 \geq \mu_1^- \geq \ldots \geq \mu_r^-)
\end{align*}
and for all $1 \leq i \leq r$, we have either $\mu_i = \mu^+_i$ or $\mu_i  = \mu^-_i$. We define 
\[
||\mu|| := \max(|\mu^+|, |\mu^-|).
\]

\begin{prop}[{\cite[\S 1.6 Proposition 2]{Ngo06}}]\label{proper}
Let $ (\mu_1, \ldots, \mu_r) \in (X_*(T)_+)^r$. Suppose that the locus $Z$ of ramification points of $\Cal{D}$ satisfies
\[
\# Z  \geq n^2 (||\mu_1|| + \ldots + || \mu_r||). 
\]
Then the morphism 
\[
\pi^{\circ} \co \Sht_{\Cal{G},(\cX)^r}^{\leq  (\mu_1, \ldots, \mu_r) }/a^{\Z} \rightarrow (\cX)^r
\]
is proper. 
\end{prop}

We need to extend this result to our integral model $X-Z = \cX \cup \{x_0\}$.  

\begin{prop}\label{prop: integral model proper}
Let $ (\mu_1, \ldots, \mu_r) \in (X_*(T)_+)^r$. Suppose that the locus $Z$ of ramification points of $\Cal{D}$ satisfies
\[
\# Z  \geq n^2 (||\mu_1|| + \ldots + || \mu_r||). 
\]
Then the morphism 
\[
\pi \co \Sht_{\Cal{G},(X-Z)^r}^{\leq  (\mu_1, \ldots, \mu_r) }/a^{\Z} \rightarrow (X-Z)^r
\]
is proper. 
\end{prop}

\begin{proof}
Let $G$ be the group scheme isomorphic to $\Cal{G}$ at places away from $x_0$ but hyperspecial at $x_0$. Then Proposition \ref{proper} applied to $\Sht_G$ shows that 
\[
\pi' \co \Sht_{G,(X-Z)^r}^{\leq  (\mu_1, \ldots, \mu_r) }/a^{\Z} \rightarrow (X-Z)^r
\]
is proper. The map $\Cal{G} \rightarrow G$ induces a projection
\[
\mrm{pr} \co  \Sht_{\Cal{G},(X-Z)^r}^{\leq  (\mu_1, \ldots, \mu_r) }/a^{\Z} \rightarrow  \Sht_{G,(X-Z)^r}^{\leq  (\mu_1, \ldots, \mu_r) }/a^{\Z}
\]
which we claim is proper. It obviously suffices to prove the claim. For that we consider the commutative diagram below, omitting some subscripts and superscripts, etc. for clarity of presentation. 
\[
\begin{tikzcd}
\Sht_{\Cal{G}} \ar[r]  \ar[ddr, dashed] \ar[d] & \Bun_{\Cal{G}}  \ar[d, "\Id \times \Frob"] \ar[ddr] \\
\Hecke_{\Cal{G}} \ar[r, "h^{\leftarrow} \times h^{\rightarrow}"] \ar[ddr, dashed] & \Bun_{\Cal{G}}  \times \Bun_{\Cal{G}} \ar[ddr] \\
&  \Sht_{G} \ar[r]   \ar[d] & \Bun_{G}  \ar[d, "\Id \times \Frob"] \\
&  \Hecke_{G} \ar[r, "h^{\leftarrow} \times h^{\rightarrow}"] & \Bun_{G}  \times \Bun_{G}
\end{tikzcd}
\]
In this diagram the squares with solid arrows are cartesian. Let $\Hecke_{\Cal{G}}'$ and $\Sht_{\Cal{G}}'$ denote the fibered products 
\[
\begin{tikzcd}
\Hecke_{\Cal{G}}' \ar[r, "h^{\leftarrow} \times h^{\rightarrow}"] \ar[d]& 
\Bun_{\Cal{G}}  \times \Bun_{\Cal{G}} \ar[d] \\
 \Hecke_{G} \ar[r, "h^{\leftarrow} \times h^{\rightarrow}"] & \Bun_{G}  \times \Bun_{G}
\end{tikzcd}
\]
and 
\[
\begin{tikzcd}
\Sht_{\Cal{G}}' \ar[r]   \ar[d] & \Bun_{\Cal{G}}  \ar[d, "\Id \times \Frob"] \\
\Hecke_{\Cal{G}}' \ar[r, "h^{\leftarrow} \times h^{\rightarrow}"] & \Bun_{\Cal{G}}  \times \Bun_{\Cal{G}}
\end{tikzcd}.
\]
The map $\Bun_{\Cal{G}} \rightarrow \Bun_G$ is proper, hence so is pullback $\Hecke_{\Cal{G}}'  \rightarrow \Hecke_G$. Since the map $\Hecke_{\Cal{G}}^{\leq  (\mu_1, \ldots, \mu_r) } \rightarrow (\Hecke_{\Cal{G}}')^{\leq  (\mu_1, \ldots, \mu_r) }$ is proper, the map $\Sht_{\Cal{G}}^{\leq  (\mu_1, \ldots, \mu_r) } \rightarrow (\Sht_{\Cal{G}}')^{\leq  (\mu_1, \ldots, \mu_r) }$ is also proper. As the map $\rightarrow (\Sht_{\Cal{G}}')^{^{\leq  (\mu_1, \ldots, \mu_r) }} \rightarrow \Sht_G^{^{\leq  (\mu_1, \ldots, \mu_r) }}$ is also proper, so is the composition
\[
\Sht_{\Cal{G}}^{\leq  (\mu_1, \ldots, \mu_r) } \rightarrow (\Sht_{\Cal{G}}')^{^{\leq  (\mu_1, \ldots, \mu_r) }}  \rightarrow \Sht_G^{^{\leq  (\mu_1, \ldots, \mu_r) }}.
\] 
\end{proof}

\section{The Kottwitz Conjecture for shtukas} \label{kottwitz for shtukas}

\subsection{Parahoric Hecke algebras}\label{parahoric HA}

Let $G$ be a split reductive group over a non-archimedean local field $F_t$ with uniformizer $t$. (The splitness assumption is not necessary, but is certainly adequate for our eventual purposes and simplifies the notation significantly.)

\subsubsection{Spherical Hecke algebra}
Let $K$ be a hyperspecial maximal compact subgroup of $G(F_t)$. By Bruhat-Tits theory we may extend $G$ to an integral model over the valuation subring $\Cal{O}_t \subset F_t$ such that $G(\Cal{O}_t) = K$. 

Let $\Cal{H}_{G,K}  = \mrm{Fun}_c(K \bs G(F_t) / K , \ol{\Q}_{\ell})$ be the corresponding spherical Hecke algebra. This has several canonical bases, so we fix notation for them. Let $T \subset G$ be a maximal split torus. As is well known, we have a Cartan decomposition 
\begin{equation}\label{eq: cartan decomp}
G(F_t) = \bigcup_{\mu \in X_*(T)_+} K t^{\mu} K 
\end{equation}
indexed by the dominant coweights $X_*(T)_+ \cong \Z^n$, where $t^{\mu}$ is the character such that for a character $\chi \in X^*(T)$, we have $\chi(t^{\mu})  = t^{\langle \chi, \mu \rangle}$. 

\begin{defn}\label{defn: hecke function double coset}
For $\mu \in X_*(T)_+$, we denote by $f_{\mu} \in \Cal{H}_{G,K}$ the indicator function of $K t^{\mu} K$.
\end{defn}

\subsubsection{Geometrization of spherical Hecke algebra}\label{sssec: geom spherical HA}
A second basis is obtained by interpreting categorifying the Hecke algebra in terms of $L^+ G$-equivariant perverse sheaves on the affine Grassmannian $\Gr_G$. Since we want to work over $\F_q$, we need a slightly modified version of the Geometric Satake equivalence; see \cite{Zhu15} for an explanation of the theory over general fields. We summarize the essential points. Recall that $\Gr_G(k_t) =G(F_t)/G(\Cal{O}_t)$ where $k_t$ is the residue field of $F_t$. Fixing $\ell$ and a choice of $\sqrt{q} \in \ol{\Q}_{\ell}$, so we get a half Tate twist as in \cite[Definition A.1]{RZ15}, Geometric Satake furnishes a fully faithful symmetric monoidal functor 
\[
\mrm{Rep}(\wh{G})  \rightarrow \mrm{Perv}_{G(\Cal{O}_t)} (\Gr_G) .
\]
The simple objects in $\mrm{Rep}(\wh{G})$ are indexed by $\mu \in X_*(T)_+$, and we denote by $\mrm{Sat}_{\Gr_G}(\mu)$ the corresponding perverse sheaf on $\Gr_G$, which is the IC sheaf of the Schubert variety $\Gr_G^{\leq\mu}$, Tate twisted to be pure of weight zero. 

For any $\Cal{F} \in \mrm{Perv}_{G(\Cal{O}_t)} (\Gr_G)$, we have under the function-sheaf dictionary \cite[\S III.12]{KW01} a trace function $f_{\Cal{F}}$ on $G(\Cal{O}_t) \bs G(F_t)/G(\Cal{O}_t)$ given by 
\begin{equation}\label{eq: trace function}
f_{\Cal{F}}(x) = \Tr(\Frob, \Cal{F}_{\ol{x}}).
\end{equation}

\begin{defn}
We define $\psi_{\mu}$ to be the trace function associated to $\Sat_{\Gr_G}(\mu)$. (By \cite{RZ15}, especially Lemma  A.13, this corresponds under the Satake equivalence to the character of the highest weight representation $V_{\mu}$.)
 \end{defn}

\begin{defn}Since $\mrm{Sat}_{\Gr_{\Cal{G}}}(\mu)$ is a $\Cal{G}(\Cal{O}_{x_0})$-equivariant sheaf on $\Gr_{\Cal{G},x_0}$, its stalks are the same on any open Schubert cell $\Gr^{= \nu}$ corresponding to $Kt^{\nu} K$ in the Cartan decomposition \eqref{eq: cartan decomp}. We denote this common stalk by $\Sat_{\Gr_G}(\mu)_{\nu}$. 
\end{defn}

\begin{lemma}\label{lem: spherical hecke formula}
We have 
\[
\psi_{\mu} = \sum_{\nu \leq u} \Tr(\Frob, \Sat_{\Gr_G}(\mu)_{\nu}) f_{\nu}
\]
\end{lemma}

\begin{proof}
This is immediate from the fact that $\Sat_{\Gr}(\mu)_{\nu}$ is supported on $\Gr^{\leq \mu}$, which is the union of the $\Gr^{= \nu}$ for $\nu \leq \mu$, and the definition of $f_{\nu}$ as the characteristic function on $K t^{\nu} K$.  
\end{proof}

 We have a Satake isomorphism 
 \[
 \Cal{H}_G(K) \xrightarrow{\sim} R(\wh{G}) \cong  \ol{\Q}_{\ell}[X_*(T)]^W
 \]
 where $W$ is the Weyl group of $T$. (The Satake isomorphism is reviewed in \S \ref{subsec: satake transform}.) Here $R(\wh{G})$ is the representation ring of $\wh{G}$, which is generated by the classes of the highest weight representations $V_{\mu}$.


\subsubsection{Parahoric Hecke algebras}

Let $J$ be a parahoric subgroup of $G$ stabilizing a facet whose closure contains the vertex corresponding to $K$ in the Bruhat-Tits building of $G(F_t)$. Let $\Cal{H}_{G,J} = \mrm{Fun}_c(J \bs G(F_t) / J , \ol{\Q}_{\ell})$ be the corresponding parahoric Hecke algebra.

\begin{thm}[Bernstein,{ \cite[Theorem 3.1.1]{Haines09}}]\label{thm: bernstein isom} Convolution with $f_0 = \mbb{I}_K$ (the identity of $\Cal{H}_{G,K}$) induces an isomorphism 
\[
- *_J \mbb{I}_K \co Z(\Cal{H}_{G,J}) \xrightarrow{\sim} \Cal{H}_{G,K}.
\]

\end{thm}

\begin{defn}\label{IC basis hecke algebra}
For $\mu \in X_*(T)_+$, 
\begin{itemize}
\item We denote by $f_{\mu}'$ the unique element of $Z(\Cal{H}_{G,J})$  such that $f_{\mu}' * \mbb{I}_K = f_{\mu}$.
\item We denote by $\psi_{\mu}' \in Z(\Cal{H}_{G,J}) $  the unique element such that  $ \psi_{\mu}' * \mbb{I}_K = \psi_{\mu} \in \Cal{H}_{G,K}$. 
\end{itemize}
\end{defn}

\subsubsection{Geometrization of parahoric Hecke algebras} There is a  geometrization of the parahoric Hecke algebra analogous to \S \ref{sssec: geom spherical HA}, which goes as follows (it will be elaborated on later in \S \ref{hecke base change}). Briefly, let $\Cal{G}$ be the parahoric group scheme corresponding to $J$ by Bruhat-Tits theory. Then the Hecke algebra $\Cal{H}_{G,J}$ the Grothendieck ring (i.e. the group completion of equivalence classes of objects, with ring structure induced by the convolution) of $\mrm{Perv}_{\Cal{G}(\Cal{O})}(\Gr_{\Cal{G}})$. 

Note that if $J=I$ is an Iwahori subgroup (the stabilizer of full alcove), then $\Gr_{\Cal{G}}$ is the affine flag variety $\mrm{Fl}_G$. In general $\Gr_{\Cal{G}}$ is a partial affine flag variety, which one can think of as a mix between the affine Grassmannian and affine flag variety. 

One might ask which sheaves the functions $\psi_{\mu}'$ correspond to. The answer is that they can be realized as nearby cycles of certain global degenerations, and it is the key point underlying this section.

\subsection{Nearby cycles}\label{subsec: nearby cycles}

We recall the definition and essential (for us) properties of the nearby cycles functor. For a reference, see \cite[Expos\'{e} XIII]{Del73}. 

A \emph{Henselian trait} is a triple $(S, s,\eta)$ where $S$ is a the spectrum of a discrete valuation ring, $s$ is the special point of $S$ and $\eta$ is the generic point of $S$. Choose geometric points $\ol{s}$ and $\ol{\eta}$ lying over $s$ and $\eta$, respectively. Let $\ol{S}$ be the normalization of $S$ in $\ol{\eta}$. We denote by 
\begin{align*}
\ol{i} &\co \ol{s} \rightarrow \ol{S} \\
\ol{j} &\co \ol{\eta} \rightarrow \ol{S}
\end{align*}
the obvious maps. 

Let $f \co Y \rightarrow S$ be a finite type scheme over $S$. One defines a topos $Y \times_s \eta$ as in \cite[Expos\'{e} XIII \S 1.2]{Del73}, so that the category $D_c^b(Y \times_s \eta, \ol{\Q}_{\ell})$ is the category of $\Cal{F} \in D_c^b(Y \times_s \ol{s}, \ol{\Q}_{\ell})$ together with a continuous $\Gal(\ol{\eta}/\eta)$-action compatible with the $\Gal(\ol{s}/s)$-action on $Y_{\ol{s}}$ via the natural map $\Gal(\ol{\eta}/\eta) \rightarrow \Gal(\ol{s}/s)$.

\begin{defn}
Given $\Cal{F} \in D_c^b(Y_{\eta}, \ol{\Q}_{\ell})$ we define the \emph{nearby cycles} $R\Psi(\Cal{F}) \in D_c^b(Y \times_s \ol{s}, \ol{\Q}_{\ell})$ by 
\[
R\Psi(\Cal{F})  := \ol{i}^* R\ol{j}_* (\Cal{F}_{\ol{\eta}})
\]
with the $\Gal(\ol{\eta}/\eta)$-action obtained by transport of structure from that on $\Cal{F}_{\ol{\eta}}$. 
\end{defn}

\begin{remark}\label{rem: nearby cycles descent}
When the nearby cycles construction is performed with $S  = \Spec \F_q[\![t]]$, the sheaf $R\Psi(\Cal{F})$ is a priori only defined over $Y_{\ol{\F}_q}$, but can be descended to $Y_{\F_q}$ by choosing a splitting $\Gal(\ol{\F}_q/\F_q) \rightarrow \Gal(\ol{\F_q((t))}/\F_q((t)))$. When dealing with nearby cycles on affine Grassmannians (or related objects) this is often what we mean (see \cite[Footnote 4 on page 8]{Gaits01}). Only after such a descent one can associate a trace function to $R\Psi(\Cal{F})$. We shall point out when this descent is being used, but as a blanket rule it is necessary every time we wish to talk about a trace function. 
\end{remark}

\begin{lemma}\label{nearby cycles proper}
If $f \co Y \rightarrow S$ is proper, then the base change homomorphism 
\[
 R\Psi f_* \rightarrow f_* R\Psi 
\]
is an isomorphism.
\end{lemma}

\begin{proof}
This is \cite[Expos\'{e} XIII (2.1.7.1)]{Del73}. 
\end{proof}

\begin{cor}\label{cor: nearby cycles = generic coh}
If $f \co Y \rightarrow S$ is proper, then the natural map
\[
H_c^i(Y_{\ol{\eta}}, \ol{\Q}_{\ell}) \rightarrow H_c^i(Y_{\ol{s}}, R \Psi(\ol{\Q}_{\ell})).
\] 
is an isomorphism. 
\end{cor}

\begin{lemma}\label{nearby cycles smooth}
If $f \co Y \rightarrow S$ is lisse, then the base change homomorphism 
\[
f^* R\Psi \rightarrow R\Psi f^*
\]
is an isomorphism.
\end{lemma}

\begin{proof}
This is \cite[Expos\'{e} XIII (2.1.7.2)]{Del73}. 
\end{proof}

\subsection{Degeneration to affine flag varieties}

Let $X$ be a smooth curve (not necessarily projective) over $\F_q$ and let $\Cal{G} \rightarrow X$ be a parahoric group scheme, with parahoric level structure at $x_0$. We consider the global affine Grassmannian for $\Cal{G}$ 
as in \S \ref{subsec: G-bundles}: 
\[
\pi \co \Gr_{\Cal{G}} \rightarrow X.
\]
Consider the restriction $\Gr_{\Cal{G}}|_{D_{x_0}}$, where $D_{x_0} := \Spec \Cal{O}_{x_0}$ is the spectrum of the completed local ring at $x_0$.  We apply the nearby cycles construction of \S \ref{subsec: nearby cycles} and in the form of Remark \ref{rem: nearby cycles descent}, to 
\begin{itemize}
\item $S = D_{x_0}$, $Y = \Gr_{\Cal{G}}|_{D_{x_0}}$, and 
\item $\Cal{F} = \mrm{Sat}_{\Gr_{\Cal{G}} }(\mu)|_{D_{x_0}^{*}}$, where $D_{x_0}^{*} = \Spec F_{x_0}$ is thought of as a local ``punctured disk'' around $x_0$. 
\end{itemize}
This produces a $\Cal{L}^+\Cal{G}$-equivariant perverse sheaf $R \Psi(\Sat_{\Gr_{\Cal{G}}}(\mu)|_{D_{x_0}^{\circ}})$ on $\Gr_{\Cal{G}}|_{x_0}$, which we will abbreviate by $R \Psi(\Sat_{\Gr_{\Cal{G}}}(\mu))$.

\begin{thm}[Gaitsgory \cite{Gaits01}, Zhu \cite{Zhu14}]\label{gaitsgory central}
The sheaf $R \Psi(\Sat_{\Gr_{\Cal{G}}}(\mu)))$ is central. 
\end{thm}

\begin{remark} 
In the present formulation and level of generality, this theorem is actually due to
X. Zhu in \cite[Theorem 7.3]{Zhu14}. Gaitsgory proved the first prototype of Theorem \ref{gaitsgory central}, but working with constant group schemes $\Cal{G}$, and a slightly different degeneration.
\end{remark}

\begin{cor}\label{gaitsgory nearby function}
Assume that $G := \Cal{G}|_{F_{x_0}}$ is split. Then the trace function (in the sense of \eqref{eq: trace function}) associated to $R \Psi(\Sat_{\Gr_{\Cal{G}}}(\mu)))$ is $\psi_{\mu}'$ (Definition \ref{IC basis hecke algebra}). 
\end{cor}

\begin{remark}
Note that we need to use Remark \ref{rem: nearby cycles descent} to descend $R \Psi(\Sat_{\Gr_{\Cal{G}}}(\mu)))$ to $\Gr_{\Cal{G}}|_{x_0}$, so that it makes sense to speak of the trace function. 
\end{remark}

\begin{proof}
 Since $R \Psi(\Sat_{\Gr_{\Cal{G}}}(\mu))$ is a $\Cal{L}^+ \Cal{G}$-equivariant perverse sheaf on $\Gr_{\Cal{G}}|_{x_0}$, which is central by Theorem \ref{gaitsgory central}, we have a priori that its trace function
\[
\Tr(\Frob, R \Psi(\Sat_{\Gr_{\Cal{G}}}(\mu)))
\]
lies in $ Z(\Cal{H}_{G,\Cal{G}(\Cal{O}_{x_0}) } (F_x))$. Since $G$ is split we can extend it to a constant group scheme over $D_{x_0}$, which we continue to denote $G$, such that $G(\Cal{O}_{x_0}) =: K$ is a hyperspecial maximal compact subgroup of $G(F_x)$. Write also $J := \Cal{G}(\Cal{O}_{x_0})$ for the parahoric subgroup. By the Bernstein isomorphism (Theorem \ref{thm: bernstein isom})
\[
- * \mbb{I}_K \co Z(\Cal{H}_{G,J }) \xrightarrow{\sim} \Cal{H}_{G,K}
\]
it suffices to check that 
\begin{equation}\label{eqn: convolve spherical}
\Tr(\Frob, R \Psi(\Sat_{\Gr_{\Cal{G}}}(\mu))) * \mbb{I}_K =  \psi_{\mu}.
\end{equation}
By \cite[Theorem 1 (d)]{Gaits01} the  map \eqref{eqn: convolve spherical} is realized sheaf-theoretically by the pushforward via the proper map 
\[
\mrm{pr} \co \Gr_{\Cal{G}} \rightarrow \Gr_{G}
\]
or in other words, 
\[
\Tr(\Frob, R \Psi(\Sat_{\Gr_{\Cal{G}}}(\mu))) * \mbb{I}_K = \Tr(\Frob, \mrm{pr}_! R \Psi(\Sat_{\Gr_{\Cal{G}}}(\mu))).
\]
Now, by Lemma \ref{nearby cycles proper} and the fact that $\mrm{pr}$ is an isomorphism over $D_{x_0}^*$ (since $G|_{D_{x_0}^*} \cong \Cal{G}|_{D_{x_0}^*}$) we have
\[
\mrm{pr}_! R\Psi (\Sat_{\Gr_{\Cal{G}}}(\mu)) = R\Psi ( \Sat_{\Gr_{G}}(\mu))
\]
but since $\Cal{G}_0|_{D_x} \rightarrow D_x$ is constant, we simply have 
\[
R\Psi ( \Sat_{\Gr_{G}}(\mu)) = \Sat_{\Gr_G}(\mu)|_{x_0},
\]
whose trace function is $\psi_{\mu}$ by definition. 
\end{proof}

 \subsubsection*{Schubert stratification} Let $\Gr_{\Cal{G}, x_0}$ be the fiber of $\Gr_{\Cal{G}}$ over $x_0$. We discuss the stratification induced by the $\Cal{G}(\Cal{O}_{x_0})$-action on $\Gr_{\Cal{G}, x_0}$. 
 
 The analogue of the Cartan decomposition \eqref{eq: cartan decomp} is 
 \[
 \Cal{G}(\Cal{O}_{x_0}) \bs \Cal{G}(F_{x_0}) / \Cal{G}(\Cal{O}_{x_0}) \cong \wt{W}_J \bs \wt{W} /\wt{W}_{J}
 \]
 where $\wt{W}$ is the extended affine Weyl group, and $\wt{W}_J$ is the subgroup corresponding to the parahoric subgroup $J :=  \Cal{G}(\Cal{O}_{x_0})$. We refer to \cite[\S 2.6]{Haines09} for the notation and definitions; all that we require are the following abstract facts: 
 \begin{itemize}
 \item The $\Cal{G}(\Cal{O}_{x_0})$-orbits on $\Gr_{\Cal{G}, x_0}$ are indexed by $\nu \in \wt{W}_J \bs \wt{W} /\wt{W}_{J}$. We denote the orbit corresponding to $\nu  \in \wt{W}_J \bs \wt{W} /\wt{W}_{J}$ by $\Gr_{\Cal{G}, x_0}^{= \nu}$ and its closure by $\Gr_{\Cal{G}, x_0}^{ \leq \nu}$. 
 \item There is a partial order on $\wt{W}_J \bs \wt{W} /\wt{W}_{J}$, which can be characterized by the property that $\mu \geq \nu$ if and only if $\Gr_{\Cal{G}, x_0}^{ \leq \mu} \supset \Gr_{\Cal{G}, x_0}^{= \nu}$. 
 \end{itemize}

\begin{defn} Since $R\Psi(\Sat_{\Gr_{\Cal{G}}}(\mu))$ is a perverse sheaf on $\Gr_{\Cal{G}, x_0}$, equivariant for the proalgebraic group underlying $J$ \cite[Theorem 1]{Gaits01}, its stalks are the same on any open Schubert cell $\Gr_{\Cal{G}}^{ =\nu}$. We denote this common stalk by $R\Psi(\Sat_{\Gr}(\mu))_{\nu}$. 
\end{defn}

\begin{lemma}\label{lem: parahoric hecke formula}
We have 
\[
\psi_{\mu}' = \sum_{\nu \leq \mu} \Tr(\Frob, R\Psi( \Sat_{\Gr_{\Cal{G}}}(\mu))_{\nu}) f_{\nu}
\]
\end{lemma}

\begin{proof}
The argument is the same as for Lemma \ref{lem: spherical hecke formula}.
\end{proof}

\subsection{Local models for shtukas}

\begin{defn}
Let $\Cal{X}$ and $\Cal{Y}$ be Artin stacks. We say that $\Cal{Y}$ is a \emph{smooth local model} for $\Cal{X}$ if there exists an Artin stack $\Cal{W}$ and a diagram
\[
\begin{tikzcd}
& \Cal{W} \ar[dl, twoheadrightarrow, "\text{smooth}"', "f"] \ar[dr, "\text{smooth}", "g"']  \\
\Cal{X} & & \Cal{Y} 
\end{tikzcd}
\]
A diagram as above is called a (smooth) \emph{local model diagram}.
\end{defn}

\begin{thm}\label{varshavsky local model}
The stack $\Gr_{\Cal{G}, X^r}^{\leq (\mu_1, \ldots, \mu_r)} $ is a smooth local model for $\Sht_{\Cal{G}}^{ \leq (\mu_1, \ldots, \mu_r)}$. 
\end{thm}

\begin{remark}
We learned from the referee that Theorem \ref{varshavsky local model} -- in fact, the stronger statement that one gets an \emph{\'{e}tale} local model -- has already been proved for smooth affine group schemes in \cite[\S 3]{HS18}. (Our proof also only uses that $\Cal{G}$ is smooth and affine.)
\end{remark}

\begin{proof}

 For ease of presentation, we assume that $r=1$ in the proof; the argument for the general case is a completely straightforward generalization.
 
As in \cite[Proposition 2.11]{Laff18} (beginning of the proof) we can add a level structure at a closed subscheme $N \subset X$ to rigidify all spaces under consideration from stacks to schemes. Since this addition of level structure induces smooth covers of all objects, we will suppress it from the notation; the upshot is that we can reason with all objects as if they were schemes. 

We let $\Bun_{G, n \Gamma}$ be the smooth torsor over $\Bun_{\Cal{G}}$ as in \S \ref{sssec: level structure}, and 
 $\Hecke_{\Cal{G}, n\Gamma}$ the pullback torsor over $\Hecke_{\Cal{G}}$. Informally, $\Hecke_{\Cal{G}, n \Gamma}$ parametrizes modifications $(x,\varphi \co \Cal{E} \dashrightarrow \Cal{E}')$ together with the additional datum of a trivialization of $\Cal{E}'$ on $\Gamma_{nx}$.

 Since for any given $\mu$ the $\Cal{L}^+ \Cal{G}$-action on $\Hecke_{\Cal{G}}^{\leq \mu}$ and $\Gr_{\Cal{G}}^{\leq \mu}$ factors through a quotient group scheme of finite type, Proposition \ref{Hecke to Gr} implies that for $n$ sufficiently large relative to fixed $\mu$, there is an isomorphism $\Hecke_{\Cal{G}, n \Gamma}^{\leq \mu} \rightarrow \Gr_{\Cal{G}}^{\leq \mu} \times \Bun_{\Cal{G},n\Gamma}$.

Consider the diagram below, where all squares are cartesian: 
\begin{equation}\label{eq: local model big diagram}
\begin{tikzcd}
\Sht_{\Cal{G}, n \Gamma}^{\leq \mu} \ar[d] \ar[r,  twoheadrightarrow, "\text{smooth}"] & \Sht_{\Cal{G}}^{\leq \mu} \ar[r]   \ar[d] & \Bun_{\Cal{G}}  \ar[d, "\Id \times \Frob"] \\
\Hecke_{\Cal{G}, n \Gamma} \ar[r,  twoheadrightarrow,  "\text{smooth}"] \ar[d, "\sim"]  & \Hecke_{\Cal{G}}^{\leq \mu}  \ar[r, "h^{\leftarrow} \times h^{\rightarrow}"] & \Bun_{\Cal{G}}  \times \Bun_{\Cal{G}} \\
 \Bun_{\Cal{G}, n \Gamma} \times \Gr_G^{\leq \mu}
\end{tikzcd}
\end{equation}
It suffices to show that map $\Sht_{\Cal{G}, n \Gamma}^{\leq \mu}  \rightarrow \Gr_G^{\leq \mu}$, induced by the composition of the leftmost vertical arrows, is smooth. This follows by the same transversality calculation in \cite[Lemma 4.3]{Var04} applied to the outer diagram in \eqref{eq: local model big diagram}
\begin{equation}
\begin{tikzcd}
\Sht_{\Cal{G}, n \Gamma}^{\leq \mu} \ar[d] \ar[r] & \Bun_{\Cal{G}}  \ar[d, "\Id \times \Frob"] \\
\Hecke_{\Cal{G}, n \Gamma}  \ar[d, "\sim"] \ar[r]  & \Bun_{\Cal{G}}  \times \Bun_{\Cal{G}} \\
 \Bun_{\Cal{G}, n \Gamma} \times \Gr_G^{\leq \mu}
\end{tikzcd}
\end{equation}
using that $ \Bun_{\Cal{G}, n \Gamma} \rightarrow \Bun_{\Cal{G}}$ is smooth. 
\end{proof}

\begin{cor}\label{cor: local model sheaf}
Let $\ul{\mu} \in X_*(T)^r$. There is a local model diagram  
 \[
\begin{tikzcd}
& \Cal{W}^{\leq \ul{\mu}} \ar[dl, twoheadrightarrow, "\text{smooth}"', "f"] \ar[dr, "\text{smooth}", "g"']  \\
\Sht_{\Cal{G}}^{\leq \ul{\mu}}  \ar[dr] & & \Gr_{\Cal{G}}^{\leq \ul{\mu}} \ar[dl] \\
&X^r
\end{tikzcd}
\]
with 
\[
f^* \Sat_{\Sht_{\Cal{G}}}(\ul{\mu}) =  g^* \Sat_{\Gr_{\Cal{G}}}(\ul{\mu}).
\]
\end{cor}

\begin{proof}
This follows from the diagram \eqref{eq: local model big diagram} and the definition of the Satake sheaves, taking $\Cal{W}^{\leq \ul{\mu}}  = \Sht_{\Cal{G}, N}^{\leq \ul{\mu}}$. 
\end{proof}

\subsection{The trace function of nearby cycles}
Recall that the nearby cycles $R \Psi(\Cal{F})$ has an action of the inertia group. We have a decomposition 
\[
R\Psi(\Cal{F}) \cong R\Psi(\Cal{F})^{\mrm{un}} \oplus R\Psi(\Cal{F})^{\mrm{non-un}}
\]
into unipotent and non-unipotent parts for this inertial action. The associated trace function is independent of the choice of splitting in Remark \ref{rem: nearby cycles descent} if $R\Psi(\Cal{F})$ is unipotent, i.e. if $R\Psi(\Cal{F})^{\mrm{non-un}}  = 0$. (Otherwise, to get a well-defined trace function we need to project to the unipotent summand -- this is the ``semisimple trace of Frobenius''.)

\begin{lemma}\label{lem: nearby cycles unipotent}
The complex $R \Psi_{x_0}(\Sat_{\Sht_{\Cal{G}}}(\ul{\mu}))$ is unipotent, i.e. 
\[
R \Psi_{x_0}(\Sat_{\Sht_{\Cal{G}}}(\ul{\mu}))^{\mrm{non-un}} = 0.
\]
\end{lemma}

\begin{proof}
By Corollary \ref{cor: local model sheaf} plus the compatibility of the inertia action with the isomorphism of Lemma \ref{nearby cycles smooth}, it suffices to know that $R \Psi_{x_0}(\Sat_{\Gr_{\Cal{G}}}(\ul{\mu}))^{\mrm{non-un}} = 0$. This is established in \cite[\S 5.1 Proposition 7.]{Gaits01}.
\end{proof}

Let $\ul{\mu} \in X_*(T)^r$. By Corollary \ref{cor: local model sheaf}, we may set 
\[
\Sat_{\Cal{W}}(\ul{\mu}) :=  f^* \Sat_{\Sht_{\Cal{G}}}(\ul{\mu}) =  g^* \Sat_{\Gr_{\Cal{G}}}(\ul{\mu}).
\]
We write $R\Psi_{x_0}$  to emphasize that we are taking nearby cycles over the point $x_0$. By Lemma \ref{nearby cycles smooth}, and implicitly using Corollary \ref{cor: local model sheaf}, we have 
\[
f^* R\Psi_{x_0}(\Sat_{\Sht_{\Cal{G}}}(\ul{\mu}))  =  R\Psi_{x_0}(\Sat_{\Cal{W}}(\ul{\mu})) = g^* R\Psi_{x_0}(\Sat_{\Gr_{\Cal{G}}}(\ul{\mu})). 
\]
Thus, for $w \in \Cal{W}^{\leq \ul{\mu}}(k)$ lying over $y \in \Sht_{\Cal{G}}^{\leq \ul{\mu}}(k)$ and $z \in \Gr_{\Cal{G}}^{\leq \ul{\mu}}(k)$, we have 
\[
\Tr(\Frob, R\Psi_{x_0}(\Sat_{\Sht_{\Cal{G}}}(\ul{\mu}))_y) = \Tr(\Frob, R\Psi_{x_0}(\Sat_{\Cal{W}}(\ul{\mu}))_w) =  \Tr(\Frob, R \Psi_{x_0}(\Sat_{\Gr_{\Cal{G}}}(\ul{\mu}))_z). 
\]
Therefore, the stalks of $R \Psi_{x_0}(\Sat_{\Sht_{\Cal{G}}}(\ul{\mu}))$ are constant along the stratification 
\[
\Sht_{\Cal{G}}^{\leq \ul{\mu}} = \coprod_{\ul{\nu} \leq \ul{\mu}} \Sht_{\Cal{G}}^{=\ul{\nu}},
\]
and we deduce:

\begin{cor}\label{cor: kottwitz for shtukas}
For $\ul{\nu} \in X_*(T)_+$, we have  
\[
\Tr(\Frob, R \Psi_{x_0}(\Sat_{\Sht_{\Cal{G}}}(\ul{\mu}))_{\ul{\nu}} ) = \Tr(\Frob, R\Psi_{x_0}( \Sat_{\Gr_{\Cal{G}}}(\ul{\mu}))_{\ul{\nu}}).
\]
\end{cor}

\begin{remark}We will actually need to work with $\Sht_{\Cal{G}}/a^{\Z}$ instead. Since this is obtained from $\Sht_{\Cal{G}}$ by gluing isomorphic components, the result is exactly the same. 
\end{remark}

\section{Counting parahoric shtukas} \label{counting fixed points}

Our eventual goal is to establish a formula for the trace of an operator, formed as a composition of Hecke operators and Frobenius, on the cohomology of the nearby  cycles sheaf of (a variant of) $\Sht_{\Cal{G}}/a^{\Z} \rightarrow X$, at a place of parahoric bad reduction. The mold for such calculations was set by Kottwitz in \cite{Kott92}, who computed this sort of trace for certain PEL Shimura varieties, at places of good (hyperspecial) reduction. It has since been extended vastly by work of many authors; we note that in particular that Kisin and Pappas constructed integral models for Shimura varieties with parahoric level structure (a problem which itself has a long history, with contributions from many authors -- see the references in \cite{KP15}) and computed the trace of Frobenius on nearby cycles for unramified groups in \cite{KP15}, and for tamely ramified groups in \cite{HRpre}. Our result is a function field analogue of this computation. 

In this section we carry out one step of this calculation, which deals with counting the number of fixed points of Frobenius composed with Hecke correspondences. (The precise setup will be explained in \S \ref{subsec: hecke setup}.) In fact most of the work has already been done by B.C. Ng\^{o}  and T. Ng\^{o} Dac, who studied the case of moduli of shtukas with hyperspecial reduction in the series of papers \cite{Ngo06}, \cite{NgoNgo}, \cite{Ngo13}, and \cite{Ngo15}. The only new element here is that we are considering parahoric reduction. We note also that our results should follow from work of Hartl and Arasteh Rad proving the analogue of the Langlands-Rapoport Conjecture for shtukas \cite{HR16}. 

\subsection{Setup}\label{subsec: hecke setup}
Throughout this section $\cX$ is an unspecified open subset $X$, which in the case of $\Cal{D}$-shtukas will be $X-Z$ where $Z$ is the set of ramification places of $D$. We let $G$ be a quasi-split, connected reductive group over $F$ with simply-connected derived group, or the group attached to a division algebra $D$ as in \S \ref{subsec: D-shtukas}. (This unwieldy hypothesis is in place because the statements of \cite{NgoNgo} and \cite{Ngo13} use the first general hypothesis, but apply also $\Cal{D}$-shtukas, cf. \cite[\S 4]{Ngo06}, and we are also interested in the latter.) Let $\Cal{G} \rightarrow X$ a parahoric group scheme, with parahoric reduction at $x_0$.

Let $K_v = \Cal{G}(\Cal{O}_v)$. Let $K_v t^{\beta_v} K_v \in K_v \backslash G(F_v) / K_v $ be a choice of double coset for all $v$, trivial for almost all $v$. Let $T' \subset \cX$ be the set of all $v$ where $\beta_v \neq  0$, i.e. where the corresponding Hecke operator $h_{\beta_v}$ is not the identity. \emph{We assume that $K_v$ is hyperspecial for all $v \in T'$.}

There is a Hecke correspondence (\S \ref{sssec: hecke corr for shtuka})
 \begin{equation}\label{eqn: hecke corr on sht}
\begin{tikzcd}
&  \Hecke(\Sht_{\Cal{G}})^{\leq \beta_v}_v /a^{\Z} \ar[dl, "h^{\leftarrow}"'] \ar[dr, "h^{\rightarrow}"] \\
\Sht_{\Cal{G}}/a^{\Z} \ar[dr, "\pi"] &  &  \Sht_{\Cal{G}}/a^{\Z}  \ar[dl, "\pi"']  \\
& X
\end{tikzcd}
\end{equation}
for each $\beta_v$. This induces a Hecke operator $h_{\beta_v}$ on the cohomology of $\Sht_{\Cal{G}}/a^{\Z}$ (Definition \ref{def: hecke action on coh}). See \cite[\S 3]{NgoNgo} for more discussion about the Hecke correspondences. 

We abbreviate $\beta := (\beta_v)_{v \in T'}$ and denote the corresponding Hecke operator $\prod h_{\beta_v}$ by $h_{\beta, T'}$. We want to compute (a variant of) 
\[
\Tr(h_{\beta, T'} \circ \Frob, \pi_! R\Psi_{x_0}(\Sat_{\Sht_{\Cal{G}}}(\mu))) 
 \]
where  $x_0$ is our fixed place of parahoric reduction. By the Grothendieck-Lefschetz trace formula, we have 
\begin{align*}	
&\Tr(h_{\beta, T'} \circ \Frob, R\Psi_{x_0}(\Sat_{\Sht_{\Cal{G}}}(\mu)))  \\
& \hspace{1in} = \sum_{\xi \in  \Fix(h_{\beta, T'} \circ \Frob)} \frac{1}{\# \Aut \xi} \Tr(h_{\beta, T'} \circ \Frob_{\Xi}, R\Psi_{x_0}(\Sat_{\Sht_{\Cal{G}}}(\mu))_{\Xi} ).
\end{align*}

We will compute this by focusing first on counting $\Fix(h_{\beta, T} \circ \Frob)$. This was done by \cite{Ngo06} for $\Cal{D}$-shtukas at points of with no level structure (good reduction), and extended by \cite{NgoNgo} for general reductive groups and \cite{Ngo13} for more complicated setups; however, these counts only account for the contribution from the ``elliptic part''. 

In the case where $G_F$ is anisotropic mod center, the elliptic part will obviously compose everything. This is one of the reasons why it is convenient to work with division algebras, and one of the difficulties in carrying out the strategy for general groups. Since $\Sht_{\Cal{G}}$ is of infinite type in general, it will have infinitely many points even over finite fields. 

\subsection{The groupoid of fixed points}
We consider a slightly more general situation. We will define a groupoid $\Cal{C}(\alpha,\beta; T,T'; d)$ which occurs as the fixed points of a composition of Hecke and Frobenius operators on a certain moduli stack of shtukas. Then we will count its mass in the sense of groupoids. 

\begin{defn}If $\Cal{C}$ is a finite groupoid with finite automorphism groups then we define 
\[
\# \Cal{C} := \sum_{c\in \Cal{C}} \frac{1}{\# \Aut(c)}.
\]
Note that the $\F_q$-points of a finite type Deligne-Mumford stack, which includes any Schubert cell in a moduli stack of shtukas, satisfies this assumption.
\end{defn}

\begin{defn}[{\cite[\S 4]{NgoNgo}}]
Let $T, T' \subset |X| - I $. Let 
\begin{align*}
\alpha &= (\alpha_v  \in K_v \backslash G(F_v)  / K_v )_{v \in T}  \\
\beta &= (\beta_v  \in K_v \backslash G(F_v) /  K_v)_{v \in T'}.
\end{align*}
(In terms of the notation of \S \ref{subsec: hecke setup}, we are identifying $\beta_v$ with $K_v t^{\beta_v} K_v$.) We define the groupoid of fixed points $\Cal{C}(\alpha,\beta; T,T'; d)$ as follows: its objects are triples $(\Cal{E}, t, t')$ with 
\begin{enumerate}
\item $t \co \Cal{E}^{\sigma} |_{\ol{X}-\ol{T}} \xrightarrow{\sim} \Cal{E} |_{\ol{X}-\ol{T}}$, with modification type $\alpha$ on $T$, and 
\item $t' \co \Cal{E}^{\sigma^d} |_{\ol{X}-\ol{T}'} \xrightarrow{\sim} \Cal{E} |_{\ol{X}-\ol{T}'}$, with modification type $\beta$ on $T'$,
\item satisfying the following compatibility:
\[
\begin{tikzcd}
\Cal{E}^{\sigma^{d+1}}  |_{\ol{X}-\ol{T} - \ol{T}'}\ar[r, "\sigma^d(t)"] \ar[d, "\sigma(t')"] &  \Cal{E}^{\sigma^d}|_{\ol{X}-\ol{T} - \ol{T}'} \ar[d, "t'"] \\
\Cal{E}^{\sigma}|_{\ol{X}-\ol{T} - \ol{T}'} \ar[r, "t"] &\Cal{E} |_{\ol{X}-\ol{T} - \ol{T}'}
\end{tikzcd}
\]
\end{enumerate}
The automorphisms of $(\Cal{E}, t, t')$ are defined to be automorphisms of $\Cal{E}$ commuting with $t$ and $t'$. 
\end{defn}

The relation to our initial problem is given by the following. 

\begin{lemma}\label{lem: groupoid description 0}
Suppose $x \in |X|$ is a point of degree $d$. Then we have an isomorphism of groupoids 
\[
\Fix(h_{\beta,T'} \circ \Frob, \Sht_{\Cal{G}}^{= \mu}|_x) \cong \Cal{C}(\mu, \beta; \{x\}, T';d).
\]
\end{lemma}

\begin{proof}
This is immediate upon writing down the definitions. 
\end{proof}

We actually want to study the truncated space $\Sht_{\Cal{G}}^{= \mu}/a^{\Z}$, so we modify the discussion accordingly. Let $\Xi \subset Z(G)(\A)$ be a cocompact lattice. Then $\Xi$ acts on $\Sht_{\Cal{G}}^{= \mu}$ via Hecke correspondences, and we define $\Sht_{\Cal{G}}^{= \mu}/\Xi$ to be the quotient. Similarly we define $\Cal{C}(\mu, \beta; \{x\}, T';d)_\Xi$ to be the quotient by the $\Xi$-action. (See \cite[near the end of \S 4]{NgoNgo}, for more details.)

\begin{lemma}\label{lem: equivalence of groupoids}
Suppose $x \in |X|$ is a point of degree $d$. Then we have an isomorphism of groupoids 
\[
\Fix(h_{\beta,T'} \circ \Frob_{x}, \Sht_{\Cal{G}}^{= \mu}/\Xi|_x) \cong \Cal{C}(\mu, \beta; \{x\}, T';d)_\Xi.
\]
\end{lemma}

\begin{proof}
Immediate by taking the quotient of Lemma \ref{lem: groupoid description 0} with respect to the $\Xi$-action. 
\end{proof}

Hence we want to study $\# \Cal{C}(\mu, \beta; \{x\}, T';d)_\Xi$. The strategy for these counts goes back to Kottwitz's study of points of Shimura varieties (with hyperspecial level structure) over finite fields \cite{Kott92}. 
\begin{enumerate}
\item We first show that there is a cohomological invariant, the \emph{Kottwitz invariant}, which controls the possible ``generic fibers'' of members of $\Cal{C}(\alpha,\beta; T,T'; d)$. 
\item We then express the \emph{size of an isogeny class} as a product of (twisted) orbital integrals. 
\item We then express the \emph{number of isogeny classes} associated to each Kottwitz invariant in terms of certain cohomology groups. 
\end{enumerate}
These steps have been carried out in papers of B.C. Ng\^{o} and T. Ng\^{o} Dac, as already mentioned, but not quite in the generality required here. In particular, these previous papers avoid the case where $T$ meets a point with non-trivial level structure (because the moduli problem was not defined over such points), which is exactly the situation that we are interested in. So we will describe the modifications needed to extend the argument to our setting, and only briefly summarize the parts that are already covered in the papers of B.C. Ng\^{o} and T. Ng\^{o} Dac.

\subsection{Kottwitz triples and classification of generic fibers}

Our first step is to define a category that looks like the category of ``generic fibers of $\Cal{C}(\alpha,\beta; T,T'; d)$''. 

\begin{defn}[{\cite[\S 5]{NgoNgo}}] Let $T, T' \subset |X| - I $. We define the groupoid $C(T,T';d)$ as follows: its objects are triples $(V, \tau, \tau')$ with 
\begin{enumerate}
\item $V$ a $G$-torsor over $F_{\ol{k}} := F\otimes_k \ol{k}$, 
\item an isomorphism $\tau \co V^{\sigma}\xrightarrow{\sim} V$, where $V^{\sigma} = V \otimes_{F_{\ol{k}}, \sigma} F_{\ol{k}}$,
\item $\tau' \co V^{\sigma^d} \xrightarrow{\sim}V$,
\end{enumerate}
satisfying the following conditions:
\begin{enumerate}
\item (``commutativity'') The following diagram commutes:
\[
\begin{tikzcd}
V^{\sigma^{d+1}} \ar[r, "\sigma^d(\tau)"] \ar[d, "\sigma(\tau')"] & V^{\sigma^d} \ar[d, "\tau'"] \\
V^{\sigma} \ar[r, "\tau"] & V
\end{tikzcd}
\]
\item For $x \notin T$, $(V_x, \tau_x)$ is isomorphic to the trivial isocrystal $(G(F_x \wh{\otimes}_{\F_q} \ol{\F}_q), \Id \wh{\otimes}_{\F_q} \sigma)$. 
\item For $x \in T$, $(V_x, \tau_x')$ is isomorphic to the trivial isocrystal $(G(F_x \wh{\otimes}_{\F_q}  \ol{\F}_q), \Id \wh{\otimes}_{\F_{q^d}} \sigma^d)$.
\end{enumerate}
The automorphisms of $(V, \tau, \tau')$ are automorphisms of $\Cal{E}$ commuting with $\tau$ and $\tau'$. 
\end{defn}

The operation of ``taking the generic fiber'' defines a functor \cite[\S 5.2]{NgoNgo}
\[
\Cal{C}(\alpha,\beta; T,T'; d) \rightarrow C(T,T' ;d). 
\]

\subsubsection{Kottwitz triples} Recall that a \emph{Kottwitz triple} is a datum $(\gamma_0, (\gamma_x)_{x \notin T}, (\delta_x)_{x \in T})$ where:
\begin{itemize}
\item $\gamma_0 $ is a stable conjugacy class of $G(F)$,
\item $\gamma_x$ is a conjugacy class of $G(F_x)$ for each $x \notin T$, and is stably conjugate to $\gamma_0$,
\item $\delta_x$ is a $\sigma$-conjugacy class of  $G(F_x \wh{\otimes}_{\F_q} \F_{q^d})$, whose norm 
\[
N(\delta_x) := \delta_x  \cdot \sigma( \delta_x) \cdot \ldots \cdot \sigma^{d-1}(  \delta_x) 
\]
is stably conjugate to $\gamma_0$. 
\end{itemize}

\begin{const}\label{construction of Kottwitz triple} We now recall from \cite[\S 6.1]{NgoNgo} how to attach to each $(V, \tau, \tau') \in C(T,T'; d)$ a Kottwitz triple $(\gamma_0, (\gamma_x)_{x \notin T}, (\delta_x)_{x \in T})$.

 First a remark on notation: for a map $\tau \co V^{\sigma} \rightarrow V$, we denote by $\tau^n \co V^{\sigma^n} \rightarrow V$ the map $\tau  \circ \sigma(\tau) \circ \ldots \circ \sigma^{n-1}(\tau)$.

\begin{enumerate}
\item \emph{Definition of $\gamma_0$.} Since $F_{\ol{k}}$ has cohomological dimension 1, the $G$-torsor $V$ is split over $F_{\ol{k}}$. Consider $\gamma =  \tau^d (\tau')^{-1}$, which is a \emph{linear} automorphism of $V$. Using the  ``commutativity'' axiom we find that
\[
\sigma(\gamma) = \sigma(\tau) \circ \sigma^2(\tau) \circ \ldots \circ \sigma^d(\tau) \circ \sigma(\tau')^{-1} = \tau^{-1} \gamma \tau.
\]
This shows that the conjugacy class of $\gamma$ is stable under $\sigma$, hence defined over $F$. Since $G$ was assumed to be quasi-split with simply connected derived subgroup, this conjugacy class must then contain an $F$-point. Thus, we have an element $\gamma_0 \in G(F)$ whose stable conjugacy class is well-defined. \\

\item \emph{Definition of $\gamma_x$, $x \notin T$.} By assumption, we can pick an isomorphism 
\[
(V_x, \tau_x) \cong (G(F_x \wh{\otimes}_{\F_q} \ol{\F}_q), \Id \wh{\otimes}_{\F_q} \sigma).
\]
Since $\tau$ and $\tau'$ commute, so do $\tau_x$ and $\tau_x'$, so that $\tau_x'$ defines an automorphism of $(G(F_x \wh{\otimes}_{\F_q} \ol{\F}_q), \Id \wh{\otimes}_{\F_q} \sigma)$. We can then write $\tau_{x}' = \gamma_x^{-1} \otimes \sigma^d$, for some $\gamma_x \in G(F_x)$ which is stably conjugate to $\gamma_0$. (The point is that picking this trivialization of $\tau_x$ amounts to setting ``$\tau_x =\Id $'' in the equation $\gamma_x = \tau_x^d (\tau_x')^{-1}$.) 

\item \emph{Definition of $\delta_x$, $x \in T$.} By assumption, we can pick an isomorphism 
\[
(V_x, \tau_x') \cong (G(F_x \wh{\otimes}_{\F_q} \ol{\F}_q), \Id \wh{\otimes}_{\F_q^d} \sigma^d).
\]
Since $\tau$ and $\tau'$ commute, so do $\tau_x$ and $\tau_x'$, so that $\tau_x$ defines an automorphism of $(G(F_x \wh{\otimes}_{\F_q} \ol{\F}_q), \Id \wh{\otimes}_{\F_q} \sigma)$. We can then write $\tau_{x} =  \delta_x \otimes \sigma$, for some $\delta_x \in G(F_x)$, well-defined up to $\sigma$-conjugacy, whose norm $N(\delta_x)  =  \delta_x  \cdot \sigma( \delta_x) \cdot \ldots \cdot \sigma^{r-1}(  \delta_x)$ is stably conjugate to $\gamma_0$. 

\end{enumerate}
\end{const}

\begin{defn}
We say that $(V, \tau, \tau') \in C(T,T';d)$ is \emph{semisimple} if $\gamma_0$  is semisimple, and \emph{elliptic} if $\gamma_0$ is elliptic. 

We say $(\Cal{E},t, t')  \in \Cal{C}(\alpha, \beta; T,T'; d)$ is \emph{semisimple} (resp. \emph{elliptic}) if the associated $(V,\tau,\tau')$ is semisimple (resp. elliptic).

\end{defn}

\subsubsection{The Kottwitz invariant} Following \cite[\S 6.2]{NgoNgo} we can attach to the Kottwitz triple $(\gamma_0, (\gamma_x), (\delta_x))$ a character $\inv(\gamma_0, (\gamma_x), (\delta_x)) \in X^*(Z(\wh{G}_{\gamma_0})^{\Gamma})$. Briefly, this is done as follows. 

\begin{itemize}
\item For $x \notin T$, since $\gamma_x$ and $\gamma_0$ are stably conjugate, by a theorem of Steinberg we can find $g \in G(F_x \wh{\otimes}_{k} \ol{k})$ such that 
\[
g \gamma_0 g^{-1} = \gamma_x.
\] 
Then (using that $\gamma_0 \in G(F)$) we have 
\[
g \gamma_0 g^{-1} = \gamma_x = \sigma(\gamma_x) = \sigma(g) \gamma_0 \sigma(g)^{-1}.
\]
This shows that $g^{-1}\sigma(g)$ is in the centralizer  of $\gamma_0$ in $G(F_x \wh{\otimes}_{k} \ol{k})$, hence defines a class in $B(G_{\gamma_0, x}) = G_{\gamma_0}(F_x \wh{\otimes}_{k} \ol{k})/\text{$\sigma$-conjugacy}$. 

\item For $x \in T$, we have $g \in G(F_x \wh{\otimes}_{k} \ol{k})$ such that $N \delta_x = g \gamma_0 g^{-1}$. Then $g^{-1}  \sigma^d(g)$ is in $G_{\gamma_0}(F_x\wh{\otimes}_{k} \ol{k})$ and defines a class in $B(G_{\gamma_0,x})$. 
\end{itemize}
For each $x$, we apply the map $B(G_{\gamma_0, x}) \rightarrow X^*(Z(\wh{G}_{\gamma_0})^{\Gamma_x})$ of \cite[\S 6]{Kott90} to get a local character $\inv_x(\gamma_0, (\gamma_x), (\delta_x)) \in  X^*(Z(\wh{G}_{\gamma_0})^{\Gamma_x})$. Almost all of the resulting characters are trivial, so that it makes sense to sum the restrictions of all these characters to $Z(\wh{G}_{\gamma_0})^{\Gamma}$, and we define this sum to be $\inv(\gamma_0, (\gamma_x), (\delta_x))$.

\begin{prop}[\cite{NgoNgo} Proposition 7.1] For elliptic $(V,\tau, \tau') \in C(T,T'; d)$, and $(\gamma_0, (\gamma_x), (\delta_x))$ the associated Kottwitz triple, if $\gamma_0$ is semisimple then we have $\inv(\gamma_0, (\gamma_x), (\delta_x)) = 0$.
\end{prop}

\begin{prop}\label{prop: gen fibers}
There exists $(V,\tau, \tau') \in C(T,T'; d)$ having a given elliptic Kottwitz triple $(\gamma_0, (\gamma_x), (\delta_x))$ if and only if $\inv (\gamma_0, (\gamma_x), (\delta_x)) = 0$. If the set of such is non-empty, then the number of isogeny classes within $C(T,T'; d)$ having the same Kottwitz triple is the cardinality of 
\[
\ker^1(F,G_{\gamma_0}) := \ker (H^1(F,G_{\gamma_0}) \rightarrow \prod_x H^1(F_x, G_{\gamma_0})).
\]
\end{prop}

\begin{proof}
This follows from the proof of \cite[Proposition 11.1]{NgoNgo} combined with \cite[Proposition 4.3]{Ngo13}.
\end{proof}

\subsubsection{Automorphisms of the generic fiber}\label{sssec: aut generic fib}

Let $(V, \tau, \tau') \in  C(T,T';d)$ with elliptic Kottwitz triple $(\gamma_0, (\gamma_x) , (\delta_x))$. By \cite[\S 3.9]{Ngo13}, the automorphisms of $(V, \tau, \tau')$ are the $F_x$-points of an inner form $J_{\gamma_0}$ of $G_{\gamma_0}$ defined over $F$. 

\begin{remark}\label{inner form local} As pointed out in \cite[\S 3.9]{Ngo13}, the Hasse principle implies that $J_{\gamma_0}$ is determined by its local components: 
\begin{itemize}
\item For $x \notin T$, $J_{\gamma_0,x}$ is the centralizer of $\gamma_x$ in $G(F_x)$, 
\item For $x \in T$, $J_{\gamma_0,x}$ is the twisted centralizer of $\delta_x$ in $G(F_x \otimes_{\F_q} \F_{q^d})$.
\end{itemize}
\end{remark}

\subsection{Counting lattices}

We now study the fibers of the functor $\Cal{C}(\alpha, \beta; T,T'; d) \rightarrow  C(T,T';d)$. 

\begin{prop}\label{prop: lattice count}
Fix an elliptic $(V, \tau, \tau') \in C(T,T';d)$. Suppose $(\Cal{E}, t, t') \in \Cal{C}(\alpha, \beta; T,T'; d)$ lies over $(V, \tau, \tau')$. The size of the isogeny class of $(\Cal{E}, t, t')$ is 
\[
 \mrm{vol}(\Xi  \cdot J_{\gamma_0}(F) \backslash G_{\gamma_0}(\A)) \cdot  \prod_{x \notin T} \mrm{O}_{\gamma_x}(\phi_{\beta_x}) \prod_{x \in  T} \mrm{TO}_{\delta_x \sigma}(\phi_{\alpha_x}). 
\]
Here we normalize Haar measures as in \S \ref{sssec: Haar measures}.
\end{prop}

\begin{proof}
Promoting $(V, \tau, \tau') $ to $(\Cal{E}, t, t')$ amounts to choosing a $\Cal{G}_x \wh{\otimes}_k \ol{k}$-bundle over $\Spec \Cal{O}_x$ for all $x \in X$, plus an $I$-level structure, such that 
\begin{itemize}
\item for $v \notin |T|$, $\Cal{E}_x$ is fixed by $\tau$, 
\item for $v \notin |T|'$, $\Cal{E}_x$ is fixed by $\tau'$, 
\item for $x \in |T|$, the relative position of $\Cal{E}_x$ and $\tau(\Cal{E}_x)$ is given by $\alpha_x$, 
\item for $v \in |T|'$ (hence outside $|T|$), the relative position of $\Cal{E}_v$ and $\tau'(\Cal{E}_v)$ is given by $\beta_v$. 
\end{itemize}
As is well-known (cf. \cite[\S 9]{NgoNgo} or \cite[\S 5]{Ngo13} for the present situation; the earliest reference we know is \cite{Kott80}), this is counted by 
\[
\int_{\Xi \cdot J_{\gamma_0}(F) \backslash \prod_{x \in T} G(F_x \otimes_{\F_q} \F_{q^d}) \times G(\A^T)} \bigotimes_{x \in T} \phi_{\alpha_x}(h_x^{-1} \delta_x \sigma(h_x)) \bigotimes_{x \notin T} \phi_{\beta_x}(h_x^{-1} \gamma_x h_x) \, dh.
\]
Here we use Remark \ref{inner form local} to view $J_{\gamma_0}(F)$ as a subset of $G(F_x)$. The statement of the proposition is a straightforward rewriting of this formula. 
\end{proof}

\subsection{Count of elliptic elements}

Define
\[
\#\Cal{C}(\alpha,\beta; T,T'; d)^{\mrm{ell}} := \sum_{(\Cal{E},t, t')  \text{ elliptic}}  \frac{1}{\#\Aut(\Cal{E},t, t') }.
\]

Combining Proposition \ref{prop: gen fibers}, \S \ref{sssec: aut generic fib}, and Proposition \ref{prop: lattice count}, we obtain:  

\begin{thm}\label{thm: elliptic fixed point count} We have 
\begin{align*}
\#\Cal{C}(\alpha,\beta; T,T'; d)^{\mrm{ell}}  &= \sum_{\substack{ (\gamma_0, (\gamma_x),(\delta_x)) \\ \inv(\gamma_0, (\gamma_x),(\delta_x) ) = 0 \\ \gamma_0 \text{ elliptic}}}
\ker^1(F, G_{\gamma_0}) \cdot \mrm{vol}(\Xi \cdot J_{\gamma_0}(F)  \backslash J_{\gamma_0} (\A_F)) \cdot dg(K)^{-1} \\ 
& \hspace{1.5in} \cdot \left( \prod_{x \notin T}  \mrm{O}_{\gamma_x}(f_{\beta_v}) \right) \prod_{x \in |T|}  \mrm{TO}_{\delta_x \sigma}(f_{\alpha_x}) 
\end{align*}
\end{thm}

\begin{remark}
This is the analogue of \cite[Th\'{e}or\`{e}me 5.1]{Ngo13}. The expressions look almost the same, but one should keep in mind that in our applications  $T = \{x_0\}$, and $f_{\alpha_{x_0}}$ should be thought of as an indicator function on a partial affine flag variety rather than an affine Grassmannian. In addition, the measure of $K_{x_0}$ is adapted accordingly. 
\end{remark}

\begin{cor}\label{cor: number fixed point formula}
Let $\Sht_{\Cal{G}}^{\leq \mu}/a^{\Z}$ be the moduli stack of $\Cal{D}$-shtukas with parahoric level structure at $x_0$, as in \S \ref{subsec: D-shtukas}. Then 
\begin{align*}
\# \Fix(h_{\beta,T'} \circ \Frob, \Sht_{\Cal{G}}^{= \mu}/a^{\Z}|_{x_0})  &= \sum_{\substack{ (\gamma_0, (\gamma_x),(\delta_x)) \\ \inv(\gamma_0, (\gamma_x),(\delta_x) ) = 0}}
\ker^1(F, G_{\gamma_0}) \cdot \mrm{vol}(a^{\Z} \cdot J_{\gamma_0}(F)  \backslash J_{\gamma_0} (\A_F)) \cdot \\ 
& \hspace{1in}\cdot dg(K)^{-1} \cdot \left( \prod_{x \notin T} \mrm{O}_{\gamma_x}(f_{\beta_v})   \right) \cdot \mrm{TO}_{\delta_{x_0} \sigma}(f_{\mu}). 
\end{align*}
\end{cor}

\begin{proof}
This is immediate from Lemma \ref{lem: equivalence of groupoids} and the observation that every non-zero element of a division algebra is elliptic, since the associated group of units is anisotropic mod center. (As mentioned at the beginning of \S \ref{counting fixed points}, the results used from \cite{NgoNgo} and \cite{Ngo13} also apply to $\Cal{D}$-shtukas, and in fact were originally proved for this case in \cite[\S 4]{Ngo06}.) 
\end{proof}

\section{Geometrization of base change for Hecke algebras} \label{hecke base change}

In this section we present a geometric interpretation of the base change homomorphism for spherical Hecke algebras, and then for the center of parahoric Hecke algebras. The results here are a generalization to arbitrary split reductive groups $G$ of results from \cite{Ngo99}, which proved the result for $\GL_n$. 


Using the work of Gaitsgory on realizing central sheaves on the affine flag variety as nearby cycles, we then deduce a geometric interpretation of base change for the center of the parahoric Hecke algebra. 


\subsection{Definition of base change homomorphism}\label{defn base change}

\emph{For this section only, we let $F$ be a local field and $G$ be a reductive  group over $F$}. Given a compact open subgroup $H \subset G(F)$, we have the \emph{Hecke algebra} 
\[
\Cal{H}_{G,H} := \mrm{Fun}_c(H \backslash G(F) /H, \ol{\Q}_{\ell}).
\]
We begin by defining base change homomorphisms for some Hecke algebras with respect to a degree $r$ unramified extension of local fields $E/F$. 

For simplicity \emph{we assume that $G$ is split over $F$}. (Our results should extend at least to quasi-split $G$ without much difficulty.) We let $E/F$ be the unramified extension of degree $r$.

\begin{defn}[{\cite[\S 1]{Haines09}}] Let $K \subset G(F)$ be a hyperspecial maximal compact subgroup. The \emph{base change homomorphism for spherical Hecke algebras} (with respect to $E/F$) is the homomorphism of $\CC$-algebras
\[
\Cal{H}_{G(E), K} \rightarrow \Cal{H}_{G(F),K} 
\]
characterized by the following property. Let $W_F$ be the Weil group of $F$. For an admissible unramified homomorphism $\psi \co W_F \rightarrow {}^L G$ let $\psi' \co W_E \rightarrow {}^L G$ denote the restriction to $W_E \subset W_F$. Let $\pi_{\psi}$ and $\pi_{\psi'}$ denote the corresponding representations of $G(E)$ and $G(F)$ under the Local Langlands Correspondence. Then for any $\phi \in \Cal{H}_{G(E), K}$ we have 
\[
\langle \text{trace } \pi_{\psi'} , \phi \rangle = \langle \text{trace } \pi_{\psi} , b(\phi) \rangle 
\]
\end{defn}

\begin{defn}[{\cite[\S 1]{Haines09}}]
Let $J \subset K$ be a parahoric subgroup and $\mbb{I}_K$ denote the characteristic function of $K \subset G(F)$. By a theorem of Bernstein (cf. \cite[Theorem 3.1.1]{Haines09}), convolution with $\mbb{I}_K$ defines an isomorphism 
\[
- *_J \mbb{I}_K \co Z(\Cal{H}_{G,J}) \xrightarrow{\sim} \Cal{H}_{G,K}.
\]
We define the \emph{base change homomorphism for parahoric Hecke algebras} to be the homomorphism 
\[
b \co \Cal{H}_{G(E), J} \rightarrow \Cal{H}_{G(F),J} 
\]
making the following diagram commute: 
\begin{equation}\label{compatibility with spherical}
\begin{tikzcd}
Z(\Cal{H}_{G(E), J}) \ar[r, "b"] \ar[d, "- *_J \mbb{I}_K", "\sim"']  & Z(\Cal{H}_{G(F), J}) \ar[d, "- *_J \mbb{I}_K", "\sim"'] \\
\Cal{H}_{G(E), K} \ar[r, "b"] & \Cal{H}_{G(F), K}
\end{tikzcd}
\end{equation}

\end{defn}

\subsubsection{Interpretation under Satake isomorphism} 
Let $T \subset G$ be a maximal split torus. We have the Satake isomorphism
\[
\Cal{S}\co \Cal{H}_{G,K}\xrightarrow{\sim} \ol{\Q}_{\ell}[X_*(T)]^W
\]
where $W$ is the Weyl group of $G$ relative to $T$. We also have the Bernstein isomorphism 
\[
B\co Z(\Cal{H}_{G,J})\xrightarrow{\sim} \Cal{H}_{G,K}.
\]

We can define the base change homomorphism on the Satake side as follows. We define the \emph{norm homomorphism} 
\[
N \co \ol{\Q}_{\ell}[X_*(T_E)]^{W} \rightarrow \ol{\Q}_{\ell}[X_*(T_F)]^{W}
\]
to be that induced by the norm $T_E  \rightarrow T_F$. Since we are working in the split setting, this simply corresponds to multiplication by $r$ on $X_*(T_E) \xrightarrow{\sim} \Z^n$. 

Then $b \co \Cal{H}_{G(E), J} \rightarrow \Cal{H}_{G(F),J} $ is determined by the commutativity of the following diagram (\cite[\S 3.2]{Haines09})
\begin{equation}\label{base change satake}
\begin{tikzcd}
Z(\Cal{H}_{G(E), J}) \ar[r, "b"]  \ar[d, "- *_J \mbb{I}_K", "\sim"']  & Z(\Cal{H}_{G(F), J}) \ar[d, "- *_J \mbb{I}_K", "\sim"'] \\
\Cal{H}_{G(E), K} \ar[r, "b"] \ar[d, "S", "\sim"'] & \Cal{H}_{G(F), K} \ar[d, "S", "\sim"']\\
\ol{\Q}_{\ell}[X_*(T_E)]^{W} \ar[r,"N"]	 &  \ol{\Q}_{\ell}[X_*(T_F)	]^{W(F)}
\end{tikzcd}
\end{equation}	

For more on the base change homomorphism, see \cite[\S 3]{Haines09}.


\subsection{Geometrization of the Satake transform}\label{subsec: satake transform}
In this section we will recall a geometric interpretation of the Satake transform.

\subsubsection{The classical Satake transform} We first review the Satake transform \cite{Gro98}. Let $\Cal{H}_{G,K}$ be the spherical Hecke algebra of $G(F)$ with respect to a hyperspecial maximal compact subgroup $K$. The Hecke algebra for $T$ may be identified as $\Cal{H}_T =  \ol{\Q}_{\ell}[X_*(T)]$. 

We choose a Borel subgroup $B$ containing $T$ and let $N$ be its unipotent radical. There is a Satake transform $\Cal{S} \co \Cal{H}_{G,K} \rightarrow \Cal{H}_T$ given by 
\[
\Cal{S}f(t) =  \delta(t)^{1/2}  \int_{N} f(t x) dx,
\]
where the Haar measure $dx$ is normalized to assign volume 1 to $N(\Cal{O}_t)$.

\begin{thm}[Satake]\label{Satake transform} The Satake transform gives a ring isomorphism 
\[
\Cal{S} \co \Cal{H}_{G,K} \cong  \Cal{H}_T^W  \cong   R(\wh{G}) ,
\]
where $R(\wh{G})$ is the representation ring of $\wh{G}$ with $\ol{\Q}_{\ell}$-coefficients. 
\end{thm}

For $\lambda \in X_*(T)$, write  $t^{\lambda}$ for the corresponding element of $X_*(T)$. 
Viewing $\Cal{S}f \in \Cal{H}_T \cong  \ol{\Q}_{\ell}[X_*(T)]$ as functions on $X_*(T)$, we may write 
\begin{equation}\label{eq: sat const term}
\Cal{S}f(t^{\lambda}) =  \delta(t^{\lambda})^{1/2}  \int_{N} f(t^{\lambda} x) dx.
\end{equation}

\subsubsection{Interpretation via semi-infinite orbits} We will interpret the function \eqref{eq: sat const term} geometrically, as the trace function associated to a certain subscheme $S_{\lambda}$, studied by Mirkovic-Vilonen in \cite[\S 3]{MV07}, in the sense that if $f_{\Cal{F}}$ is the trace function associated to $\Cal{F}$, then 
\[
\Cal{S}f_{\Cal{F}}(t^{\lambda}) = q^{-\rho(\lambda)} \Tr(\Frob, R\Gamma_c (S_{\lambda} \otimes_k \ol{k}, \Cal{F}))
\]
where $\rho$ is the usual half sum of the positive roots. 

Following the notation of \cite{MV07}, for $\lambda \in X_*(T)$ we let $L_{\lambda} =  t^{\lambda} G(\Cal{O}) $ denote the image of $\lambda$ in the affine Grassmannian $\Gr_{G}$, and $S_{\lambda}$ be its orbit under $N_F$.

\begin{lemma}\label{geometrize satake transform} Let $\Cal{F}$ be a sheaf on $\Gr_G$ defined over $k$, and $f_{\Cal{F}}$ its associated trace function (cf. \eqref{eq: trace function}). Then we have 
\[
\Cal{S}f_{\Cal{F}}(t^{\lambda}) = q^{-\rho( \lambda)} \Tr(\Frob, R\Gamma_c (S_{\lambda} \otimes_k \ol{k}, \Cal{F})).
\]
\end{lemma}

\begin{proof}We study the rational points $S_{\lambda}(k)$, in preparation for an application of the Grothendieck-Lefschetz trace formula. The stabilizer of $t^{\lambda}$ is 
\begin{align*}
\Stab_{N(F)}(L_{\lambda}) &= \{ n\in N(F) \co nt^{\lambda} G(\Cal{O})  = t^{\lambda} G(\Cal{O})\}   \\
&= \{ n\in N(F) \co t^{-\lambda} n t^{\lambda} \in G(\Cal{O})\}.
\end{align*}
This says that the map $n \mapsto t^{-\lambda} n t^{\lambda}$ defines a bijection between $N(\Cal{O})$ and $\Stab_{N(F)}(L_{\lambda})$. Therefore the map $n \mapsto t^{\lambda} n$ defines a bijection between $S_{\lambda} = [N(F)/\Stab_{N(F)}(L_{\lambda})] \cdot L_{\lambda}$ and $N(F)/N(\Cal{O})$. An application of the Grothendieck-Lefschetz trace formula then yields
\[
\Tr(\Frob, R\Gamma_c (S_{\lambda} \otimes_k \ol{k}, \Cal{F})) =  
\sum_{n \in S_{\lambda}(k)} f_{\Cal{F}}(n) = \sum_{n \in N(F)/N(\Cal{O})} f_{\Cal{F}}(t^{\lambda} \cdot n) 
\]
which is exactly what was claimed upon recalling that $\delta(t^{\lambda}) = q^{\rho(\lambda)}$. 
\end{proof}

\begin{thm}[Mirkovic-Vilonen]\label{thm: mv thm 1}
There is a natural equivalence of functors
\[
H^*(-) \cong \bigoplus_{\lambda \in X_*(T)} H_c^{2\rho(\lambda)}(S_{\lambda}, -) \co P_{G_{\Cal{O}}} (\Gr_G, \ol{\Q}_{\ell}) \rightarrow \mrm{Mod}_{\ol{\Q_{\ell}}}.
\]
\end{thm}

\begin{proof}
In the required generality, this is a consequence of \cite[Theorem 3.16]{HRpre}. For some historical perspective, this is a result of Mirkovic-Vilonen for the affine Grassmannian \emph{over $\CC$} \cite[Theorem 3.6]{MV07}. See also \cite[Theorem 5.3.9]{Zhu15} and the references indicated in \cite[Remark 5.3.10]{Zhu15}.
\end{proof}

\begin{prop}\label{cohomology of convolution}
We have
\[
H^*_c(S_{\lambda}, \Cal{F}_{\mu} * \Cal{F}_{\mu'})   \cong  \bigoplus_{\lambda_1+\lambda_2 = \lambda} H^*_c(S_{\lambda_1}, \Cal{F}_{\mu}) \otimes H^*_c(S_{\lambda_2}, \Cal{F}_{\mu'}) 
\]
\end{prop}

\begin{proof}

In the required generality, this is a consequence of \cite[Theorem 3.16]{HRpre}. For some historical perspective, this statement is proved implicitly in \cite[Proposition 6.4]{MV07} for the affine Grassmannian over $\CC$. In fact, in view of Theorem \ref{thm: mv thm 1} it is equivalent to \cite[Proposition 6.4]{MV07}. It is formulated in the general setting in \cite[Proposition 5.3.14]{Zhu15}, and a proof is sketched there.

\end{proof}

\subsection{Base change for spherical Hecke algebras}\label{subsec: BC spherical HA}

For $r \in \N$, let $k_r$ be the (unique) extension of $k$ of degree $r$ and $F_r$ denote the unique unramified field extension of $F$ of degree $r$. For each $\mu \in X_*(T)$, we let $\Sat_{\Gr_G, r}(\mu)$ be the associated perverse sheaf on $\Gr_{G,k_r}$ and we let $\psi_{r, \mu}$ be the trace function on $\Gr_{G}(k_r)$ associated to $\Sat_{\Gr_G, r}(\mu)$. We will give a geometric interpretation for $b(\psi_{r, \mu})$. 

\subsubsection{Weil restriction} It will be useful to adopt a different perspective on the Hecke algebra $\Cal{H}_{G(F_r), K_r}$, where $K_r$ denotes the maximal compact subgroup of $G(F_r)$ corresponding to our chosen hyperspecial vertex. As is usual, we can think of $\Cal{H}_{G(F_r), K_r}$ as functions on $\Gr_{G}(k_r) \rightarrow \ol{\Q}_{\ell}$ which are invariant with respect to the left $K_r$-action. However, using the identification 
\[
\Gr_{G}(k_r) = ( \Res_{k_r/k} \Gr_{G, k_r}) (k)
\]
we can instead consider $\Cal{H}_{G(F_r), K_r}$ as functions on $ ( \Res_{k_r/k} \Gr_{G, k_r}) (k)$. Let $\tau$ be the cyclic permutation on $\Gr_{G}^r$ given by 
\[
\tau(y_1, \ldots, y_r) = (y_r, y_1, \ldots, y_{r-1}).
\]
By the definition of Weil restriction, we have a canonical bijection
\[
\Gr_{G}(k_r) \xrightarrow{\sim} \Fix(\Frob \circ \tau, \Gr_{G}^r) \quad \text{sending} \quad 
y \mapsto  (y, \Frob(y), \ldots, \Frob^{r-1}(y)).
\]

\begin{defn}For each $\mu \in X_*(T)$, we let $\Cal{F}_{\mu} := \Sat_{\Gr_G}(\mu)$ be the perverse sheaf on $\Gr_{G}$. Consider the perverse sheaf 
\[
\Cal{F}_{\mu}^{(1)} \boxtimes \ldots \boxtimes \Cal{F}_{\mu}^{(r)} \in D_c^b(\Gr_{G}^r)
\]
where $\Cal{F}_{\mu}^{(i)} \cong \Cal{F}_{\mu}$; the superscripts are just labellings. The endomorphism $\tau$ on $\Gr_{G}^r$ lifts to an endomorphism $\wt{\tau}$ of $\Cal{F}_{\mu}^{(1)} \boxtimes \ldots \boxtimes \Cal{F}_{\mu}^{(r)}$ in an obvious way. Define a function
\[
\zeta_{r,\mu} \co \Fix(\Frob \circ \tau, \Gr_{G}^r) \rightarrow \ol{\Q}_{\ell}
\]
by 
\[
\zeta_{r,\mu}(y)   := \Tr(\Frob \circ \wt{\tau}, (\Cal{F}_{\mu}^{(1)} \boxtimes \ldots \boxtimes \Cal{F}_{\mu}^{(r)})_{\ol{y}}).
\] 
\end{defn}

\begin{prop}
The $\zeta_{r,\mu}$ form a basis for $\Cal{H}_{G(F_r),K_r}$. 
\end{prop} 

\begin{proof}
This is well-known. It amounts to the fact that the change-of-basis matrix between the standard (double coset) basis of the Hecke algebra and the basis consisting of the $\zeta_{r,\mu}$ is upper-triangular for the Bruhat order. 
\end{proof}

\subsubsection{Convolution product} Recall that there is a convolution product $*$ on $\mrm{Perv}_{G(\Cal{O})} (\Gr_G)$ \cite[\S 5]{Zhu15}.

\begin{defn}\label{def: sph conv product} Consider $r$th convolution product 
\[
\Cal{F}_{\mu}^{*r} := \Cal{F}_{\mu}^{(1)} *  \ldots * \Cal{F}_{\mu}^{(r)} 
\]
as a perverse sheaf on $\Gr_{G}$. This has an automorphism $\kappa'$ of order $r$ given by the composition	
\[
 \Cal{F}_{\mu}^{(1)} *  \ldots * \Cal{F}_{\mu}^{(r)}  \xrightarrow{\kappa} \Cal{F}_{\mu}^{(r)}  *  \Cal{F}_{\mu}^{(1)} *  \ldots * \Cal{F}_{\mu}^{(r-1)}   \xrightarrow{\iota} \Cal{F}_{\mu}^{(1)} *  \ldots * \Cal{F}_{\mu}^{(r)}
\]
where 
\begin{itemize}
\item $\kappa$ is the cyclic permutation obtained from the commutativity constraint $\Cal{F} * \Cal{B} \cong \Cal{B} * \Cal{F}$ of Geometric Satake  \cite[Proposition 5.2.6]{Zhu15}, and 
\item $\iota$ is the tautological isomorphism $\Cal{F}_{\mu}^{(i\text{ mod }r) } \cong \Cal{F}_{\mu}^{(i+1\text{ mod }r ) 	}$ coming from the fact that all the $\Cal{F}_{\mu}^{(i)}$ are defined to be the same perverse sheaf. 
\end{itemize}
 
Define $\phi_{r, \mu} \co \Fix(\Frob , \Gr_g(\ol{k})) \rightarrow \ol{\Q}_{\ell} $ by
\[
\phi_{r,\mu}(y) =  \Tr(\Frob \circ \kappa', (\Cal{F}_{\mu}^{(1)} * \ldots * \Cal{F}_{\mu}^{(r)})_y).
\]
\end{defn}

\subsubsection{The base change identity} We now explain the relationship between these functions and the base change homomorphism.

\begin{prop}\label{prop1}
We have $b(\zeta_{r,\mu}) =  \phi_{r, \mu}$. 
\end{prop}

\begin{proof}
By Theorem \ref{Satake transform} it suffices to equate the Satake transforms of both sides. In other words, we must prove the following identity: 
\[
\int_{N(F)} b(\zeta_{r,\mu})(t^{\lambda} x) \, dx = \int_{N(F)} \phi_{r,\mu}(t^{\lambda} x) \, dx \text{ for all } \lambda \in X_*(T).
\]
By \eqref{base change satake}, this is equivalent to establishing the two equations 
\begin{equation}\label{eq: bc required eq 1}
\int_{N(F_{r})} \zeta_{r,\mu}(t^{\lambda} x_r) \, dx_r = \int_{N(F)} \phi_{r,\mu}(t^{r \lambda} x) \, dx
\end{equation}
\begin{equation}\label{eq: bc required eq 2}
\int_{N(F)} \phi_{r,\mu}(t^{\lambda} x) \, dx = 0 \quad \text{
 if $ \lambda \notin r \cdot X_*(T)$ }.
\end{equation}

To do this we use the Lefschetz trace formula\footnote{Note that this is applicable because $\Frob \circ \tau$ is the Frobenius for a twisted form of $Y$: see the discussion beginning in the last sentence of \cite[p. 651]{Ngo99}.}:
\begin{equation}\label{LTF sigma}
\sum_{y \in \Fix(\Frob \circ \tau, Y(\ol{k}))} \Tr(\Frob \circ \wt{\tau}, \Cal{F}_y) = \Tr(\Frob \circ \wt{\tau}, R  \Gamma_c(Y \otimes_k \ol{k}, \Cal{F})).
\end{equation}
\begin{equation}\label{LTF kappa}
\sum_{y \in \Fix(\Frob \circ \kappa', Y(\ol{k}))} \Tr(\Frob \circ \kappa', \Cal{F}_y) = \Tr(\Frob \circ \kappa', R  \Gamma_c(Y \otimes_k \ol{k}, \Cal{F})).
\end{equation}


Applying \eqref{LTF sigma} to $Y:=S_{\lambda} \times \ldots \times S_{\lambda}$ and $ \Cal{F} := \Cal{F}_{\mu}^{(1)} \boxtimes \ldots \boxtimes \Cal{F}_{\mu}^{(r)} $ and using Lemma \ref{geometrize satake transform} plus the K\"{u}nneth theorem gives
\[
\int_{N(F_{r})} \zeta_{r,\mu}(t^{\lambda} x_r) dx_r  = q^{-\rho(\lambda)}\Tr(\Frob \circ \wt{\tau}, R\Gamma_c(S_{\lambda} \otimes_k \ol{k}, \Cal{F}_{\mu})^{\otimes  r}).
\]
We note that here $\wt{\tau}$ acts by cyclically permuting the tensor factors $R\Gamma_c(S_{\lambda} \otimes_k \ol{k}, \Cal{F}_{\mu})^{\otimes  r}$, and $\Frob$ acts factorwise. 

Applying \eqref{LTF kappa} to $Y=S_{\lambda}$ and $\Cal{F} = \Cal{F}_{\mu}^{*r} = \Cal{F}_{\mu}^{(1)} *  \ldots * \Cal{F}_{\mu}^{(r)} $ and using Lemma \ref{geometrize satake transform} gives 
\begin{equation}\label{kappa  coh}
\int_{N(F)} \phi_{r,\mu}(t^{\lambda} x) dx  = q^{-\rho(\lambda)} \Tr(\Frob \circ \kappa', R\Gamma_c(S_{\lambda} \otimes_k \ol{k},\Cal{F}_{\mu}^{*r})).
\end{equation}
We first digest the expression \eqref{kappa  coh}. By Proposition \ref{cohomology of convolution} we have 
\[
R\Gamma_c(S_{\lambda} \otimes_k \ol{k}, \Cal{F}_{\mu}^{(1)}* \ldots * \Cal{F}_{\mu}^{(r)}) \cong \bigoplus_{\lambda_1 + \ldots + \lambda_r = \lambda} \bigotimes_{i=1}^r R\Gamma_c(S_{\lambda_i} \otimes_k \ol{k}, \Cal{F}_{\mu}).
\]
Let's try to understand the action of $\kappa'$, which is a composition $\kappa' = \iota \circ \kappa$. The map $\kappa$ acts by cyclic permutation of both the spaces and sheaves, so it induces the permutation 
\begin{align*}
\kappa^* \co & H^*_c(S_{\lambda_1} \otimes_k \ol{k}, \Cal{F}_{\mu}^{(1)}) \otimes H^*_c(S_{\lambda_2} \otimes_k \ol{k}, \Cal{F}_{\mu}^{(2)}) \otimes\ldots \otimes  R\Gamma_c(S_{\lambda_r} \otimes_k \ol{k}, \Cal{F}_{\mu}^{(r)})  \\
& \rightarrow H^*_c(S_{\lambda_r} \otimes_k \ol{k}, \Cal{F}_{\mu}^{(r)}) \otimes H^*_c(S_{\lambda_1} \otimes_k \ol{k}, \Cal{F}_{\mu}^{(1)})  \otimes \ldots \otimes  R\Gamma_c(S_{\lambda_{r-1}} \otimes_k \ol{k}, \Cal{F}_{\mu}^{(r-1)}).
\end{align*}
Next, the map $\iota$ relabels the sheaves only, so the conclusion is that $\kappa'$ induces 
\begin{align*}
(\kappa')_* \co & H^*_c(S_{\lambda_1} \otimes_k \ol{k}, \Cal{F}_{\mu}^{(1)}) \otimes H^*_c(S_{\lambda_2} \otimes_k \ol{k}, \Cal{F}_{\mu}^{(2)}) \otimes\ldots \otimes  R\Gamma_c(S_{\lambda_r} \otimes_k \ol{k}, \Cal{F}_{\mu}^{(r)})  \\
& \rightarrow H^*_c(S_{\lambda_r} \otimes_k \ol{k}, \Cal{F}_{\mu}^{(1)}) \otimes H^*_c(S_{\lambda_1} \otimes_k \ol{k}, \Cal{F}^{(2)})  \otimes \ldots \otimes  R\Gamma_c(S_{\lambda_{r-1}} \otimes_k \ol{k}, \Cal{F}_{\mu}^{(r	)}).
\end{align*}
In particular, we emphasize that the composition, at the level of cohomology, effects a permutation of the \emph{spaces}. Now, Frobenius preserves each tensor and summand. Therefore, $\Frob \circ \kappa'$ acts on the summands of the form 
\[
\bigoplus_{j=0}^{r-1} \bigotimes_{i=1}^r H_c^*(S_{\lambda_{i+j \pmod{r}}} \otimes_k \ol{k}, \Cal{F}_{\mu})
\]
by a cyclic permutation followed by a factorwise endomorphism. From this form we see that if $\lambda_1,\ldots, \lambda_r$ are not all equal then $\Frob \circ \kappa'$ permutes the summands freely, so $\Frob \circ \kappa'$ has trace $0$. In particular, if $\lambda$ is not divisible by $r$ then the $\lambda_i$ cannot be all equal, so that the trace is $0$. This establishes  \eqref{eq: bc required eq 2}.

On the other hand, our analysis above implies that if $\lambda = r\lambda'$ then the contribution to $\Tr(\Frob \circ \kappa')$ all comes from the terms with all $\lambda_1 = \ldots =  \lambda_r = \lambda'$, and we have 
\[
\Frob \circ \kappa'|_{H_c^*(S_{\lambda'} \otimes_k \ol{k}, \Cal{F}_{\mu})^{\otimes r} }
= \Frob \circ \wt{\tau}|_{H_c^*(S_{\lambda'} \otimes_k \ol{k}, \Cal{F}_{\mu})^{\otimes r}}
\]
so that
\[
\int_{N(F_{r,t})} \psi_{r,\mu}(t^{\lambda'} x_r ) dx_r = \int_{N(F_t)} \phi_{r,\mu}(t^{\lambda} x) \,dx
\]
which establishes \eqref{eq: bc required eq 1}.
\end{proof}

\begin{lemma}\label{lemma2}
We have $\zeta_{r,\mu} = \psi_{r,\mu}$. 
\end{lemma}

\begin{proof} By Theorem \ref{Satake transform} it suffices to check equality of the Satake transforms of both sides. Using Lemma \ref{geometrize satake transform}, this amounts to showing
\[
\Tr(\Frob^r, R \Gamma_c(S_{\lambda} \otimes_k \ol{k}, \Cal{F}_{\mu})) = \Tr(\Frob \circ \wt{\tau}, R\Gamma_c(S_{\lambda} \otimes_k \ol{k}, \Cal{F}_{\mu})^{\otimes r}).
\]
This follows from the general linear algebra fact asserted in Lemma \ref{Saito-Shintani} below.
\end{proof}	

\begin{lemma}[Saito-Shintani\footnote{We first found this explicitly stated, without proof, in \cite{Ngo06}, where it is said to be implicit in the work of Saito and Shintani on base change.}]\label{Saito-Shintani} Let $V$ be a finite-dimensional representation over a field $k$. Let $\tau$ be the endomorphism of $V^{\otimes r}$ defined by 
\[
v_1 \otimes \ldots \otimes v_r \mapsto v_r \otimes v_1 \otimes \ldots \otimes v_{r-1}.
\]
Then for all $T_1, \ldots, T_r \in \End(V)$, we have 
\[
\Tr(T_1 \ldots T_r , V) = \Tr( (T_1 \otimes \ldots \otimes T_r ) \tau , V^{\otimes r}).
\]

\end{lemma}

\begin{proof}
 If $\{e_i\}$ is a basis for $V$, then 
\[
(T_1 \otimes  \ldots \otimes T_r) \tau  \cdot (e_{i_1}  \otimes \ldots \otimes e_{i_r}) = T_1(e_{i_r}) \otimes \ldots\otimes  T_r(e_{i_{r-1}}).
\]
We expand out both sides of the desired equality:  
\begin{align*}
\Tr( (T_1 \otimes  \ldots \otimes T_r) \tau )  &= \sum_{i_1, \ldots, i_r} \langle e_{i_1}\otimes \ldots \otimes e_{i_r}, T_1(e_{i_r}) \otimes \ldots \otimes T_r(e_{i_{r-1}} ) \rangle \\
&= \sum_{i_1 \ldots, i_r} \langle e_{i_1}, T_1(e_{i_r})  \rangle  \ldots \langle e_{i_r}, T_r(e_{i_{r-1}}) \rangle
\end{align*}
and 
\[
\Tr(T_1 \ldots T_r) = \sum_i \langle e_i, T_1 \cdots T_r e_i \rangle.
\]
Thus want to show that 
\[
\sum_{i_1, \ldots, i_r} \langle e_{i_1}, T_1(e_{i_r})  \rangle  \ldots \langle e_{i_r}, T_r(e_{i_{r-1}}) \rangle =  \sum_i \langle e_i, T_1 \cdots T_r e_i \rangle.
\]
This follows by repeated iteration of the following more general identity. 
\end{proof}

\begin{lemma}
Let $V$ be a finite-dimensional vector space with basis $\{e_i\}$. For any $T \in \End(V)$, we have 
\begin{equation}\label{general identity}
\sum_j \langle x, Te_j \rangle \langle e_j, y \rangle = \langle x, Ty \rangle.
\end{equation}
\end{lemma}

\begin{proof}
It suffices to establish the equation for $y$ ranging over a basis of $V$; taking $y=e_i$ the left hand side is $\langle x, Te_i\rangle$, and so is the right hand side. 
\end{proof}

Combining these results, we obtain the main formula of interest: 

\begin{thm}\label{spherical BC} 
We have $b(\psi_{r,\mu})  =\phi_{r, \mu}$.
\end{thm}

\begin{proof}
This follows immediately upon combining Proposition \ref{prop1} and Lemma \ref{lemma2}.
\end{proof}

\subsection{Base change for the centers of parahoric Hecke algebras} 

We now establish an identity for central functions of parahoric Hecke algebras analogous to Theorem \ref{spherical BC}. This is based on a degeneration from the spherical case. 

\subsubsection{Setup}\label{sssec: BC center parahoric setup} We first set some notation. Pick a smooth global curve $X/\F_q$ (not necessarily projective) with a rational point $x_0 \in X(\F_q)$. (The reader may imagine that $X, x_0$ are as previously fixed, but this discussion applies more generally.)  Let $\Cal{G} \rightarrow X$ be a parahoric group scheme, such that $\Cal{G}|_{X-x_0} \cong G \times X$ and $\Cal{G}(\Cal{O}_{x_0}) = J$ is a parahoric subgroup. We form the affine Grassmannian 
\[
\pi \co \Gr_{\Cal{G}} \rightarrow X.
\]
Note that for $x \in X-x_0$ we have 
\[
\Gr_{\Cal{G}}|_x \cong \Gr_G \times_k k(x).
\]

For each $\mu \in X_*(T)$, we let $\Cal{F}_{\mu} := \Sat_{\Gr_{\Cal{G}}}(\mu)$ be the (shifted) perverse sheaf on $\Gr_{\Cal{G}}$ and $\Cal{F}_{\mu, x_0}$ be the restriction to $\Gr_{\Cal{G}, x_0}$. (We have normalized our shifts so that $\Cal{F}_{\mu, x_0}$ is perverse.) We let $\psi_{r,\mu}' \in \Cal{H}_{G(F_{x_0} \otimes_{\F_q} \F_{q^r}),J}$ be the function as in Definition \ref{IC basis hecke algebra}.

\subsubsection{Convolution product}By Theorem \ref{gaitsgory central}, $R\Psi(\Cal{F}_{\mu}) := R\Psi_{x_0}(\Cal{F}_{\mu})$ is a central sheaf on $\Gr_{\Cal{G}, x_0}$. We therefore have, as in Definition \ref{def: sph conv product}, an automorphism  
\begin{align*}
\kappa' \co  R\Psi(\Cal{F}_{\mu})^{(1)} * \ldots * R\Psi(\Cal{F}_{\mu})^{(r)}   & \xrightarrow{\kappa} R\Psi(\Cal{F}_{\mu})^{(r)} * R\Psi(\Cal{F}_{\mu})^{(1)} * \ldots * R\Psi(\Cal{F}_{\mu})^{(r-1)} \\
& \xrightarrow{\iota} R\Psi(\Cal{F}_{\mu})^{(1)} * \ldots * R\Psi(\Cal{F}_{\mu})^{(r)}.
\end{align*}

\begin{defn}

Let 
\[
R\Psi(\Cal{F}_{\mu})^{*r} := \underbrace{R\Psi(\Cal{F}_{\mu})^{(1)} * \ldots * R\Psi(\Cal{F}_{\mu})^{(r)}   }_{r \text{ times}}.
\]
Define $\phi_{\mu}' \co \Fix(\Frob, \Gr_{\Cal{G}, x_0}(\ol{k})) \rightarrow \ol{\Q}_{\ell} $ by
\[
\phi_{\mu}'(y) =  \Tr(\Frob \circ \kappa', (R\Psi(\Cal{F}_{\mu})^{*r})_y).
\]
\end{defn}

\begin{thm}\label{parahoric BC} We have $b(\psi_{r,\mu}') = \phi_{\mu}'$.
\end{thm}

\begin{proof}
By Theorem \ref{gaitsgory central}, $\phi_{r, \mu}'(y)$ is in the center of the Iwahori-Hecke algebra. (Since $R\Psi(\Cal{F}_{\mu})$ is central by Theorem \ref{gaitsgory central}, $\phi_{r, \mu}'(y)$ clearly commutes with all the other functions of the form  $\phi_{r, \nu}'(y)$; then use that such things form a basis for the Iwahori Hecke algebra as $\nu$ runs over the extended affine Weyl group.)

Now the argument is essentially the same as for Corollary \ref{gaitsgory nearby function}. Consider the map 
\[
\mrm{pr} \co \Gr_{\Cal{G}}^{\leq \mu} \rightarrow  \Gr_{G \times X}^{\leq \mu}
\]
induced by forgetting the level structure at $x_0$. Since $\mrm{pr}$ is proper, by Lemma \ref{nearby cycles proper} and the fact that $\mrm{pr}$ is an isomorphism away from $x_0$ we have
\[
\mrm{pr}_! R\Psi^{\Gr_{\Cal{G}}}_{x_0}(\Cal{F}_{\mu}) = R\Psi^{\Gr_{G \times X}}_{x_0}(\mrm{Sat}_{\Gr_{G \times X}}(\mu)).	
\]
By Lemma \ref{nearby cycles smooth} and the fact that $\Gr_{G \times X} \rightarrow X$ is smooth, we have 
\[
R\Psi^{\Gr_{G \times X}}_{x_0}(\mrm{Sat}_{\Gr_{G \times X}}(\mu))\cong \mrm{Sat}_{\Gr_{G \times X}}(\mu).
\]

Since $\mrm{pr}_!$ corresponds to $- *_J \mbb{I}_K$ at the level functions, this implies
\[
 \phi_{\mu}' *_J \mbb{I}_K = \phi_{\mu}.
\]
Thus by Theorem \ref{spherical BC} and \eqref{compatibility with spherical}, we have that
\begin{equation}\label{eq: parahoric bernstein equal}
 \phi_{\mu}' *_J \mbb{I}_K = \phi_{\mu} = b(\psi_{r,\mu}) = b(\psi_{r,\mu}') *_J \mbb{I}_K.
\end{equation}
In view of the Bernstein isomorphism (Theorem \ref{thm: bernstein isom}), the fact that $ \phi_{\mu}' $ and $b(\psi_{r,\mu}') $ are central plus \eqref{eq: parahoric bernstein equal} implies that they are equal.  
\end{proof}

\subsubsection{A global reformulation}\label{sssec: bc global reformulation} We now recast Theorem \ref{parahoric BC} into a form that will be more suitable for our eventual needs. 

Let $\Cal{G}$ and $\Gr_{\Cal{G}}$ be as in \S \ref{sssec: BC center parahoric setup}. We first recall a construction of the convolution product $\Cal{F}_{\mu} *\Cal{F}_{\mu'}$. Recall the iterated global affine Grassmannian $\Gr_{\Cal{G}, X^2}$ from \S \ref{sssec: iterated aff gr}. We can form the twisted tensor product $\Cal{F}_{\mu} \wt{\boxtimes} \Cal{F}_{\mu'}  := \Sat_{\Gr_{\Cal{G}, X^2}}(\mu, \mu')$ on $\Gr_{\Cal{G}, X^2}$ \cite[\S A]{Zhu15}, which is supported on the Schubert variety $\Gr_{\Cal{G},X}^{\leq (\mu, \mu')}$.

Restricting to the diagonal $X \subset X^2$, we have the multiplication map 
 \[
m \co \Gr_{\Cal{G},X^2}^{\leq (\mu, \mu')}|_{\Delta} \rightarrow \Gr_{\Cal{G}}^{\leq \mu+\mu'}
 \]
 defined on points by  
 \[
 (x,x,\Cal{E}_1 \stackrel{\varphi}\dashrightarrow \Cal{E}_2  \stackrel{\beta}\dashrightarrow\Cal{E}^0) \mapsto (x,\Cal{E}_1 \stackrel{\beta  \circ \varphi} \dashrightarrow \Cal{E}^0).
 \]
 Then the convolution product is defined by (cf. \cite[\S 4]{MV07} or \cite[\S 5.1]{Zhu15}) 
 \begin{equation}\label{eq: conv product}
	 \Cal{F}_{\mu} *\Cal{F}_{\mu'} : = Rm_! (\Cal{F}_{\mu} \wt{\boxtimes} \Cal{F}_{\mu'} ) \in \mrm{Perv}_{L^+ \Cal{G}}(\Gr_{\Cal{G}}^{\leq \mu+ \mu'}).
 \end{equation}
 
Let us now write down our particular situation of interest. Consider the diagram
\[
\begin{tikzcd} 
\Gr_{\Cal{G}, X^r}^{\leq (\mu_1, \ldots, \mu_r)}|_{\Delta} \ar[r] \ar[d,  "m"]	 &  \Delta(X) \subset X^r  \ar[d] \\
\Gr_{\Cal{G}, X}^{\leq \mu_1 + \ldots + \mu_r} \ar[r]  & X 
\end{tikzcd}
\]
Then by \eqref{eq: conv product} we have
\[
\Cal{F}(\mu_1) *  \ldots  * \Cal{F}(\mu_r) = Rm_!(\Sat_{\Gr_{\Cal{G}}, X^r}(\mu_1, \ldots, \mu_r)) 
\] 
Now Theorem \ref{parahoric BC} can be reformulated as follows, using Corollary \ref{nearby cycles proper} to commute $Rm_!$ and nearby cycles. 

\begin{prop}\label{base change formula}
Let $r.\mu := (\mu, \ldots, \mu) \in X_*(T)_+^r$. Let $f_{\nu,x_0} \in \Cal{H}_{G(F_{x_0}),J}$ be the function $f_{\nu}$ viewed in the parahoric Hecke algebra of $F_{x_0}$, and define $\psi_{r\mu,x_0}' \in Z(\Cal{H}_{G(F_{x_0}),J})$ similarly. Then we have
\[
b(\psi_{r,r\mu,x_0}') = \sum_{\nu \leq r \mu} \Tr(\sigma \circ \kappa', R\Psi^{X}_{x_0} (Rm_! \Sat_{\Gr_{\Cal{G}}^{\leq r.\mu}, X^r}(r.\mu))_{\nu}) f_{\nu,x_0}.
\]
\end{prop}

\section{Comparison of two moduli problems}\label{moduli  problems}

\subsection{Setup}\label{subsec: A and B setup}
We now let $G$ be the group scheme of units of a global algebra $\Cal{D}$ as in \S \ref{subsec: D-shtukas} and $\Cal{G}$ a parahoric group scheme corresponding to some choice of level structure at $x_0$, so $\Sht_{\Cal{G}}$ are the $\Cal{D}$-shtukas studied in \S \ref{subsec: D-shtukas}. We continue to assume, that $\Cal{G}$ is reductive away from $x_0$ (though we could avoid this simply by shrinking $X^{\circ}$ to remove the points where $\Cal{G}$ is not reductive). 

Let $Z \subset X$ be the set of places of ramification for $\Cal{D}$. \emph{We assume throughout that $\# Z  \geq n^2 (||\mu_1|| + \ldots + || \mu_r||)$, so as to be able to apply Proposition \ref{proper}.}

Let $\cX := (X-Z-\{x_0\})$. We will now define and compare two different moduli stacks of shtukas.

\subsection{Situation A}Let 
\[
\Sht_A^{\mu} := (\Sht_{\Cal{G}, X}^{\leq \mu}/a^{\Z})^r.
\]
We have a map 
\[
\pi_A \co \Sht_A^{\mu}|_{(X-Z)^r}  \rightarrow (X-Z)^r.
\]
By Proposition \ref{proper} the restriction $\pi_A^{\circ} := \pi_A|_{(\cX)^r}$ is proper.

\begin{defn}Let $r.\mu = \underbrace{(\mu, \ldots, \mu)}_{r \text{ times}}$. We define $\Cal{A}^{\mu}_r \in D_c^b((\cX)^r) $ as follows: 
\[
\Cal{A}^{\mu}_r := R\pi_{A*}^{\circ} ( \mrm{Sat}_{\Sht_A^{\mu} }(r.\mu)).
\]
\end{defn}

We have the following easy but crucial property. 

\begin{prop}\label{A is local system}
The complex $\Cal{A}^{\mu}_r \in D_c^b((\cX)^r) $ is locally constant on $(\cX)^r$, in the sense that each $R^i\pi_{A*}^{\circ} ( \mrm{Sat}_{\Sht_A^{\mu}}(r.\mu))$ is a local system. 
\end{prop}

\begin{proof}
By the properness of $\pi_{A*}^{\circ}$ we know that $R^i\pi_{A*}^{\circ}  (  \mrm{Sat}_{\Sht_A^{\mu}}(r.\mu))$ is constructible, and the local acyclicity from Proposition \ref{ULA} then implies that it is locally constant.
\end{proof}

Note that by the K\"{u}nneth formula, we have 
\[
\Cal{A}_r^{\mu} \cong (\Cal{A}_1^{\mu})^{\boxtimes r}
\]
Choose a basepoint $x \in  \cX$, and let $x^r \in (\cX)^r$ denote the diagonal point $(x,\ldots, x)$. Then the symmetric group $S_r$ acts on $(\cX)^r$, hence also $\pi_1((\cX)^r,x^r)$. This lifts to an $S_r$-equivariant structure on the local system $\Cal{A}_r^{\mu}$, i.e. an action of $\pi_1((\cX)^r, x^r) \rtimes S_r$ on $(\Cal{A}_r^{\mu})_{x^r}$, commuting with the action of the global Hecke algebra $\Cal{H}^{\otimes  r}$. 

\subsection{Situation B}
Let 
\[
\Sht_B^{\mu} := \Sht_{\Cal{G}, X^r}^{\leq r.\mu}/a^{\Z}
\]
where $r.\mu = \underbrace{(\mu, \ldots, \mu)}_{r \text{ times}}$. We have a map
\[
\pi_B \co \Sht_B^{\mu}|_{(X-Z)^r}  \rightarrow (X-Z)^r.
\]
By the assumption that $\Cal{D}$ is totally ramified at sufficiently many places, the map $\pi_B^{\circ} := \pi_B|_{(\cX)^r}$ is proper by Proposition \ref{proper}.

\begin{defn} We define $\Cal{B}^{\mu}_r \in D_c^b((\cX)^r) $ as follows: 
\[
\Cal{B}^{\mu}_r := R\pi_{B*}^{\circ} ( \mrm{Sat}_{\Sht_B^{\mu}}(r.\mu)).
\]
\end{defn}

\begin{prop}\label{B is local system}
The complex $\Cal{B}^{\mu}_r \in D_c^b((\cX)^r)$ is locally constant on $(\cX)^r$, in the sense that each $R^i\pi_{B*}^{\circ}  ( \mrm{Sat}_{\Sht_B^{\mu}}(r.\mu))$ is a local system. 
\end{prop}

\begin{proof}
The proof is the same as for Proposition \ref{A is local system}.
\end{proof}

Again, we have commuting actions of the Hecke algebra $\Cal{H}$ and $\pi_1((\cX)^r, x^r) \rtimes S_r $ on $\Cal{B}^{\mu}_r$.

\subsection{The comparison theorem}\label{subsec: comparison}

\begin{thm}[\cite{Ngo06}]\label{ngo main}
Let $\tau \in S_r$ be an $r$-cycle, i.e. $\langle \tau \rangle \cong \Z/r\Z$. For $g \in \pi_1((\cX)^r, x^r)$ and $h \in \Cal{H}$ we have 
\[
\Tr(( h\otimes 1 \ldots \otimes 1 ) g \tau , (\Cal{A}^{\mu}_r)_{x^r}) = \Tr(h g \tau, (\Cal{B}^{\mu}_r)_{x^r}).
\]
\end{thm}

\begin{proof}
This is \cite[\S 3.3 Theorem 1]{Ngo06}. Since this is really crucial for us, we outline for the sake of exposition how the proof goes. Keep in mind that $\Cal{A}^{\mu}_r$ and $\Cal{B}^{\mu}_r$ are both local systems. 

By an application of the Cebotarev density theorem, it suffices to prove the equality for $g = \Frob_{(x_1, \ldots, x_r)}$ for a dense open subset of $X^r$, and in particular on the locus $(x_1, \ldots, x_r)$ where the $x_i$ are pairwise distinct. On this locus (and under a certain further restriction on the points $(x_1, \ldots, x_r)$) Ng\^{o} independently computes both sides of the equation, following the Langlands-Kottwitz paradigm, and verifies that they are equal by direct comparison. 

Let us say a little more about this computation, which is carried out in \cite[\S 5]{Ngo06}. Using the Grothendieck-Lefschetz trace formula to re-express both sides, there are two main inputs: (1) a count of fixed points, and (2) a computation of the trace of Frobenius on the stalks of the relevant sheaves. The counting step is done as in \S \ref{counting fixed points}, and the analysis of the stalks enters via results as in \S \ref{subsec: BC spherical HA}. The interesting feature is that the pairwise distinctness of the points $(x_1, \ldots, x_r)$, plus the extra restriction that we have omitted, turns out to imply that the point counting formulas involve \emph{no twisted orbital integrals}. Therefore, no fundamental lemma is required to prove the desired equality.

\end{proof}

\begin{remark}For a heuristic that underlies the theorem, coming from a conjectural description of the cohomology of shtukas, see \cite[\S 2.2, 3.3]{Ngo06}. The punchline is that after admitting this conjectural description, the identity in Theorem \ref{ngo main} reduces to Lemma \ref{Saito-Shintani}.
\end{remark}

\begin{defn}\label{def: A and B sheavs}
Let 
\[
R\Psi_{x_0^r}(\Cal{A}_r^{\mu}) := R\pi_{A!} R\Psi_{x_0^r}( \mrm{Sat}_{\Sht_A^{\mu} }(r.\mu)|_{\Delta(\cX)})  \in D_c^b(x_0)
\]
be the cohomology of nearby cycles at $x_0^r \in \Delta(X-Z)$, and let 
\[
R\Psi_{x_0^r}(\Cal{B}_r^{\mu})  := R\pi_{B!} R\Psi_{x_0^r}(\mrm{Sat}_{\Sht_B^{\mu} }(r.\mu)|_{\Delta(\cX)})  \in D_c^b(x_0)
\]
be the cohomology of nearby cycles at $x_0^r \in \Delta(X-Z)$.\footnote{This is a small abuse of the notation used in \S \ref{subsec: nearby cycles}.} Again we apply Remark \ref{rem: nearby cycles descent} to equip $R\Psi_{x_0^r}(\Cal{A}_r^{\mu})$ and $R\Psi_{x_0^r}(\Cal{B}_r^{\mu})$ with $\F_q$-structures.

Thanks to Proposition \ref{A is local system} and Proposition \ref{B is local system}, the complex $R\Psi_{x_0^r}(\Cal{A}_r^{\mu})$ is equipped with commuting actions of $\pi_1((\cX)^r, x^r) \rtimes S_r$ and $(\Cal{H})^{\otimes r}$, while $R\Psi_{x_0^r}(\Cal{B}_r^{\mu})$ is equipped with commuting actions of $\pi_1((\cX)^r, x^r) \rtimes S_r$ and $\Cal{H}$. 

\end{defn}

\begin{cor}\label{equality traces}
For $g \in \pi_1((\cX)^r, x^r)$ and $h \in \Cal{H}$ we have 
\begin{equation}\label{eq: trace A and B}
\Tr(( h\otimes 1 \ldots \otimes 1 ) \circ \Frob \circ \tau , R\Psi_{x_0^r}(\Cal{A}_r^{\mu})) = \Tr(h \circ \Frob  \circ \tau, R\Psi_{x_0^r}(\Cal{B}_r^{\mu})).
\end{equation}
\end{cor} 

\begin{proof}
This immediate from Theorem \ref{ngo main}, Proposition \ref{prop: integral model proper} plus Corollary \ref{cor: nearby cycles = generic coh}, and the Cebotarev density theorem. 
\end{proof}

\section{Calculation of traces on the cohomology of nearby cycles	}\label{computation of trace}

Our next step is to combine the work of \S \ref{kottwitz for shtukas}, \S \ref{counting fixed points} and \S \ref{hecke base change} to prove Kottwitz-style formulas for both sides of \eqref{eq: trace A and B}. We maintain the notation of those preceding sections.

\subsection{Calculating the trace in situation A}\label{trace A section}

\begin{defn}For a Kottwitz triple $(\gamma_0, (\gamma_x), (\delta_x))$ write 
\[
c(\gamma_0, (\gamma_x), (\delta_x)) := 
\ker^1(F, G_{\gamma_0}) \cdot \mrm{vol}(\Xi \cdot J_{\gamma_0}(F)  \backslash J_{\gamma_0} (\A_F)) \cdot dg(K)^{-1}  
\]
where the notation is as in \S \ref{counting fixed points}.
\end{defn}

\begin{thm}\label{trace situation A}
Let $T' \subset |\cX|$. Assume that $K_v := \Cal{G}(\Cal{O}_v)$ is spherical at all $v \in T'$. Let
\[
\beta = (\beta_v  \in K_v \backslash G(F_v) /  K_v)_{v \in T'}
\]
and $h_{\beta}  \in \Cal{H}$ be the corresponding Hecke operator. Let $R\Psi_{x_0^r}(\Cal{A}_r^{\mu})$ be as in Definition \ref{def: A and B sheavs}, let $\tau$ be as in Theorem \ref{ngo main}, let $f_{\beta_v}$ be as in Definition \ref{defn: hecke function double coset}, and let $\psi_{r,\mu'}$ be as in \S \ref{sssec: BC center parahoric setup}. Then we have 
\begin{align*}
&\Tr(( h_{\beta} \otimes 1 \ldots \otimes 1 ) \circ \Frob_{x_0} \circ  \tau ,  R\Psi_{x_0^r}(\Cal{A}_r^{\mu}))  \\
& \hspace{1cm} =   \sum_{\substack{ (\gamma_0, (\gamma_x),(\delta_x)) \\ \inv(\gamma_0, (\gamma_x),(\delta_x) ) = 0  }} 
c(\gamma_0, (\gamma_x), (\delta_x))   \cdot  \left(\prod_{v \neq x_0}\mrm{O}_{\gamma_v}(f_{\beta_v})  \right) \mrm{TO}_{\delta_{x_0} \sigma}(\psi_{r,\mu}')
\end{align*}
\end{thm}

\begin{proof}
We'll use the Lefschetz trace formula to rewrite the trace in terms of a sum of traces over fixed points. The effect of $\sigma \tau$ on a point of $\Sht_A$ is illustrated below: 
\[
\left\{ \begin{array}{@{}c@{}}
{}^{\sigma}\Cal{E}_1|_{X-x_0} \xrightarrow{\leq \mu} \Cal{E}_1|_{X-x_0} \\
{}^{\sigma}\Cal{E}_2|_{X-x_0}\xrightarrow{\leq \mu} \Cal{E}_2|_{X-x_0} \\
 \vdots \\
{}^{\sigma}\Cal{E}_r|_{X-x_0} \xrightarrow{\leq \mu} \Cal{E}_r|_{X-x_0}
\end{array} \right\} \xrightarrow{\sigma \tau} 
\left\{ \begin{array}{@{}c@{}}
{}^{\sigma^2}\Cal{E}_r|_{X-x_0} \xrightarrow{\leq \mu} {}^{\sigma}\Cal{E}_r|_{X-x_0} \\
{}^{\sigma^2}\Cal{E}_1|_{X-x_0} \xrightarrow{\leq \mu} {}^{\sigma}\Cal{E}_1 |_{X-x_0}\\
 \vdots \\
{}^{\sigma^2}\Cal{E}_{r-1}|_{X-x_0} \xrightarrow{\leq \mu} {}^{\sigma}\Cal{E}_{r-1}|_{X-x_0}.
\end{array} \right\}.
\]
Therefore, a fixed point of the correspondence $( h_{\beta} \otimes 1 \ldots \otimes 1 ) \circ \Frob_{x_0} \circ  \tau$ corresponds to a point as above such that
\begin{align*}
\Cal{E}_2 & = {}^{\sigma} \Cal{E}_1 \\
& \vdots  \\
\Cal{E}_r & = {}^{\sigma} \Cal{E}_{r-1} \\
\Cal{E}_{1} &\xrightarrow{=\beta} {}^{\sigma} \Cal{E}_r.
\end{align*}
By substitution this can be rewritten in terms of $\Cal{E}_1$, and we find that a fixed point is equivalent to the data of commuting modifications
\begin{align*}
t \co {}^{\sigma} \Cal{E}_1|_{X-x_0} \xrightarrow{\leq \mu} \Cal{E}_1|_{X-x_0} \\
t' \co {}^{\sigma^r} \Cal{E}_1|_{X-T'} \xrightarrow{=\beta} \Cal{E}_1|_{X-T'}.
\end{align*}
Hence in the notation of \S \ref{counting fixed points} we see that 
\[
\Fix(( h_{\beta}\otimes 1 \ldots \otimes 1 ) \circ \Frob_{x_0} \circ  \tau)   = \bigcup_{\nu \leq \mu} \Cal{C}(\nu_{x_0}, \beta; x_0,T'; r).
\]

By Lemma \ref{nearby cycles proper} plus Proposition \ref{prop: integral model proper}, we have 
\[
R\Psi_{x_0^r}(\Cal{A}_r^{\mu}) \cong  R\pi_{A!} (R\Psi_{x_0^r}(\mrm{Sat}_{\Sht_A}(r.\mu))).
\]
Now invoking the Grothendieck-Lefschetz trace formula, we have 
\begin{align*}
&\Tr(( h_{\beta} \otimes 1 \ldots \otimes 1 ) \circ \Frob_{x_0} \circ  \tau , R\Psi_{x_0^r}(\Cal{A}_r^{\mu})) \\
&\hspace{1cm} = \sum_{\nu \leq \mu} \sum_{\xi \in \Cal{C}(\nu_{x_0}, \beta; x_0,T'; r)}  \Tr(( h_{\beta} \otimes 1 \ldots \otimes 1 ) \circ \Frob \circ  \tau , R\Psi_{x_0^r}(\mrm{Sat}_{\Sht_A}(r.\mu))_{\xi}).
\end{align*}
By Corollary \ref{cor: kottwitz for shtukas}, for all $\xi \in  \Cal{C}(\nu_{x_0}, \beta; x_0,T'; r)$ we have 
\[	
R\Psi_{x_0^r}(\mrm{Sat}_{\Sht_A}(r.\mu))_{\xi}  = R\Psi_{x_0^r} (\mrm{Sat}_{\Gr_{\Cal{G}}^r}(r.\mu))_{\nu}.
\]
Now using Corollary \ref{cor: number fixed point formula}, we can rewrite our formula as
\begin{align}\label{eq: A long eq 1}
&\Tr(( h_{\beta} \otimes 1 \ldots \otimes 1 ) \circ \Frob_{x_0} \circ  \tau , R\Psi_{x_0^r}(\Cal{A}_r^{\mu}))  \nonumber \\ 
& \hspace{1cm} =  \sum_{\substack{ (\gamma_0, (\gamma_x),(\delta_x)) \\ \inv(\gamma_0, (\gamma_x),(\delta_x) ) = 0  }} 
c(\gamma_0, (\gamma_x), (\delta_x))   \cdot \prod_{v \neq x_0} \mrm{O}_{\gamma_v}(f_{\beta_v})  \nonumber \\
&\hspace{2cm} \cdot 	\sum_{\nu \leq \mu}  \mrm{TO}_{\delta_{x_0} \sigma}(f_{\nu})  \cdot \Tr(( h_{\beta} \otimes 1 \ldots \otimes 1 ) \circ \Frob \circ  \tau , R\Psi_{x_0^r} (\mrm{Sat}_{\Gr_{\Cal{G}}^r}(r.\mu))_{\nu}).
\end{align}
Since the Hecke operator $h_{\beta}$ supports a modification at $T'$, which is disjoint from $x_0$, it acts trivially on all the stalks lying over $x_0^r$, so we may ignore it when computing the trace in \eqref{eq: A long eq 1}. Since
\begin{equation}\label{eq: kunneth}
R\Psi_{x_0^r} (\mrm{Sat}_{\Gr_{\Cal{G}}^r}(r.\mu)) \cong R\Psi_{x_0}(\mrm{Sat}_{\Gr_{\Cal{G}}}(\mu))^{\boxtimes r},
\end{equation}
the trace of $\Frob \circ \tau$ coincides with the trace of Frobenius for the Satake sheaf on the Weil restriction $\Res_{k_r/k}( \Gr_{\Cal{G}}^{\leq \mu} \otimes_{\F_q} \F_{q^r})$. Therefore, by \cite[\S 5.2  Proposition 3]{Ngo06}, we have\footnote{The proof of this formula, which is not explicitly written in \cite{Ngo06}, goes as follows. By \eqref{eq: kunneth} we have 
\begin{align*}
\Tr( \Frob  \circ  \tau , R\Psi_{x_0^r} (\mrm{Sat}_{\Gr_{\Cal{G}}^r}(r.\mu))_{\nu}) & = \Tr(\Frob \circ \tau, R\Psi_{x_0}(\mrm{Sat}_{\Gr_{\Cal{G}}}(\mu))^{\boxtimes r}_{\nu}) \\
&= \Tr(\Frob^r  ,R\Psi_{x_0}(\mrm{Sat}_{\Gr_{\Cal{G}}}(\mu))_{\nu}),
\end{align*}
where in the last equality we used Lemma \ref{Saito-Shintani}.}
\begin{equation}\label{eq: A weil restriction}
\Tr(\Frob \circ  \tau , R\Psi_{x_0^r} (\mrm{Sat}_{\Gr_{\Cal{G}}^r}(r.\mu))_{\nu}) = \Tr(\Frob^r, R\Psi_{x_0^r} (\Sat_{ \Gr_{\Cal{G}} }(\mu))_{\nu}).
\end{equation}

Substituting \eqref{eq: A weil restriction} into \eqref{eq: A long eq 1} we arrive at
\begin{align}\label{eq: A long eq 2}
&\Tr(( h_{\beta} \otimes 1 \ldots \otimes 1 ) \circ \Frob_{x_0} \circ  \tau , R\Psi_{x_0^r}(\Cal{A}_r^{\mu}))  \nonumber \\ 
& \hspace{1cm} =  \sum_{\substack{ (\gamma_0, (\gamma_x),(\delta_x)) \\ \inv(\gamma_0, (\gamma_x),(\delta_x) ) = 0  }} 
c(\gamma_0, (\gamma_x), (\delta_x))   \cdot \prod_{v \neq x_0} \mrm{O}_{\gamma_v}(f_{\beta_v})  \nonumber \\
&\hspace{3cm} \cdot 	\sum_{\nu \leq \mu}  \mrm{TO}_{\delta_{x_0} \sigma}(f_{\nu})  \cdot \Tr(\Frob^r, R\Psi_{x_0^r} (\Sat_{ \Gr_{\Cal{G}} }(\mu))_{\nu}).
\end{align}
By Lemma \ref{lem: parahoric hecke formula} we have the following identity:
\[
\psi_{r,\mu}' = \sum_{\nu \leq \mu} \Tr(\Frob^r, R\Psi_{x_0^r}(\Sat_{ \Gr_{\Cal{G}} }(\mu))_{\nu})  f_{\nu} .
\]
Substituting this in \eqref{eq: A long eq 2}, we finally find
\begin{align*}
&\Tr(( h_{\beta} \otimes 1 \ldots \otimes 1 ) \circ \Frob_{x_0} \circ  \tau , R\Psi_{x_0^r}(\Cal{A}_r^{\mu}))  \\
& \hspace{1cm} =   \sum_{\substack{ (\gamma_0, (\gamma_x),(\delta_{x_0})) \\ \inv(\gamma_0, (\gamma_x),(\delta_x) ) = 0  }} 
c(\gamma_0, (\gamma_x), (\delta_x))   \cdot  \left( \prod_{v \neq x_0}\mrm{O}_{\gamma_v}(f_{\beta_v})  \right) \cdot  \mrm{TO}_{\delta_{x_0} \sigma}(\psi_{r,\mu}')
\end{align*}
which is what we wanted to show. 
\end{proof}

\subsection{Calculating the trace in situation B}\label{trace B section}
We now want to prove an  analogous formula for the trace in situation B. The computation in this case is a little more involved. The main reason is that the action of $S_r$ on $\Cal{B}_{x_0^r}$ is difficult to express explicitly, since it is obtained by ``continuation'' from a locus where it is described explicitly. More precisely, it was obtained from the fact that $\Cal{B}|_U$ was a local system, so that we could extend it over $X$. However, this process obfuscates the geometric meaning of this action, and we will need to use the results of \S \ref{hecke base change}, particularly the geometric model of base change studied in \S \ref{sssec: bc global reformulation}, in order to understand it.

\begin{thm}\label{trace situation B}
Let $T' \subset |\cX|$. Assume that $K_v := \Cal{G}(\Cal{O}_v)$ is spherical at all $v \in T'$. Let 
\[
\beta = (\beta_v  \in K_v \backslash G(F_v) /  K_v)_{v \in T'}
\]
and $h_{\beta}  \in \Cal{H}$ be the corresponding Hecke operator. Let $R\Psi_{x_0^r}(\Cal{B}_r^{\mu})$ be as in Definition \ref{def: A and B sheavs}, let $\tau$ be as in Theorem \ref{ngo main}, and let $\psi_{r,\mu'}$ be as in \S \ref{sssec: BC center parahoric setup}. Then we have 
\begin{align*}
&\Tr( h_{\beta} \circ \Frob_{x_0} \circ  \tau , R\Psi_{x_0^r}(\Cal{B}_r^{\mu}))  \\
& \hspace{1cm} =   \sum_{\substack{ (\gamma_0, (\gamma_x),(\delta_x)) \\ \inv(\gamma_0, (\gamma_x),(\delta_x) ) = 0  }} 
c(\gamma_0, (\gamma_x), (\delta_x))   \cdot  \left(\prod_{v \neq x_0} \mrm{O}_{\gamma_v}(f_{\beta_v}) \right) \cdot \mrm{O}_{\gamma_{x_0} }(b(\psi_{r,\mu}')).
\end{align*}
\end{thm}

\begin{proof}
Let $X^{(r)} = \Sym^r(X)$ by the $r$th symmetric power of $X$, and  
\[
\Sht_{B, X^{(r)}}^{\mu} :=  \Sht_{\Cal{G}, X^{(r)}}^{\leq r.\mu}/a^{\Z}.
\]
The idea here is to push down the computation from $X^r$ to $X^{(r)}$, which trivializes the $S_r$ action on the fiber $\Sht_{B, X^{(r)}}$ over a point, transferring the effect of this action completely to the sheaf theory, where it was studied in \S \ref{hecke base change}. Consider the following commutative diagram, in which the front cartesian square is the fiber of the back cartesian square over the diagonal copy of $X$ in $X^{(r)}$. 
\begin{equation}\label{eq: sit B big diagram}
\begin{tikzcd} 
& \Sht_B^{\mu}\ar[r] \ar[d, "m"]	 &  X^r  \ar[d, "\mrm{add}"] \\
&  \Sht_{B, X^{(r)}}^{\mu}	 \ar[r]  & X^{(r)} \\
\Sht_B^{\mu}|_{\Delta} \ar[uur, hook, dashed]  \ar[d, "m"] \ar[r] &   \Delta (X) \ar[uur, hook, dashed] \ar[d, "\mrm{add}"] \\
\Sht_{\Cal{G},X}^{\leq r\mu}/a^{\Z}  \ar[r, "\pi"]  \ar[uur, hook, dashed] & \Delta(X) \ar[uur, hook, dashed] 
\end{tikzcd}
\end{equation}
The map $\mrm{add} \co X^r \rightarrow X^{(r)}$ is totally ramified over the diagonal $\Delta(X) \subset X^r$, so for any \'{e}tale sheaf $\Cal{F}$ on $X^r$ the stalk of $\Cal{F}$ at $x^r \in X^r$ is canonically identified with the stalk of $\mrm{add}_*\Cal{F}$ at $\mrm{add}(x^r) \in X^{(r)}$. Therefore, from the front cartesian square we have
\[
\Tr( h_{\beta} \circ \Frob_{x_0} \circ  \tau , R\Psi_{x_0^r}(\Cal{B}_r^{\mu}))  = \Tr( h_{\beta} \circ \Frob_{x_0} \circ  \tau ,  R\Psi_{x_0^r}(R\pi^{\circ}_! Rm_! \Sat_{\Sht_B^{\mu}}(r.\mu))) 
\]
where $\pi^{\circ}$ is the restriction of $\pi$ to the fiber over $X^{\circ}$. We now proceeding as before, using the Grothendieck-Lefschetz trace formula to rewrite the trace in terms of a sum of traces over fixed points. We begin by describing the fixed points of $h_{\beta} \circ \Frob_{x_0} \circ  \tau $ on $\Sht_{\Cal{G},X}^{\leq r\mu}/a^{\Z} $. 

Now, on  $\Sht_{B, X^{(r)}}^{\mu}	$ the permutation $\tau$ evidently acts trivially. A point of $\Sht_{\Cal{G},X}/a^{\Z}|_{x_0}$ is a modification 
\[
{}^{\sigma} \Cal{E}_1|_{X-x_0} \xrightarrow{\leq r \mu} \Cal{E}_1|_{X-x_0} 
\]
occurring over $x_0$. The map $\sigma  $ takes this to 
\[
{}^{\sigma^2} \Cal{E}_1|_{X-x_0}  \xrightarrow{\leq r \mu} {}^{\sigma} \Cal{E}_1|_{X-x_0} .
\]
Therefore, a fixed point of the correspondence $ h_{\beta} \circ \Frob_{x_0} \circ  \tau$ is equivalent to the data of commuting modifications
\begin{align*}
t \co {}^{\sigma} \Cal{E}_1|_{X-x_0}  \xrightarrow{\leq r \mu} \Cal{E}_1 |_{X-x_0} \\
t' \co {}^{\sigma} \Cal{E}_1 |_{X-T'} \xrightarrow{=\beta} \Cal{E}_1 |_{X-T'} .
\end{align*}

Hence we see that (remember that $x_0$ is assumed to have degree $1$)
\[
\Fix(h_{\beta} \circ \Frob_{x_0} \circ  \tau)   = \bigcup_{\nu \leq r\mu} \Cal{C}(\nu_{x_0}, \beta; x_0,T'; 1)
\]
so by the Grothendieck-Lefschetz trace and arguing as in \S \ref{trace A section} for situation A, 
\begin{align*}
&\Tr( h_{\beta} \circ \Frob_{x_0} \circ  \tau , R\Psi_{x_0^r}(\Cal{B}_r^{\mu}))  = \sum_{\nu \leq r\mu} \# \Cal{C}(\nu_{x_0}, \beta; x_0,T'; 1) \\
&  \hspace{2.5in} \cdot  \Tr( h_{\beta} \circ \Frob \circ  \tau,  R\Psi_{x_0^r}( Rm_!\Sat_{\Sht_B^{\mu}}(r.\mu))_{\nu}).
\end{align*}
Using Corollary \ref{cor: number fixed point formula}, we rewrite this as 
\begin{align}\label{B eq 1}
&
\Tr( h_{\beta} \circ \Frob_{x_0} \circ  \tau , R\Psi_{x_0^r}(\Cal{B}_r^{\mu}))   \nonumber \\
& \hspace{1cm} =  \sum_{\substack{ (\gamma_0, (\gamma_x),(\delta_x)) \nonumber \\ \inv(\gamma_0, (\gamma_x),(\delta_x) ) = 0  }} 
c(\gamma_0, (\gamma_x), (\delta_x))   \cdot \prod_{v \neq x_0} \mrm{O}_{\gamma_v}(f_{\beta_v}) \nonumber \\
& \hspace{3cm} \cdot \sum_{\nu \leq \mu}  \mrm{TO}_{\delta_{x_0}}(f_{\nu})  \cdot \Tr(h_{\beta} \circ \Frob \circ  \tau, R\Psi_{x_0^r}(Rm_! \Sat_{\Sht_B^{\mu}}(r.\mu))_{\nu}) .
\end{align}

As  in the previous calculation for situation A, the Hecke operator acts trivially on the stalk at $x_0$ because is supported on a disjoint set of points. We use the affine Grassmannian as a local model to calculate $ \Tr( \Frob_{x_0} \circ  \tau, R\Psi_{x_0^r}(Rm_! \mrm{Sat}_{\Sht_B}(r.\mu))_{\nu}) $. By Theorem \ref{varshavsky local model}, we have 
\begin{equation}\label{eq: B eq 2}
\Tr( \Frob_{x_0} \circ  \tau, R\Psi_{x_0^r}( Rm_! \Sat_{\Sht_B^{\mu}}(r.\mu))_{\nu}) = \Tr(\Frob_{x_0} \circ \kappa', R \Psi_{x_0^r} (Rm_! \Sat_{\Gr_{\Cal{G},X^r}^{\leq r.\mu}}(r.\mu))_{\nu}),
\end{equation}
where the notation on the right hand side is as in Corollary \ref{base change formula}, if we can show that the permutation $\tau$ on the left side is carried by Theorem \ref{varshavsky local model} to the permutation $\kappa'$ studied in \S \ref{sssec: bc global reformulation}. (The same issue is raised and explained in \cite[\S 5.6 Proposition 3]{Ngo06}.) To prove it, consider the diagram
\[
\begin{tikzcd} 
& \wt{W}^{\leq r.\mu}_{X^r} \ar[dr, twoheadrightarrow, "\text{\'{e}tale}"] \ar[dl]  \ar[d, "m"] \\
\Gr_{\Cal{G},X^r}^{\leq r.\mu} \ar[d, "m"] & \wt{W}^{\leq r\mu}_{X^{(r)}} \ar[dl] \ar[dr, twoheadrightarrow, "\text{\'{e}tale}"]  &  \Sht_B^{\mu} \ar[r] \ar[d,  "m"]	 & X^r  \ar[d, "\mrm{add}"] \\
\Gr_{\Cal{G},X^{(r)}}^{\leq r\mu} & & \Sht_{B, X^{(r)}}^{\mu} \ar[r]  & X^{(r)} 
\end{tikzcd}
\]
Here $\wt{W}^{\leq r.\mu}_{X^r} $ expresses $\Gr_{\Cal{G},X^r}^{\leq r.\mu}$ as a local model for $ \Sht_B^{\leq r.\mu}|_{\Delta}$ and $\wt{W}^{\leq r\mu}_{X^{(r)}}$ expresses $\Gr_{\Cal{G},X^{(r)}}^{\leq r\mu} $ as a local model for $ \Sht_{B, X^{(r)}}^{\mu} $. The existence of such a commutive diagram is immediate from the proof of Theorem \ref{varshavsky local model}. The claim is then that the permutation actions on $R\Psi_{x_0^r}(Rm_! \Sat_{\Gr_{\Cal{G},X^r}}(r.\mu))$ and $R\Psi_{x_0^r}(Rm_! \mrm{Sat}_{\Sht_B}(r.\mu))$, induced by middle extension from $(\cX)^r$ to $(X-Z)^r$, are compatible. This is clear from the diagram and the fact that the identity can be checked on the locus where the points $(x_1, \ldots, x_r)$ are distinct, where it is evidently given by the same geometric permutation action in both cases. 

Now combining Proposition \ref{base change formula}, Corollary \ref{cor: kottwitz for shtukas}, and Lemma \ref{lem: parahoric hecke formula}, we have 
\begin{equation}\label{eq: B eq 3}
b(\psi_{r,\mu}') = \sum_{\nu \leq r\mu} \Tr( \Frob \circ  \kappa', R \Psi_{ x_0^r} (Rm_!\Sat_{\Gr_{\Cal{G}}}(\mu))_{\nu})  f_{\nu} .
\end{equation}
Putting together \eqref{B eq 1}, \eqref{eq: B eq 2}, and \eqref{eq: B eq 3} gives
\begin{align*}
&\Tr(( h_{\beta} \otimes 1 \ldots \otimes 1 ) \circ \Frob_{x_0} \circ  \tau , R\Psi_{x_0^r}(\Cal{B}_r^{\mu}))  \\
& \hspace{1cm} =   \sum_{\substack{ (\gamma_0, (\gamma_x),(\delta_{x_0})) \\ \inv(\gamma_0, (\gamma_x),(\delta_x) ) = 0  }} 
c(\gamma_0, (\gamma_x), (\delta_x))   \cdot  \left( \prod_{v \neq x_0}\mrm{O}_{\gamma_v}(f_{\beta_v})   \right) \mrm{TO}_{\delta_{x_0}}(b(\psi_{r,\mu}')),
\end{align*}
which is what we wanted to show. 
\end{proof}

\subsection{The base change fundamental lemma for parahoric Hecke algebras}\label{FL proof}

We can now deduce some cases of the base change fundamental lemma. 

\begin{cor}\label{twisted and untwisted} Let $T' \subset |\cX|$. Assume that $K_v := \Cal{G}(\Cal{O}_v)$ is spherical at all $v \in T'$. Let 
\[
\beta = (\beta_v  \in K_v \backslash G(F_v) /  K_v)_{v \in T'}
\]
and $h_{\beta}  \in \Cal{H}$ be the corresponding Hecke operator. Let $\psi_{r,\mu'}$ be as in \S \ref{sssec: BC center parahoric setup} and the base change homomorphim $\psi_{r,\mu'} \mapsto b(\psi_{r,\mu'})$ be as in \S \ref{defn base change}. Then we have
\begin{align*}
&\sum_{\substack{ (\gamma_0, (\gamma_x),(\delta_x)) \\ \inv(\gamma_0, (\gamma_x),(\delta_x) ) = 0  }} 
c(\gamma_0, (\gamma_x), (\delta_x))   \cdot  \left(\prod_{v \neq x_0} \mrm{O}_{\gamma_0}(f_{\beta_v}) \right) \cdot   \mrm{TO}_{\delta_{x_0} \sigma}(\psi_{r,\mu}') \\
& \hspace{2cm} = 
 \sum_{\substack{ (\gamma_0, (\gamma_x),(\delta_{x_0})) \\ \inv(\gamma_0, (\gamma_x),(\delta_x) ) = 0  }} 
c(\gamma_0, (\gamma_x), (\delta_x))   \cdot  \left(\prod_{v \neq x_0} \mrm{O}_{\gamma_0}(f_{\beta_v})\right) \cdot  \mrm{O}_{\gamma_0}(b(\psi_{r,\mu}'))
\end{align*}
\end{cor}

\begin{proof}
This follows immediately from substituting Theorem \ref{trace situation A} and Theorem \ref{trace situation B} into Corollary \ref{equality traces}, and the following comment about changing the $\gamma_v$ to $\gamma_0$: since by definition of $X-Z$ we have that $\Cal{G}(F_x) \cong \GL_n(F_x)$ for all $x \in X-Z$, the notion of stable conjugacy coincides with the notion of conjugacy. 
\end{proof}

It seems to be  ``well-known'' how to deduce a fundamental lemma from a statement such as Corollary \ref{twisted and untwisted}.\footnote{It is remarked on p.84 of the Arxiv version 2 of \cite{Ngo06} that this is ``standard'', and a reference is given to \cite{Clo90}.} Nevertheless, let us give a proof for completeness, following \cite[\S 5.7 Th\'{e}or\`{e}me 1]{Ngo06}. First we introduce a piece of notation. 

\begin{defn}
For $\mu  = (\mu_1, \ldots, \mu_n) \in X_*(\GL_n) \cong \Z^n$, we define 
\[
|\mu| := \mu_1 + \ldots + \mu_n.
\]
The stack $\Sht_{\Cal{G}}^{\leq \mu}$ is non-empty if and only if $|\mu|=0$, because a $\Cal{G}$-bundle has the notion of degree on $X$, which is preserved by the Frobenius twist $\sigma$ on $S$. Let 
\[
Z(\Cal{H}_{G(F_{x_0,r}),J})_0 \subset Z(\Cal{H}_{G(F_{x_0}),J})
\] 
be the subspace generated by the $\psi_{r,\mu}'$ with $|\mu| =0$, which is the same as the subspace generated by the $\psi_{r,\mu}$ with $|\mu|=0$. 
\end{defn}

\begin{thm}
Let $F_{t}$ be a local field of characteristic $p$, and $F_{t,r}/F_{t}$ the unramified extension of degree $r$. Let $\delta$ be a $\sigma$-conjugacy class in $\GL_n(F_{t,r})$, with norm $N \delta_{t} = \gamma_{t} \in \GL_n(F_{t})$. Assume $\gamma_{t}$ is regular semisimple and separable. If $\phi \in Z(\Cal{H}_{G(F_{t,r}),J})_0$, then we have
\[
\mrm{TO}_{\delta_{t} \sigma }(\phi) = \mrm{O}_{\gamma_{t}}(b(\phi)).
\]
\end{thm}

\begin{proof}
Let $\F_q$ be the residue field of $F_{t}$. Choose a global curve $X$ over $\F_q$ having a rational point $x_0$, and function field $F$, so that $F_{x_0} \cong F_t \cong \F_q((t)) $. Choose a division algebra $D$ as in \S \ref{subsec: A and B setup}, and define $G$ and $\Cal{G}$ as in \S \ref{subsec: D-shtukas}. We can then apply Corollary \ref{twisted and untwisted}. 

For a fixed function $h \in \Cal{H}_{G,K}(\A)$ the orbital integrals and twisted orbital integrals are locally constant near regular semisimple separable elements. Therefore, by weak approximation we can choose $\wt{\gamma} \in G(F)$ close enough to $\gamma_t$ in the $t$-adic topology so that $\wt{\gamma} = N(\wt{\delta}_{x_0}) \in G(F_t \otimes_{\F_q} \F_{q^r})$ for some $\wt{\delta}_{x_0} \in G(F_t \otimes_{\F_q} \F_{q^r})$, and such that   
\begin{align*}
 \mrm{O}_{\gamma_{x_0}}(b(\psi_{\mu})) &=  \mrm{O}_{\wt{\gamma}}(b(\psi_{\mu})) \\
 \mrm{TO}_{\delta_{x_0} \sigma_{x_0} }(\psi_{\mu}) &=  \mrm{TO}_{\wt{\delta}_{x_0} \sigma_{x_0}}(\psi_{\mu}).
\end{align*}
We can choose an appropriate Hecke operator $h = (h_v) \in \Cal{H}_G(\A)$ so that $\mrm{O}_{\wt{\gamma}}(h_v) \neq 0$ for $v \neq x$. 

Because a fixed choice of $h$ is the identity at all but finitely many places, any Kottwitz triple for which the product of orbital integrals is non-zero forces the $\gamma_v$ to be in $K_v$ at all but finitely many $v$. Then by \cite[Proposition 7.1]{Kott86} there are only finitely many possibilities for the Kottwitz triple, as all $\gamma_v$ outside a fixed finite set must be (rationally) conjugate to $\gamma$. (Technically this discussion is unnecessary here because we are only dealing with $\GL_n$ at this point.) Therefore, since the support of any adelic Hecke operator is compact open in $G(\A)$, while $G(F)$ is discrete, for any fixed $h \in \Cal{H}_{G,K}(\A)$ there are only finitely many non-zero summands in Corollary \ref{twisted and untwisted}. 

For $\GL_n$, different conjugacy classes of $\gamma_0$ are also different stable conjugacy classes. Hence we may choose the Hecke operator at an unramified auxiliary place appropriately to ensure that 
\[
\left(\prod_{v \neq x_0} \mrm{O}_{\gamma_0 }(h_v) \right) \cdot   \mrm{TO}_{\delta_{x_0} \sigma}(\psi_{r,\mu}') \quad \text{ and } \left(\prod_{v \neq x_0} \mrm{O}_{\gamma_0}(h_v)\right) \cdot  \mrm{O}_{\gamma_0}(b(\psi_{r,\mu}'))
\]
vanish except for the chosen $\gamma_{x_0}$. Then we have 
\begin{align*}
&
c(\gamma_0, (\gamma_x), (\delta_x))   \cdot  \left(\prod_{v \neq x_0} \mrm{O}_{\gamma_0}(h_v) \right) \cdot   \mrm{TO}_{\delta_{x_0} \sigma}(\psi_{r,\mu}') \\
& \hspace{2cm} = 
c(\gamma_0, (\gamma_x), (\delta_x))   \cdot  \left(\prod_{v \neq x_0} \mrm{O}_{\gamma_0}(h_v)\right) \cdot  \mrm{O}_{\gamma_0}(b(\psi_{r,\mu}'))
\end{align*}
Since $|\mu|=0$, $\Sht_{\Cal{G}}^{\leq \mu}$ is non-empty so these terms are not $0$. Dividing out by the common (necessarily non-zero) factor $c(\gamma_0, (\gamma_x), (\delta_x))   \cdot  \left(\prod_{v \neq x_0} \mrm{O}_{\gamma_0}(h_v)\right)$ then yields the desired equality for all $\psi_{r, \mu}'$ with $|\mu|=0$. We conclude by observing that these span $ Z(\Cal{H}_{G(F_{x_0,r}),J})_0$. 
\end{proof}

\bibliographystyle{amsalpha}

\bibliography{Bibliography}

\end{document}